\documentclass{memo-l}

\usepackage{amsmath}
\usepackage{amsthm}
\usepackage{amsopn}
\usepackage{amssymb}
\usepackage[all]{xy}
\usepackage{graphicx}

\allowdisplaybreaks[1]

\newcommand{\stackem}[2]{\genfrac{}{}{0pt}{}{#1}{#2}}

\newcommand{\mc}[1]{\mathcal{#1}}
\newcommand{\mb}[1]{\mathbb{#1}}
\newcommand{\mr}[1]{\mathrm{#1}}
\newcommand{\mbf}[1]{\mathbf{#1}}
\newcommand{\mit}[1]{\mathit{#1}}

\newcommand{\abs}[1]{\lvert #1 \rvert}
\newcommand{\norm}[1]{\lVert #1 \rVert}
\newcommand{\bra}[1]{\langle #1 \rangle}
\newcommand{\ul}[1]{\underline{#1}}
\newcommand{\br}[1]{\overline{#1}}

\newcommand{\td}[1]{\widetilde{#1}}
\newcommand{\tdd}[1]{\widetilde{\widetilde{#1}}}
\newcommand{\ZZ}{\mathbb{Z}}
\newcommand{\RR}{\mathbb{R}}

\newcommand{\FF}{\mathbb{F}}

\newcommand{\DD}{\mathbb{D}}

\newcommand{\Sq}{\mathrm{Sq}}
\newcommand{\Tr}{\mathit{Tr}}
\newcommand{\Id}{\mathrm{Id}}

 \newtheorem{thm}[equation]{Theorem}
 \newtheorem{cor}[equation]{Corollary}
 \newtheorem{lem}[equation]{Lemma}
 \newtheorem{prop}[equation]{Proposition}
\newtheorem*{thm*}{Theorem}
\newtheorem*{cor*}{Corollary}
\newtheorem*{lem*}{Lemma}
\newtheorem*{prop*}{Proposition}

\theoremstyle{definition}

 \newtheorem{ex}[equation]{Example}
 
 \newtheorem{rmk}[equation]{Remark}
\newtheorem{ques}[equation]{Question}
\newtheorem{conj}[equation]{Conjecture}
\newtheorem*{defn*}{Definition}
\newtheorem*{ex*}{Example}
\newtheorem*{exs*}{Examples}
\newtheorem*{rmk*}{Remark}
\newtheorem*{claim*}{Claim}

\newtheorem*{Ack}{Acknowledgments}

\numberwithin{section}{chapter}
\numberwithin{equation}{section}
\numberwithin{figure}{section}
\DeclareMathOperator{\Aut}{Aut}
\DeclareMathOperator{\im}{im}

\DeclareMathOperator{\Top}{Top}
\DeclareMathOperator{\Sp}{Sp}

\begin{document}

\frontmatter

\title{The Goodwillie tower and the EHP sequence}
\author{Mark Behrens}
\address{Department of Mathematics, Massachusetts Institute of Technology, 
Cambridge, MA 02139}
\email{mbehrens@math.mit.edu}
\thanks{The author was partially supported by the NSF, a grant from the Sloan
foundation, and an NEC fund}

\date{September 27, 2010}


\begin{abstract}
We study the interaction between the EHP sequence and the Goodwillie tower of the identity evaluated at spheres at the prime $2$.  Both give rise to spectral sequences (the EHP spectral sequence and the Goodwillie spectral sequence, respectively) which compute the unstable homotopy groups of spheres.  We relate the Goodwillie filtration to the $P$ map, and the Goodwillie differentials to the $H$ map.  Furthermore, we study an iterated Atiyah-Hirzebruch spectral sequence approach to the homotopy of the layers of the Goodwillie tower of the identity on spheres.  We show that differentials in these spectral sequences give rise to differentials in the EHP spectral sequence.  We use our theory to re-compute the $2$-primary unstable stems through the Toda range (up to the $19$-stem).  We also study the homological behavior of the interaction between the EHP sequence and the Goodwillie tower of the identity.  This homological analysis involves the introduction of Dyer-Lashof-like operations associated to M.~Ching's operad structure on the derivatives of the identity.  These operations act on the mod $2$ stable homology of the Goodwillie layers of any functor from spaces to spaces.
\end{abstract}

\maketitle

\setcounter{page}{4}

\tableofcontents

\chapter*{Introduction}

All spaces and spectra in this book are implicitly localized at the prime $2$.  
The purpose of this book is to describe the relationship between two machines for computing the $2$-primary unstable homotopy groups of spheres.

The first of these machines is the EHP sequence.  James \cite{James} constructed fiber sequences
$$ \Omega^2 S^{2n+1} \xrightarrow{P} S^{n} \xrightarrow{E} \Omega S^{n+1} \xrightarrow{H} \Omega S^{2n+1}. $$
Applying $\pi_*$ to the resulting filtered space
\begin{equation}\label{eq:EHPfilt}
 \cdots \rightarrow \Omega^n S^n \rightarrow \Omega^{n+1} S^{n+1} \rightarrow \cdots \rightarrow QS^0
\end{equation}
yields the EHP spectral sequence (EHPSS)
$$ E^1_{m,t} = \pi_{t+m+1} S^{2m+1} \Rightarrow \pi_t^s. $$
By truncating the filtration (\ref{eq:EHPfilt}) we also get truncated EHPSS's
$$
E^1_{m,t}(n) = \begin{cases}
\pi_{t+m+1} S^{2m+1}, & m < n, \\
0, & \text{otherwise} 
\end{cases}
\Rightarrow \pi_{t+n}(S^{n}).
$$
The Curtis algorithm (see, for instance, \cite[Sec.~1.5]{Ravenel}) gives an inductive technique for using the various truncated EHPSS's to compute the $2$-primary unstable stems, provided one is able to compute differentials in the EHPSS.  Unfortunately, as the differentials come from the $P$ map, they are closely related to Whitehead products, and therefore are difficult to compute.  

The second computational machine is the Goodwillie tower of the identity.  Let $\mr{Top}_*$ denote the category of pointed topological spaces.  Goodwillie calculus associates to a reduced finitary homotopy functor 
$$ F : \mr{Top}_* \rightarrow \mr{Top}_* $$
(i.e. a functor which preserves weak equivalences and filtered homotopy colimits, and satisfies $F(\ast) \simeq \ast$) 
a tower $P_i(F)$ of $i$-excisive approximations \cite{G1}, \cite{G2}, \cite{G3}.
$$ 
\xymatrix{
& D_3(F) \ar[d]
& D_2(F) \ar[d] 
& D_1(F) \ar@{=}[d]
\\
\cdots \ar[r]
& P_3(F) \ar[r] 
& P_2(F) \ar[r] 
& P_1(F) 
}
$$
The fibers of the tower take the form of infinite loop spaces $D_i(F) = \Omega^\infty \mb{D}_i(F)$, where
$$ \mb{D}_i(F)(X) \simeq (\partial_i(F) \wedge X^{\wedge i})_{h\Sigma_i} $$
for a $\Sigma_i$-spectrum $\partial_i(F)$ (the $i$th derivative of $F$).
Under favorable conditions ($F$ analytic, $X$ sufficiently highly connected), $F(X)$ is the homotopy inverse limit of the tower $P_i(F)(X)$.
The Goodwillie tower of the identity functor
$$ P_i(X) := P_i(\mr{Id})(X) $$
converges for $X$ connected and $\ZZ$-complete.  Applying $\pi_*$ to the tower $P_i(X)$ gives the Goodwillie spectral sequence (GSS)
$$ E_1^{i,t}(X) = \pi_t \mb{D}_i(X) \Rightarrow \pi_t(X). $$
For $X = S^n$, a ($2$-local) sphere,  Arone and Mahowald showed that $\mb{D}_{i}(S^n) \simeq \ast$ unless $i = 2^k$ \cite{AroneMahowald}.  In this case the GSS may be re-indexed to take the form
$$ E_1^{k,t}(S^n) = \pi_t \mb{D}_{2^k}(S^n) \Rightarrow \pi_t(S^n). $$
Arone and Dwyer \cite{AroneDwyer} identified the layers as
\begin{equation}\label{eq:AroneDwyer}
 \mb{D}_{2^k}(S^n) \simeq \Sigma^{k-n} L(k)_n
\end{equation} 
where $L(k)_n$ is the stable summand of the Thom complex $(B\FF_2^k)^{n \bar{\rho}}$ associated to the Steinberg idempotent, studied by Kuhn, Mitchell, Priddy, and Takayasu \cite{KMP}, \cite{Takayasu}.  Thus the GSS computes the unstable homotopy groups of spheres from the stable homotopy groups of some well-understood spectra.  The differentials in the GSS are mysterious.

In this book we shall demonstrate that when the EHPSS and GSS are computed in tandem, well understood aspects of each sheds light on the more mysterious aspects of the other.  
\begin{enumerate}
\item
\ul{Goodwillie filtration.} Elements in Goodwillie filtration $2^k$
decompose as $k$ iterates of desuspensions of images of the $P$
homomorphism (Whitehead product) in the EHP sequence.

\item
\ul{EHP differentials.}
Goodwillie
filtration preserving differentials in the EHP sequence may be computed
from the differentials in Atiyah-Hirzebruch type spectral sequences for 
$L(k)_n$.

\item
\ul{Goodwillie differentials.}
The differentials in the Goodwillie spectral sequence are given by the $H$
homomorphism (Hopf invariant) in the EHP sequence.
\end{enumerate}

To demonstrate the effectiveness of the approach advocated in this book, we apply these observations to reproduce the computation of the EHPSS, and compute the GSS's through the Toda range (the first $20$ unstable stems).

The the homotopical analysis of the interaction between the Goodwillie tower and the EHP sequence that is undertaken in this book is predicated by a homological analysis of this interaction.  Key to this homological analysis is an understanding of the Arone-Mahowald computation of $H_*(\DD_{2^k}(S^n))$ as a module over the Steenrod algebra.  In order to understand the effect of the natural transformations in the EHP sequence on this homology, we must re-interpret their computation in terms of natural Dyer-Lashof-like operations which act on the stable (mod $2$) homology of the layers of any functor $F$ from spaces to spaces. Using M.~Ching's operad structure on $\partial_*(\Id)$ \cite{Ching}, we construct natural homology operations
$$ \bar{Q}^j: H_*(\DD_{i}(F)(X)) \rightarrow H_{*+j-1}(\DD_{2i}(F)(X)). $$
The construction of these operations should be of independent interest.

Most of the computations of this book are performed with sequences of spectral sequences.  Iterated spectral sequences can complicate calculations, as the accumulated indeterminacy of the associated graded represented by the $E_\infty$-terms can exacerbate the computation of differentials in any  successive spectral sequence.  These difficulties are circumvented through the use of transfinite spectral sequences, as defined by P.~Hu \cite{Hu}.

We outline the contents of this book:  

In Chapter~\ref{sec:Ching}, we briefly review Ching's operad structure, as well as some well known aspects of the homology of symmetric groups.  We then construct the operations $\bar{Q}^i$, and re-interpret the Arone-Mahowald stable homology calculations in terms of these operations.

In Chapter~\ref{sec:EHP}, we show that the EHP sequence induces fiber sequences which relate the layers of the Goodwillie tower of the identity evaluated at spheres.  Using the naturality of the operations $\bar{Q}^i$, we compute the effect of these fiber sequences on stable homology.  The EHP fibrations give rise to fiber sequences between the $L(k)_n$-spectra, which were previously constructed by S.~Takayasu \cite{Takayasu} and N.~Kuhn \cite{Kuhnloop}.  These fiber sequences give an iterated Atiyah-Hirzebruch spectral sequence approach to the homotopy groups of $L(k)_n$, and allow these homotopy groups to be computed inductively in $k$.  These iterated spectral sequences are formulated as \emph{transfinite Atiyah-Hirzebruch spectral sequences} (TAHSS's).  The TAHSS's compute the $E_1$-term of the GSS: taken in succession these yield the \emph{transfinite Goodwillie spectral sequence} (TGSS).

Chapter~\ref{sec:Gfilt} analyzes the meaning of the filtration on the unstable stems arising from the Goodwillie tower.  Another filtration (by \emph{degree of instability}) is defined on the unstable stems.  This filtration allows one to trace through the EHP sequence and see how an unstable element is recursively built up: the unstable element is eventually traced forward to a stable element, and the record of this ancestry is defined to be the \emph{lineage} of an unstable element.  The $P$ and $E$ maps are observed to induce maps on the TAHSS and TGSS.  Under favorable circumstances, the lineage of an element can be read off of the name of a detecting element in the $E^1$-term of the TGSS.

 In Chapter~\ref{sec:Ghopf}, we prove theorems which compute differentials in the TGSS.   Often, these differentials are given by Hopf invariants.  The Goodwillie $d_1$-differentials are related to stable Hopf invariants.  In the presence of an unstable Hopf invariant, we introduce the notion of a \emph{generalized Hopf invariant} (GHI). 
The generalized Hopf invariant is defined by means of a transfinite spectral sequence which concatenates the TGSS with the EHPSS: the resulting refinement of the EHPSS is called the \emph{transfinite EHP spectral sequence} (TEHPSS).  Generalized Hopf invariants give rise to longer differentials in the GSS.  TGSS differentials not arising from Hopf invariants are considered exotic.  Two forms of exotic differentials are identified: those arising from the \emph{geometric boundary effect}, and the \emph{bizarre differentials}.  The geometric boundary effect differentials arise from Hopf invariants through a geometric boundary type theorem derived from the interaction of the TGSS with the EHP sequence.  The bizarre differentials defy explanation, but must exist for the TGSS of $S^1$ to converge to the known values of $\pi_*S^1$.  The Goodwillie-Whitehead conjecture, as proposed by Arone, Kuhn, Lesh, Mahowald and others, predicts that the GSS for $S^1$ collapses at $E_2$.  The relationship of this conjecture to our $d_1$'s is briefly discussed.

In Chapter~\ref{sec:EHPdiff}, we explain how to lift differentials from the TAHSS and the TGSS to the EHPSS (or more precisely, the TEHPSS).  We explain how TEHPSS names are closely related to EHPSS names via the notion of lineage.  In the Toda range, all but one differential is obtained from lifting corresponding differentials from the TAHSS or TGSS.  This one exception gives rise to an example of a deviation between degree of instability and Goodwillie filtration, and a deviation between TGSS names and lineage.

Chapter~\ref{sec:calculations} gives a complete computation of the TAHSS's, TGSS's and TEHPSS in the Toda range (through the $19$-stem).  These computations demonstrate that the tools of this book can be used to give an independent treatment of the unstable $2$-primary computations of \cite{Toda}.  However, unlike \cite{Toda}, we do not solve the additive and multiplicative extension problems presented by these spectral sequences.

Appendix~\ref{sec:Hu} reviews Hu's notion of a transfinite spectral sequence, and describes the transfinite spectral sequence associated to a transfinite tower.  We only work under a strong locally finite hypothesis which guarantees excellent convergence properties of the transfinite spectral sequences --- this local finiteness hypothesis is satisfied by all of the examples of transfinite spectral sequences in this book.  We then prove a general geometric boundary type theorem, which explains the behavior of the transfinite spectral sequences associated to a fiber sequence of transfinite towers.  This geometric boundary theorem specializes to give the geometric boundary effect TGSS differentials, as well as the lifting of TAHSS and TGSS differentials to the TEHPSS.

\begin{Ack}
This author benefited greatly from conversations with Greg Arone, Michael Ching, Nick Kuhn, Katherine Lesh, Mark Mahowald, and Haynes Miller.  
\end{Ack}

\aufm{Mark Behrens}
\vfill

\pagebreak

\section{Conventions}\label{conventions}
As stated at the beginning of the introduction, all spaces and spectra are implicitly localized at the prime $2$.  Homology and cohomology is implicitly taken with $\FF_2$-coefficients.  Mod $2$ binomial coefficients $\binom{a}{b} \in \FF_2$ are defined for all $a, b \in \ZZ$ by
$$ \binom{a}{b} = \text{coefficient of $t^b$ in $(1+t)^a$.} $$
We will denote the connectivity of a space $Y$ by $\mr{conn}(Y)$.  For a spectrum $E$, we denote its Spanier-Whitehead dual by $E^\vee$.  We shall use $S^n$ to denote the $n$-sphere as a space, and use $\ul{S}^n$ to denote the $n$-sphere as a suspension spectrum.  We will use $S$ to denote the sphere spectrum $\ul{S}^0$.

Throughout this book we will be dealing with \emph{completely unadmissible sequences} of integers (which we shall sometimes refer to briefly as ``CU sequences'').  A CU sequence is a sequence of integers
$$ (j_1, \ldots, j_k) $$
with $j_s \ge 2j_{s+1}+1$.  We shall often refer to a CU sequence by a single capital letter (e.g. $J = (j_1, \ldots, j_k)$).  We associate to CU sequences the following quantities:
\begin{itemize}
\item Length: $\abs{J} := k$,
\item Degree: $\norm{J} := j_1+\cdots+j_k$,
\item Excess: $e(J) := j_k$.
\end{itemize}
The empty sequence $\emptyset$ is regarded as a sequence of length zero, degree zero, and excess $\infty$.
We introduce a total ordering on CU sequences of all lengths: we declare that 
$$ (j_1, \ldots, j_k) \le (j'_1, \ldots, j'_{k'}) $$
if either $k > k'$ or, for $k = k'$, if $J \le J'$ with respect to right lexicographical ordering.  For CU sequences
\begin{align*}
J & = (j_1, \ldots, j_k) \\
J' & = (j'_1, \ldots, j'_{k'})
\end{align*}
with $j_k \ge 2j'_1+1$, we shall let $[J, J']$ denote the CU sequence
$$ (j_1, \ldots, j_k, j'_1, \ldots, j'_{k'}). $$

\mainmatter

\chapter{Dyer-Lashof operations and the identity functor}\label{sec:Ching}

Johnson's computation \cite{Johnson} of the derivatives of the identity $\partial_i(\mr{Id})$ admit a description in terms of the poset of partitions of the set $\ul{i}$ \cite{AroneK}, \cite{AroneMahowald}.
Ching observed that this description yields an equivalence
$$ \partial_i(\mr{Id}) \simeq B(1,\mr{Comm}, 1)^\vee_i $$
where $B(1, \mr{Comm}, 1)_i$ is the $i$th space of the operadic bar construction on the commutative operad in spectra \cite{Ching}.  Using this description, Ching gives $\partial_*(\mr{Id})$ the structure of an operad, and therefore the layers 
$\partial_*(\mr{Id}) \wedge \Sigma^\infty X^{\wedge *}$
form a left module over $\partial_*(\mr{Id})$.  More generally, the work of Arone and Ching \cite{AroneChing} gives $\partial_*(F)$ the structure of a bimodule over $\partial_*(\mr{Id})$ for any reduced finitary homotopy functor $F : \mr{Top}_* \rightarrow \mr{Top}_*$.

In Section~\ref{sec:opbar} we review Ching's topological model for the operadic bar construction.  In Section~\ref{sec:cooperad} we recall Ching's cooperad structure on this bar construction.  Dualizing this gives Ching's operadic structure on $\partial_*(\Id)$, as explained in Section~\ref{sec:IDoperad}.

We then use this model to define Dyer-Lashof-like operations $\bar{Q}^j$ on $H_*(\mb{D}_*(F)(X))$, and relate these to the actual Dyer-Lashof operations which appear in the Arone-Mahowald computation of $H_*(\mb{D}(S^n))$.    
We review some well-known aspects of the homology of extended powers in Section~\ref{sec:extendedpowers}.  In Section~\ref{sec:DLlikeops} we construct the operations $\bar{Q}^j$, and explain how the Arone-Mahowald computation of the stable homology of the layers of the Goodwillie tower of the identity evaluated on spheres is given by these operations.

\section{The operadic bar construction}\label{sec:opbar}

Our symmetric sequences shall be regarded as functors
$$ \mr{Fin} \rightarrow \mr{Top}_* $$
where $\mr{Fin}$ is the category of finite sets and bijections.
Let $\Phi$ be a reduced operad in $\mr{Top}_*$.  Ching gives a topological model for the realization of the operadic bar construction
$$ B(\Phi)  := B(1, \Phi, 1). $$
For a finite set $A$, a point in $B(\Phi)(A)$ is given by a tuple
$$ (T, \alpha, (x_v), l) $$
consisting of:
\begin{enumerate}
\item A rooted tree $T$ with one root and $\abs{A}$ leaves.  Each $v$ in $V(T)$, the set of internal vertices, has a set $I(v)$ of incoming edges, with $\abs{I(v)} \ge 2$, and a single outgoing edge.  A vertex $v$ is the source $s(e)$ of its outgoing edge $e$, and is the target $t(e')$ of each of its incoming edges $e'$.
\item A labeling of the leaves, given by a bijection $\alpha$ from the set of leaves to the set $A$.
\item A labeling of the vertices, which assigns to each vertex $v \in V(T)$ a point $x_v$ in $\Phi(I(v))$.
\item A metric on the edges, which assigns to each edge $e$ a non-negative length $l(e)$, such that the distance from each leaf to the root is $1$. 
\end{enumerate}
These tuples are subject to the following identifications:
\begin{itemize}
\item If $l(e) = 0$ for any external edge $e$, then 
$$(T, \alpha, (x_v), l) \sim \ast. $$
\item If $x_v = \ast$ for some $v$, then 
$$ (T, \alpha, (x_v), l) \sim \ast. $$
\item If $l(e) = 0$ for an internal edge $e$, then
$$ (T, \alpha, (x_v) , l) \sim (T/e, \alpha', (x'_v), l') $$
where $T/e$ is the tree with edge $e$ collapsed, and vertices $s(e)$ and $t(e)$ identified.  The leaf labeling $\alpha'$ is induced by $\alpha$ under the natural bijection between the leaves of $T$ and $T/e$.  The vertex labeling $(x'_v)$ is given by
$$
x'_v = 
\begin{cases}
x_{t(e)} \circ_{e} x_{s(e)}, & v \: \text{is the image of $s(e) \sim t(e)$}, \\
x_v & \text{otherwise}.
\end{cases}
$$
The metric $l'$ assigns to the edges $e'$ in $T/e$ their lengths in $T$.
\end{itemize}

As the operadic bar construction arises as the geometric realization of a simplicial symmetric sequence in $\Top_*$
$$ B(\Phi) = \abs{B_\bullet (\Phi)}, $$
the spaces $B(\Phi)(i)$ admit decompositions as simplicial spaces, as is described in \cite[Sec.~4.2]{Ching}.  Let   
$\mc{P}_s(i)$ be a sequence of refining partitions
$$ \pmb{\lambda} = (\lambda_0 \le \lambda_1  \le \cdots \le \lambda_{s+1}) $$
of $\ul{i}$ with 
\begin{align*}
\lambda_{0} & = \{1, 2, \cdots, i \}, \\
\lambda_{s+1} & = \{1\}\{2\} \cdots \{i\}. 
\end{align*}
Associated to $\pmb{\lambda}$ is a rooted tree $(T_{\pmb{\lambda}}, \alpha_{\pmb{\lambda}})$ with $k+2$ levels and leaves labeled by $\ul{i}$. (See Figure~7 of \cite[Sec.~4.2]{Ching}, but our partition indexing is in the opposite order of Ching's.)  
Given a labeling $(x_v)$ of the internal vertices $T_{\pmb{\lambda}}$ by points $x_v \in \Phi(\abs{I(v)})$, we get a map 
$$
\sigma_{\pmb{\lambda}, (x_v)}: \Delta^s \rightarrow B(\Phi)(i)
$$
which assigns to a point
$$ \mbf{t} = (t_0, t_1, \ldots, t_s) \in \Delta^s $$
with $t_j \ge 0$ and $t_0 + \cdots + t_s = 1$, the point
$$ (T_{\pmb{\lambda}}, \alpha_{\pmb{\lambda}}, (x_v), l_\mbf{t}) \in  B(\Phi)(i) $$
where $l_\mbf{t}(e) = t_j$ for $e$ an edge going from level $j+1$ to level $j$. 

\section{The cooperadic structure on $B(\Phi)$}\label{sec:cooperad}

Ching showed that $B(\Phi)$ is a cooperad in $\Top_*$.  
The cooperadic structure map
$$ \circ_{a} : B(\Phi)(A \cup_a B) \rightarrow B(\Phi)(A) \wedge B(\Phi)(B)
$$
sends a point $(T, \alpha, (x_v), l)$
to the point
$$
\begin{cases}
(T', \alpha', (x'_v), l') \wedge (T'', \alpha'', (x''_v), l'') & \text{if $(T,\alpha)$ is obtained by grafting $(T'', \alpha'')$} \\
& \quad \text{onto the leaf labelled by $a$ of $(T', \alpha')$}, \\
\ast, & \text{otherwise},
\end{cases}
$$
where:
\begin{itemize}
\item The labellings $(x'_v)$ and $(x''_v)$ arise from the bijection $V(T) \cong V(T') \amalg V(T'')$. 
\item We have 
$$
l'(e) = 
\begin{cases}
h(T''),  & s(e) = a, \\
l(e), & \text{otherwise},
\end{cases}
$$
where $h(T'')$ is the height of $T''$, when viewed as a subtree of the metric tree $(T,l)$.
\item The metric $l''$ is given by
$$ l''(e) = l(e)/h(T''). $$
\end{itemize}

\section{Operad structure on $\partial_*(\mr{Id})$}\label{sec:IDoperad}

Let $\mr{Comm}$ be the commutative operad in $\Top_*$, with 
$$ \mr{Comm}(i) = S^0. $$ 
Ching observes that the partition poset description of $\partial_i(\mr{Id})$ gives rise to an equivalence
$$ \partial_i(\mr{Id}) \simeq \Sigma^\infty B(\mr{Comm})(i)^\vee $$ 
of $\Sigma_i$-spectra.  As $B(\mr{Comm})$ is a cooperad in $\Top_*$, Ching deduces that $\partial_*(\mr{Id})$ is an operad in spectra.  Arone and Ching also show that, for any reduced finitary homotopy functor $F: \Top_* \rightarrow \Top_*$, the derivatives $\partial_*(F)$ are a bimodule over $\partial_*(\mr{Id})$ \cite{AroneChing}.  In the case where $F = \mr{Id}$, this bimodule structures is the one given by letting $\partial_*(\mr{Id})$ act on itself on both sides.

Our Dyer-Lashof-like operations will rely on a good understanding of $\partial_2(\mr{Id})$.
There is only one rooted tree with $2$ leaves
$$
\xymatrix@C-1em@R-1em{
*=0{\bullet} \ar@{-}[dr] && *=0{\bullet} \ar@{-}[dl] \\
& *=0{\bullet} \ar@{-}[d] \\
& *=0{\bullet}
}
$$
and the space of weightings on this tree is homeomorphic to an interval $I$.  Taking the identifications into account, we deduce that $B(\mr{Comm})(2)$ is homeomorphic to $I/\partial I = S^1$, and thus
$$ \partial_2(\mr{Id}) = S^{-1} $$
(with trivial $\Sigma_2$-action).

\section{Homology of extended powers}\label{sec:extendedpowers}

Let $Y$ be a spectrum.
We review the structure of the homology of the extended powers $Y^{\wedge 2}_{h\Sigma_2}$.

Lemma~1.3 of \cite{May} and Corollary~I.2.3 of \cite{Hinfty} imply the following.

\begin{lem}
Let $\{ y_i \}_{i \in I}$ be an ordered basis of $H_*(Y)$.  Then
$$ H_*(Y^{\wedge 2}_{h \Sigma_2}) = \FF_2\{ e_k \otimes y_i \otimes y_i\}_{i,k} \oplus \FF_2 \{ e_0 \otimes y_{i_1} \otimes y_{i_2} \}_{i_1 < i_2}. $$ 
\end{lem}

Following standard conventions, for $y \in H_d(Y)$ and $j \ge d$, we define the Dyer-Lashof operation $Q^j$ by
$$ Q^j(y) := e_{j-d} \otimes y \otimes y \in H_{d+j}(Y^{\wedge 2}_{h\Sigma_2}). $$
If a sequence $(j_1, \ldots, j_k)$ is \emph{allowable} (i.e. $j_s \ge j_{s+1}+\cdots + j_k +d$ for all $s$, or equivalently, $(j_1, \ldots, j_k)$ has \emph{excess} $\ge d$) then iterating the extended power construction gives elements
$$ Q^{j_1} \wr \cdots \wr Q^{j_k} y \in H_{d+j_1+\cdots +j_k}(Y^{\wedge 2^k}_{h\Sigma_2^{\wr k}}). $$
Define $\td{\mc{R}}_n(k)$ to be the $\FF_2$-module
$$ \td{\mc{R}}_n(k) := \FF_2\{\td{Q}^J  \: : \: \text{$J = (j_1, \ldots, j_k)$ has excess $\ge n$} \}$$
where
$$ \td{Q}^J := Q^{j_1} \wr \cdots \wr Q^{j_k}. $$
The homology of the iterated extended power $H_*(Y^{\wedge 2^k}_{h\Sigma_2^{\wr k}})$ contains a summand given by
$$ \td{\mc{R}}(k) \, \widehat{\otimes} \, H_*(Y) := \FF_2\{\td{Q}^J y_i \: : \:  i \in I, \: J = (j_1, \ldots, j_k), \: Q^J y_i \: \text{allowable} \} $$
(so $\td{\mc{R}}_n(k) = \td{\mc{R}}(k) \, \widehat{\otimes}\, H_*(S^n)$).
The summand $\td{\mc{R}}_n(k) \, \widehat{\otimes} \, H_*(Y)$ is closed under the dual action of the Steenrod algebra, and the dual Steenrod action is computed by the Nishida relations
$$ \Sq^r_* Q^s = \sum_t \binom{s-r}{r-2t} Q^{s-r+t} \Sq^t_*. $$ 

For $t \ge 0$, the diagonal map $S^t \rightarrow S^{t} \wedge S^{t}$ induces a suspension map
$$ E^t: (\Sigma^{-t} Y )^{\wedge 2}_{h\Sigma_2} \rightarrow \Sigma^{-t} Y^{\wedge 2}_{h\Sigma_2} $$ 
(see, for example, \cite[Sec.~3]{Hinfty}).  The following lemma follows from \cite[Lem.~II.5.6]{Hinfty}.

\begin{lem}\label{lem:E}
The induced map on homology
$$ E^t_*: H_*((\Sigma^{-t} Y)^{\wedge 2}_{h\Sigma_2}) \rightarrow H_*(\Sigma^{-t} Y^{\wedge 2}_{h\Sigma_2}) $$
satisfies 
$$ E^t_* Q^j \sigma^{-t} y =
\begin{cases}
\sigma^{-t} Q^j y, & j \ge \abs{y}, \\
0, & \text{otherwise}.
\end{cases}
$$
\end{lem}

For $n \ge 0$, let $\mc{R}_n$ denote the Dyer-Lashof algebra over $\FF_2$, generated by $Q^j$ for $j \ge n$ and subject to the relations
\begin{align*}
Q^{j_1} \ldots Q^{j_k} & = 0 \qquad \text{if $j_1 < j_2+\cdots j_k +n$}, \\
Q^{r} Q^s & = \sum_{t} \binom{t+s-r}{2t-r} Q^{r+s-t} Q^t \qquad \text{(Adem relations)}.
\end{align*}
A sequence $(j_1, \ldots, j_k)$ is \emph{admissible} if $j_s \le 2j_{s+1}$ for all $s$.  The Adem relations express every monomial in $\mc{R}_n$ of length $k$ into a unique sum of admissible monomials, each of which also has length $k$.  Writing $$ \mc{R}_n = \bigoplus_{k \ge 0} \mc{R}_n(k), $$
where $\mc{R}_n(k)$ is the summand of $\mc{R}_n$ spanned by monomials of length $k$, we have
$$ \mc{R}_n(k) = \FF_2\{Q^{J} \: : \: \text{$J = (j_1, \ldots, j_k)$ admissible, and of excess $\ge n$}\}. $$
The homology of the extended power $Y^{\wedge 2^k}_{h\Sigma_{2^k}}$ contains a summand 
$$ \mc{R}(k) \, \widehat{\otimes} \,  H_*(Y) := \FF_2\{ Q^J y_i \: : \: i \in I, \: \text{$J = (j_1, \ldots, j_k)$ is admissible of excess $\ge \abs{y_i}$} \} $$
(so $\mc{R}(k)_n = \mc{R}(k) \, \widehat{\otimes} \, H_*(S^n)$).
This summand is closed under the dual action of the Steenrod algebra, given by the Nishida relations.  The restriction map
$$ H_*(Y^{\wedge 2^k}_{h\Sigma_2^{\wr k}}) \rightarrow H_*(Y^{\wedge 2^k}_{h\Sigma_{2^k}}) $$
acts on the length $k$ summand by
$$ \td{Q}^J y \mapsto Q^J y. $$

Suppose that $Y = \Sigma^\infty X$ is a suspension spectrum.  The effect of the transfer
$$ \Tr : H_*(Y^{\wedge 4}_{h\Sigma_4}) \rightarrow H_* (Y^{\wedge 2}_{h\Sigma_2 \wr \Sigma_2}) $$
on length $2$ summands is given by 
$$ Q^r Q^s y \mapsto Q^r \wr Q^s y+ \sum_t \left[ \binom{s-r+t}{s-t} + \binom{s-r+t}{2t-r} \right] Q^{r+s-t} \wr Q^{t} y $$
(see Theorem~7.1 and Example 7.5 of \cite{Kuhn}).
The following lemma is an observation of Priddy, which relates the above formula to a computation of Kahn-Priddy \cite{Priddy}.

\begin{lem}\label{lem:Priddy}
Suppose that $s < r \le 2s$.  Then we have 
\begin{align*}
\Tr(Q^r Q^s y) & = \sum_{t} \binom{s-r+t}{s - t} Q^{r+s-t} \wr Q^t y \\
& = Q^r \wr Q^s y + \sum_{\ell = 0}^{r-s-1} \binom{2s-r+1+2\ell}{\ell}Q^{2s+1+\ell} \wr Q^{r-s-1-\ell} y.
\end{align*}
In particular
$$ \Tr(Q^rQ^s y) = Q^r \wr Q^s y + \text{terms $Q^{r'}Q^{s'}y$ with $r' \ge 2s'+1$}. $$
\end{lem}

\begin{proof}
In the formula
$$ Q^r \wr Q^s y+ \sum_t \binom{s-r+t}{s-t} Q^{r+s-t} \wr Q^{t} y + \sum_t \binom{s-r+t}{2t-r} Q^{r+s-t} \wr Q^t y $$
the last sum is the Adem relation in the Dyer-Lashof algebra.  As $(r,s)$ is admissible, this sum yields only the term $Q^r \wr Q^s y$.  This establishes the first equality of the lemma:
$$ \Tr(Q^r Q^s y)  = \sum_{t} \binom{s-r+t}{s - t} Q^{r+s-t} \wr Q^t y $$
For the second equality we observe that if $t \ge r-s$ then 
$$ \binom{s-r+t}{s-t} = \binom{s-r+t}{2t-r}. $$
As this is again the binomial coefficient appearing in the Adem relation, and $(r,s)$ is admissible, we only get a contribution for $t = s$ (note that $s \ge r-s$).  Thus we can write
$$ \sum_{t} \binom{s-r+t}{s - t} Q^{r+s-t} \wr Q^t y = Q^r \wr Q^s y + \sum^{r-s-1}_{t = 0}  \binom{s-r+t}{s - t} Q^{r+s-t} \wr Q^t y. $$
We have (letting $\ell = r-s-1-t$)
\begin{align*}
\sum^{r-s-1}_{t = 0}  \binom{s-r+t}{s - t} Q^{r+s-t} \wr Q^t y 
& = \sum^{r-s-1}_{t = 0}  \binom{r -1-2t}{s - t} Q^{r+s-t} \wr Q^t y \\
& = \sum^{r-s-1}_{t = 0}  \binom{r -1-2t}{r-s-1 -t} Q^{r+s-t} \wr Q^t y \\
& = \sum^{r-s-1}_{t = 0}   \binom{2s-r+1+2\ell}{\ell}Q^{2s+1+\ell} \wr Q^{r-s-1-\ell} y. \\
\end{align*}
In particular, we have
$$ 2s+1+\ell \ge r +1+\ell > r-2-2\ell \ge 2r -2s-2-2\ell. $$
\end{proof}

For $n \ge 1$ let $\bar{\mc{R}}_n$ be the free $\FF_2$-algebra generated by $\bar{Q}^j$ for $j \ge 0$, subject to the relations
\begin{enumerate}
\item $$ \bar{Q}^{j_1} \cdots \bar{Q}^{j_k} = 0 \: \text{if $j_1 < j_2+\cdots +j_k +n$, and} $$ 
\item $$ \bar{Q}^r \bar{Q}^s = \sum_{\ell = 0}^{r-s-1} \binom{2s-r+1+2\ell}{\ell}\bar{Q}^{2s+1+\ell} \bar{Q}^{r-s-1-\ell} \: \text{ if $s < r \le 2s$.} $$
\end{enumerate}
For a sequence $J = (j_1, \ldots, j_k)$ we define
$$ \bar{Q}^J := \bar{Q}^{j_1} \cdots \bar{Q}^{j_k}. $$
We shall say that a monomial $\bar{Q}^J$ is CU if the sequence $J$ is a CU sequence (see \ref{conventions}).
The relations amongst the $\bar{Q}^j$ express every monomial in $\bar{\mc{R}}_n$ of length $k$ into a unique sum of CU monomials, each of which also has length $k$.  We give the algebra $\bar{\mc{R}}_n$ an internal grading by setting
$$ \abs{\bar{Q}^j} = j-1. $$
(The reason we have chosen this to be $j-1$ instead of $j$ will be made clear in Section~\ref{sec:DLlikeops}.)
 
Writing 
$$ \bar{\mc{R}}_n = \bigoplus_{k \ge 0} \bar{\mc{R}}_n(k), $$
where $\bar{\mc{R}}_n(k)$ is the summand of $\bar{\mc{R}}_n$ spanned by monomials of length $k$, we have
$$ \bar{\mc{R}}_n(k) = \FF_2\{\bar{Q}^{J} \: : \: \text{$J$ CU, $\abs{J} = k$, $e(J) \ge n$}\}. $$
(Here, $e(J)$ is the \emph{excess} of the CU sequence $J$, as defined in \ref{conventions}.)
The formula for the transfer gives a topological raison d'\^{e}tre for these modules:  
$$ \Sigma^k \bar{\mc{R}}_n(k)\{\iota_n\} = \im \left( \mc{R}_n(k)\{\iota_n\} \hookrightarrow H_*((S^n)^{\wedge 2^k}_{h\Sigma_{2}^{\wr k}}) \rightarrow \frac{H_*((S^n)^{\wedge 2^k}_{h\Sigma_{2}^{\wr k}})}{\sum_i \Tr(H_*((S^n)^{2^k}_{h\td{P}_i}))} \right) $$
where
\begin{equation}\label{eq:Pi} 
\td{P}_i =  \Sigma_2^{\wr (i-1)} \wr \Sigma_4 \wr \Sigma_{2}^{\wr (k-i-1)} \subseteq \Sigma_{2^k} 
\end{equation}
(see the proof of Theorem 3.16 of \cite{AroneMahowald}).  The $\FF_2$-vector space $\bar{\mc{R}}_n(k)$ has an induced action of the dual Steenrod operations, given by combining the Nishida relations with the relations in the algebra $\bar{\mc{R}}_n$.

\begin{rmk}
It is well known that as modules over the opposite of the Steenrod algebra (c.f. \cite[Prop.~3.20]{KuhnWhitehead}, \cite[Thm.~3.17]{AroneDwyer}), there is an isomorphism
$$ \bar{\mc{R}}_0(k) \cong \left( \frac{F_k \mc{A}}{F_{k+1} \mc{A}} \right)^*. $$
Here $\{ F_k \mc{A} \}$ is the length filtration on the Steenrod algebra.  However, despite any misconceptions the reader might infer from the proofs of \cite[Prop.~3.20]{KuhnWhitehead} and \cite[Thm.~3.17]{AroneMahowald}, for $k \ge 2$ we do \emph{not} have a correspondence
$$ \bar{Q}^J \leftrightarrow (\Sq^{J+1})^* $$
(such a correspondence \emph{does} hold for $k=1$). 
Rather, it is the case that 
$$ \bar{Q}^J = (\Sq^{J+1})^* + \text{terms $(\Sq^{J'+1})^*$ with $J' < J$}.  $$
For instance:
$$ \bar{Q}^6 \bar{Q^2} = (\Sq^{7} \Sq^3)^*+(\Sq^8 \Sq^2)^*. $$
\end{rmk}

\section{Dyer-Lashof-like operations}\label{sec:DLlikeops}

Consider the maps $\xi_i$ given by the composites
\begin{align*}
\xi_i: \Sigma^{-1} \mb{D}_{i}(F)(X)^{\wedge 2}_{h\Sigma_2} & \simeq (\partial_2(\mr{Id}) \wedge \partial_i(F)^{\wedge 2} \wedge \Sigma^{\infty} X^{\wedge 2i})_{h\Sigma_2 \wr \Sigma_i} \\
& \rightarrow (\partial_{2i}(F) \wedge X^{\wedge 2i})_{h\Sigma_{2i}} \\
& \simeq \mb{D}_{2i}(F)(X)
\end{align*}
induced by the $\partial_*(\mr{Id})$-module structure on $\partial_*(F)$.  Define, for $j \ge d$, operations
$$ \bar{Q}^j: H_d(\mb{D}_i(F)(X)) \rightarrow H_{d+j-1}(\mb{D}_{2i}(F)(X)) $$
as follows.
For $x \in H_d(\mb{D}_{i}(F)(X))$, define
$$ \bar{Q}^j x := (\xi_i)_* \sigma^{-1} Q^j x. $$

Our main observation is the following reinterpretation of the Arone-Mahowald computation of $H_*(\mb{D}_{2^k}(S^n))$.

\begin{thm}\label{thm:AroneMahowald}
As an module over the operators $\bar{Q}^j$, we have
$$ \bigoplus_{k \ge 0} H_*(\mb{D}_{2^k}(S^n)) = \bar{\mc{R}}_n \{ \iota_n \}. $$
where
$$ \iota_n \in H_n(\mb{D}_1(S^n)) \cong \td{H}_n(S^n) $$
is the fundamental class.
In particular, we have
$$ H_*(\mb{D}_{2^k}(S^n)) =  \FF_2 \{ \bar{Q}^J \iota_n \: : \: J \: \text{is CU}, \: \abs{J} = k, \: e(J) \ge n \}  = \bar{\mc{R}}_n(k) \{\iota_n\} $$
and the operators $\bar{Q}^j$ act according to the relations in the algebra $\bar{\mc{R}}_n$.
If $J$ is a CU sequence of length $k$ with $e(J) \ge n$, then under the isomorphism of \cite[Thm.~3.16]{AroneMahowald}, we have
$$ Q^J \iota_n \mapsto \sigma^k \bar{Q}^J \iota_n. $$ 
\end{thm}

\begin{figure}
$$
\xymatrix@C-2em@R-1em{
*=0{\bullet} \ar@{-}[dr] && 
*=0{\bullet} \ar@{-}[dl] && 
*=0{\bullet} \ar@{-}[dr]  && 
*=0{\bullet}  \ar@{-}[dl] && 
*=0{\bullet} \ar@{-}[dr]  && 
*=0{\bullet}  \ar@{-}[dl] && 
*=0{\bullet} \ar@{-}[dr]  && 
*=0{\bullet} \ar@{-}[dl]
\\
& *=0{\bullet} \ar@{-}[drr] &&&&
*=0{\bullet} \ar@{-}[dll] &&&&
*=0{\bullet} \ar@{-}[drr]  &&&&
*=0{\bullet} \ar@{-}[dll] 
\\
&&& *=0{\bullet} \ar@{-}[drrrr] &&&&&&&&
*=0{\bullet} \ar@{-}[dllll] 
\\
&&&&&&& *=0{\bullet} \ar@{-}[d]
\\
&&&&&&& *=0{\bullet}
}
$$
\caption{The tree $T_3$.}\label{fig:tree}
\end{figure}

Let $T_{k}$ denote the unique tree with $k+2$ levels, $2^k$ leaves, and $\abs{I(v)} = 2$ for all $v$ (see Figure~\ref{fig:tree}).  The tree $T_{k}$ has $2^k-1$ internal vertices, its automorphism group is given by
$$ \Aut(T_{k}) \cong \Sigma_2^{\wr k} $$
and the set of isomorphism classes of leaf labellings is in bijective correspondence with  $\Sigma_{2^k}/\Sigma_2^{\wr k}$.
The proof of Theorem~\ref{thm:AroneMahowald} rests on the following technical lemma.

\begin{lem}\label{lem:technicallemma}
Let $\Phi$ be an operad in $\Top_*$, and let
$$ \abs{B_\bullet(\Phi)(i)}^{[s]} \subseteq \abs{B_\bullet(\Phi)(i)} = B(\Phi)(i) $$
denote the $s$th skeletal filtration of the geometric realization of the simplicial space $B_\bullet(\Phi)(i)$.  Then there are  factorizations of the cooperadic structure maps:
$$ 
\xymatrix{
\abs{B_\bullet(\Phi)(2^{k+1})}^{[k+1]} \ar@{.>}[r] \ar@{^{(}->}[d] & 
\abs{B_{\bullet}(\Phi)(2)}^{[1]} \wedge  \abs{B_\bullet(\Phi)(2^k)}^{[k]} \wedge \abs{B_\bullet(\Phi)(2^k)}^{[k]}  \ar@{^{(}->}[d]
\\
B(\Phi)(2^{k+1}) \ar[r] & 
B(\Phi)(2) \wedge  B(\Phi)(2^k) \wedge B(\Phi)(2^k)  
}
$$
Furthermore, the following diagrams of pointed $\Sigma_2 \wr \Sigma_{2^k}$ spaces commute
$$
\xymatrix@C-1em{
B(\Phi)(2^{k+1}) \ar[r] & 
B(\Phi)(2) \wedge B(\Phi)(2^k) \wedge B(\Phi)(2^k)
\\ 
\abs{B_\bullet(\Phi)}^{\lbrack k+1\rbrack} 
 \ar[d]_{p_{T_{{k+1}}}} 
\ar[r] \ar@{^{(}->}[u] &
\abs{B_\bullet(\Phi)(2)}^{[1]} \wedge \abs{B_\bullet(\Phi)(2^k)}^{[k]} \wedge \abs{B_\bullet(\Phi)(2^k)}^{[k]} 
\ar@{^{(}->}[u] \ar[d]^{p_{T_1} \wedge p_{T_{k}} \wedge p_{T_{k}}} 
\\
\Sigma_{2^{k+1}} \underset{\Sigma_2^{\wr {(k+1)}}}{\times} \frac{\Delta^{k+1}}{\partial\Delta^{k+1}} \wedge \Phi(2)^{2^{k+1}-1} \ar[r]_-{\beta_k} &
\Sigma_2 \wr \Sigma_{2^k} \underset{\Sigma_2^{\wr {(k+1)}}}{\times} \frac{\Delta^1 \times \Delta^k \times \Delta^k}{\partial(\Delta^1 \times \Delta^k \times \Delta^k)} \wedge \Phi(2)^{2^{k+1}-1}
}
$$
Here, $p_{T}$ is the projection onto the space of simplices corresponding to the tree $T$, the map $\beta_k$ is given by
$$
\beta_k[\sigma, (t_0, \ldots, t_{k+1}) , (x_v)] = 
\begin{cases}
[\sigma, \alpha_k(t_0, \ldots, t_{k+1}), (x_v)], & \sigma \in \Sigma_2 \wr \Sigma_{2^k} \subset \Sigma_{2^{k+1}}. \\
\ast, & \text{otherwise},
\end{cases}
$$
where
$$ {\alpha}_k : \frac{\Delta^{k+1}}{\partial \Delta^{k+1}} \rightarrow \frac{\Delta^1 \times \Delta^k \times \Delta^k}{\partial (\Delta^1 \times \Delta^k \times \Delta^k)} $$
is the map given by
\begin{multline*}
\alpha_k(t_0, \cdots , t_{k+1}) = 
(t_0, t_1 + \cdots + t_{k+1}) \times \left( \frac{t_1}{t_1 + \cdots + t_{k+1}}, \ldots , \frac{t_{k+1}}{t_1 + \cdots + t_{k+1}} \right) \\
\times \left(\frac{t_1}{t_1 + \cdots + t_{k+1}}, \ldots , \frac{t_{k+1}}{t_1 + \cdots + t_{k+1}} \right),
\end{multline*}
and the group $\Sigma_2^{\wr(k+1)}$ acts trivially on $\frac{\Delta^{k+1}}{\partial \Delta^{k+1}}$ and permutes the two factors of $\Delta^k$ in $\frac{\Delta^1 \times \Delta^k \times \Delta^k}{\partial (\Delta^1 \times \Delta^k \times \Delta^k)}$.
\end{lem}

\begin{proof}
Observe that a point
$$ [T, \alpha, (x_v), l] \in B(\Phi)(i) $$
lies in the subspace $\abs{B_\bullet(\Phi)(i)}^{[k]}$ if the metric tree $(T,l)$ has at most $k+2$ distinct levels.  (By a \emph{level}, we mean a height $h \in [0,1]$ for which there exists at least one (possibly external) vertex of distance $h$ from the root, see Figure~\ref{fig:levels}).

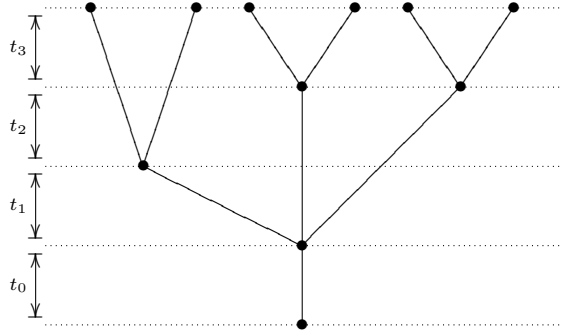
\begin{figure}
$$
\xymatrix@C-1em@R-0em{
\, \ar@{.}[r] \ar@{|<->|}[d]_{t_3} & 
*=0{\bullet} \ar@{-}[ddr]  \ar@{.}[rr] && 
*=0{\bullet} \ar@{-}[ddl] \ar@{.}[r]  & 
*=0{\bullet} \ar@{-}[dr] \ar@{.}[rr] &&
*=0{\bullet} \ar@{-}[dl] \ar@{.}[r] &
*=0{\bullet} \ar@{-}[dr] \ar@{.}[rr] && 
*=0{\bullet}  \ar@{-}[dl] \ar@{.}[r] &
\\
\, \ar@{|<->|}[d]_{t_2} \ar@{.}[rrrrr] &&&&&
*=0{\bullet} \ar@{-}[dd] \ar@{.}[rrr] &&&
*=0{\bullet} \ar@{-}[ddlll] \ar@{.}[rr] &&
\\
\, \ar@{|<->|}[d]_{t_1} \ar@{.}[rr] && 
*=0{\bullet} \ar@{-}[drrr] \ar@{.}[rrrrrrrr] &&&&&&&&
\\
\, \ar@{|<->|}[d]_{t_0} \ar@{.}[rrrrr] &&&&&
*=0{\bullet} \ar@{-}[d] \ar@{.}[rrrrr] &&&&&
\\
\, \ar@{.}[rrrrr] &&&&&
*=0{\bullet} \ar@{.}[rrrrr] &&&&&
}
$$
\caption{A metric tree with $5$ levels}\label{fig:levels}
\end{figure}

Since there is only one tree with two leaves (the tree $T_1$), and this tree has $3$ levels, we have
$$ B(\Phi)(2) = \abs{B_\bullet(\Phi)(2)}^{[1]}. $$
Furthermore, if a metric tree $(T,l)$ is obtained by grafting the metric trees $(T', l')$ and $(T'', l'')$ onto each of the limbs of $(T_1, l''')$, then if $(T', l')$ has $k'+2$ levels, and $(T'', l'')$ has $k''+2$ levels, then 
$$ \# \text{levels of $(T,l)$} \ge \max (k',k'')+3. $$
So if $(T, l)$ has at most $k+3$ levels, then $k' \le k$ and $k'' \le k$.  This proves the desired factorization of the first part of the lemma.

The commutativity of the diagram in the second part of the lemma is a direct consequence of the explicit description of the cooperad structure on $B(\Phi)$ given in Section~\ref{sec:cooperad}, together with the fact that the metric tree $(T_{{k+1}}, l_\mbf{t})$ with $k+3$ levels, with
$$ \mbf{t} = (t_0, \ldots, t_{k+1}) \in \Delta^{k+1}$$
is obtained by grafting two copies of the metric tree $(T_{{k}}, l_{\mbf{t}''})$, with
$$ \mbf{t}'' = \left( \frac{t_1}{t_1+\cdots+t_{k+1}} , \ldots, \frac{t_{k+1}}{t_1+\cdots+t_{k+1}} \right) $$
onto the limbs of $(T_1, l_{\mbf{t}'})$, with
$$ \mbf{t}' = (t_0, t_1+\cdots+t_{k+1}). $$
\end{proof}

\begin{proof}[Proof of Theorem~\ref{thm:AroneMahowald}]
The Arone-Mahowald calculations imply that, applying homology to the zigzag
$$ \mb{D}_{2^{k}}(S^n) \xrightarrow{f_k}  F(\abs{B_\bullet(\Phi)(2^{k})}^{[k]}, \ul{S}^n)_{h\Sigma_{2^{k}}}
\xleftarrow{g_k} \Sigma^{-k} (\ul{S}^n)^{\wedge 2^{k}}_{h \Sigma_2^{\wr k}}
$$
the map $(f_k)_*$ is injective, and $\im (f_k)_* = \im (g_k)_*$.  The summand
$$  \Sigma^{-k} \td{\mc{R}}\{\iota_n\} = \FF_2\{ \sigma^{-k} \td{Q}^J \iota_n \: : \: J = (j_1, \ldots, j_k), \: \text{$Q^J \iota_n$ allowable} \} \subseteq H_*(\Sigma^{-k} (\ul{S}^n)^{\wedge 2^{k}}_{h \Sigma_2^{\wr k}}) $$
surjects onto $\im (f_k)_*$ under the map $(g_k)_*$, with kernel spanned by the elements
$$ \sigma^{-k}\Tr_i(Q^{j_1} \wr \cdots \wr Q^{j_{i-1}} \wr Q^{j_i} Q^{j_{i+1}} \wr Q^{j_{i+2}} \wr \cdots \wr Q^{j_k} ) $$
where $\Tr_i$ denotes the transfer associated to the subgroup $P_i \subseteq \Sigma_{2^k}$ (\ref{eq:Pi}).

We prove the theorem by induction on $k$.  The cases of $k = 0$ is trivial.  Assume the result for $k$: for $J = (j_1, \ldots, j_k)$ a CU sequence with $j_k \ge n$ we assume that
$$ (f_k)_*(\bar{Q}^{J}  \iota_n) = (g_k)_*(\sigma^{-k}Q^J \iota_n). $$
   Observe that the map $\alpha_k$ of Lemma~\ref{lem:technicallemma} is homeomorphic to the suspension of the diagonal
$$ \Sigma \Delta: \Sigma S^k \rightarrow \Sigma S^k \wedge S^k. $$
Therefore the second diagram of 
Lemma~\ref{lem:technicallemma}, when applied to $\Phi = \mr{Comm}$ and dualized, gives rise to a diagram
$$
\xymatrix{
\Sigma^{-1} \mb{D}_{2^k}(S^n)^{\wedge 2}_{h \Sigma_2} \ar[r]^{\xi_{2^k}} \ar[d]_{f_k^{\wedge 2}} &
\mb{D}_{2^{k+1}}(S^n) \ar[d]^{f_{k+1}} 
\\
\Sigma^{-1} F(\abs{B_\bullet(\Phi)(2^k)}^{[k]}, \ul{S}^n)^{\wedge 2}_{h\Sigma_2 \wr \Sigma_{2^k}} \ar[r]^{\td{\xi}_{2^k}} &
F(\abs{B_\bullet(\Phi)(2^{k+1})}^{[k+1]}, \ul{S}^n)_{h\Sigma_{2^{k+1}}} 
\\
\Sigma^{-1} \left( \Sigma^{-k}(\ul{S}^n)^{\wedge 2^k}_{h\Sigma_2^{\wr k}} \right)^{\wedge 2}_{h\Sigma_2} \ar[r]_{E^k} \ar[u]^{g_k^{\wedge 2}} &
\Sigma^{-k-1} (\ul{S}^n)^{\wedge 2^{k+1}}_{h \Sigma_2^{\wr (k+1)}} \ar[u]_{g_{k+1}} 
}
$$
We therefore have
\begin{align*}
(f_{k+1})_*(\bar{Q}^{j} \bar{Q}^J \iota_n) 
& = (f_{k+1})_* (\xi_{2^k})_* (\sigma^{-1} Q^{j} \bar{Q}^J \iota_n) \\
& = (\td{\xi}_{2^k})_*  (\sigma^{-1} Q^{j} (f_{k})_* \bar{Q}^J \iota_n) \\
& = (\td{\xi}_{2^k})_*  (\sigma^{-1} Q^{j} (g_{k})_* \sigma^{-k} Q^J \iota_n) \\
& = (g_{k+1})_* E^k_* (\sigma^{-1} Q^{j}  \sigma^{-k} Q^J \iota_n) \\
& = (g_{k+1})_* (\sigma^{-k-1} Q^{j} Q^J \iota_n).
\end{align*}
\end{proof}

\chapter{The Goodwillie tower of the EHP sequence}\label{sec:EHP}

In this chapter we explain how the Goodwillie tower interacts with the EHP sequence.  In Section~\ref{sec:EHPfiberseq}, we show that the EHP sequence induces fiber sequences on the layers of the Goodwillie tower of the identity when evaluated on spheres.  The homological behavior of these fiber sequences is described in Section~\ref{sec:HEHPfiberseq}.  The transfinite Atiyah-Hirzebruch spectral sequences (TAHSS's) which inductively compute $\pi_* L(k)_n$ are constructed in Section~\ref{sec:TAHSS}.  The transfinite Goodwillie spectral sequence (TGSS) is constructed in Section~\ref{sec:TGSS}.

\section{Fiber sequences associated to the EHP sequence}\label{sec:EHPfiberseq}
The EHP sequence arises from the sequence of functors and natural transformations
\begin{equation}\label{eq:EHP}
\Id \xrightarrow{E} \Omega \Sigma \xrightarrow{H}  \Omega \Sigma \mr{Sq}. 
\end{equation} 
Here $\mr{Sq}: \mr{Top}_* \rightarrow \mr{Top}_*$ is the squaring functor
$$ \mr{Sq}(X) = X \wedge X $$
and $H$ is the adjoint to the projection onto the second summand of the James splitting
$$ \td{H}: \Sigma \Omega \Sigma X \simeq \Sigma \bigvee_i X^{\wedge i} \rightarrow \Sigma X^{\wedge 2}. $$
There exist models for the map $\td{H}$ which are natural in $X$ (see, for instance, \cite{CohenMayTaylor}). 
$2$-locally, the sequence (\ref{eq:EHP}) yields fiber sequences when evaluated on spheres.  

\begin{lem}\label{lem:EHPfiber}
The EHP sequence induces fiber sequences
\begin{gather*}
P_i(\mr{Id})(S^n) \xrightarrow{E} P_i(\Omega \Sigma)(S^n) \xrightarrow{P} P_i(\Omega \Sigma \mr{Sq})(S^n), \\
\mb{D}_i(\mr{Id})(S^n) \xrightarrow{E} \mb{D}_i(\Omega \Sigma)(S^n) \xrightarrow{P} \mb{D}_i(\Omega \Sigma \mr{Sq})(S^n). 
\end{gather*}
\end{lem}

\begin{proof}
Given a topologically enriched functor $F: \rm{Top}_* \rightarrow \rm{Top}_*$, and a pointed space $X$, 
one gets an associated functor $G: \mr{Vect} \rightarrow \mr{Top}_*$ with
$$ G(V) := \Omega^V F(\Sigma^V X). $$
(Here, $\mr{Vect}$ is the category of finite dimensional real inner product spaces and isometries.)
Weiss's orthogonal calculus \cite{Weiss} yields a sequence of polynomial approximations $P^\mr{orth}_i(G)(V)$ with fibers which deloop to give spectra $\mb{D}^\mr{orth}_i(G)(V)$.  Furthermore, we have (see \cite[Lem.~1.2]{Arone}, \cite[Ex.~5.7]{Weiss})
\begin{align*}
P^{\mr{orth}}_i(G)(V) & \simeq \Omega^V P_{i+1} (F)(\Sigma^V X), \\
\mb{D}_i^{\mr{orth}}(G)(V) & \simeq \Omega^V \mb{D}_{i+1} (F)(\Sigma^V X).
\end{align*}
As the functors and natural transformations in the EHP sequence are enriched in topological spaces, 
we have fiber sequences
$$ \Omega^V S^{V+ m} \xrightarrow{E} \Omega^{V} \Omega \Sigma S^{V + m} + \Omega^{V} \Omega \Sigma \mr{Sq} S^{V+m} $$
of functors $\mr{Vect} \rightarrow \mr{Top}_*$.  The lemma follows from applying orthogonal calculus to these fiber sequences.
\end{proof}

\begin{lem}\label{lem:Sq}
Let $F: \mr{Top}_* \rightarrow \mr{Top}_*$ be a reduced finitary homotopy functor which is stably $i$-excisive for all $i$.
Then there are equivalences
\begin{align*}
P_i(F \mr{Sq})(X) & \simeq P_{\lfloor i/2 \rfloor} (F)(X \wedge X), \\
\mb{D}_i(F \mr{Sq})(X) & \simeq 
\begin{cases}
\mb{D}_{i/2} (F)(X \wedge X), & i \: \text{even}, \\
\ast, & i \: \text{odd}.
\end{cases}
\end{align*}
\end{lem}

\begin{proof}
The second equivalence follows from the first.  By \cite{G3}, to prove the first equivalence it suffices to prove (1) that $P_{\lfloor i/2 \rfloor}(F)\mr{Sq}$ is $i$-excisive, and (2) that $F\rm{Sq}$ and $P_{\lfloor i/2 \rfloor}(F)\mr{Sq}$ agree to order $i$.

The Goodwillie tower for $P_{\lfloor i/2 \rfloor}(F)$, when evaluated at $X \wedge X$, has fibers of the form
$$ \Omega^{\infty} (\partial_j(F) \wedge (X \wedge X)^{\wedge j})_{h\Sigma_j} \simeq \Omega^{\infty} ([(\Sigma_{2j})_+ \wedge_{\Sigma_j} \partial_j(F)] \wedge X^{2j})_{h\Sigma_{2j}}  $$
for $j \le \lfloor i/2 \rfloor$.  These layers are all $i$-excisive in $X$.  This establishes (1).

For (2), observe that since the functors $F$ and $P_{\lfloor i/2 \rfloor}(F)$ agree to order $\lfloor i/2 \rfloor$ under the natural transformation
$$ F \rightarrow P_{\lfloor i/2 \rfloor}(F) $$
there exists a $c$ so that for $Y$ sufficiently highly connected, the map
$$ F(Y) \rightarrow P_{\lfloor i/2 \rfloor}(F)(Y) $$
is $(-c+(\lfloor i/2 \rfloor +1)\mr{conn}(Y))$-connected.  Setting $Y = X \wedge X$ for $X$ sufficiently highly connected, we deduce that
$$ F(\mr{Sq}(X)) \rightarrow P_{\lfloor i/2 \rfloor}(F)(\mr{Sq}(X)) $$
is $(-c+2(\lfloor i/2 \rfloor +1)\mr{conn}(X))$-connected.  Therefore it is $(-c+(i+1)\mr{conn}(X))$-connected, and we have established (2).
\end{proof}

\begin{cor}\label{cor:EHPfiber}
The fiber sequences of Lemma~\ref{lem:EHPfiber} are equivalent to the following fiber sequences.
\begin{gather*}
P_{2m}(S^n) \xrightarrow{E} \Omega P_{2m}(S^{n+1}) \xrightarrow{H} \Omega P_{m}(S^{2n+1}) \\
P_{2m+1}(S^n) \xrightarrow{E} \Omega P_{2m+1}(S^{n+1}) \xrightarrow{H} \Omega P_{m}(S^{2n+1}) \\
\DD_{2m}(S^n) \xrightarrow{E} \Sigma^{-1} \DD_{2m}(S^{n+1}) \xrightarrow{H} \Sigma^{-1} \DD_{m}(S^{2n+1}) \\
\DD_{2m+1}(S^n) \xrightarrow{E} \Sigma^{-1} \DD_{2m+1}(S^{n+1}) \xrightarrow{H} \ast
\end{gather*}
\end{cor}

\begin{proof}
The corollary follows immediately from Lemma~\ref{lem:EHPfiber}, Lemma~\ref{lem:Sq}, and the identities
\begin{align*}
P_i(\Omega F)(X) & \simeq \Omega P_i(F)(X), \\
P_i(F\Sigma)(X) & \simeq P_i(F)(\Sigma X), \\
\DD_i(\Omega F)(X) & \simeq \Sigma^{-1} \DD_i(F)(X), \\
\DD_i(F\Sigma)(X) & \simeq \DD_i(F)(\Sigma X).
\end{align*}
\end{proof}

\begin{rmk}
Corollary~\ref{cor:EHPfiber} gives an amusing alternative proof of the $2$-primary case of Theorem~3.13 of \cite{AroneMahowald}: if $i = s2^k$ for $s$ odd, then if $s \ne 1$, 
$$ \DD_n(S^n) \simeq \ast. $$
This can be deduced by induction on $k$.  For $k = 0$, Corollary~\ref{cor:EHPfiber} implies that the suspension
$$ E: \DD_{s}(S^n) \rightarrow \Sigma^{-1} \DD_{s}(S^{n+1}) $$
is an equivalence.  Taking the colimit of these maps, we deduce that there are equivalences
$$ \DD_s(S^n) \xrightarrow{\simeq} \DD_1(\DD_s)(S^n) \simeq \ast $$
(since $s \ne 1$).  Suppose inductively that
$$ \DD_{s2^k}(S^n) \simeq \ast $$
for $s \ne 1$ odd.  Then the inductive hypothesis, together with Corollary~\ref{cor:EHPfiber}, implies that the suspension
$$ E: \DD_{2^{k+1}s}(S^n) \rightarrow \Sigma^{-1} \DD_{2^{k+1}s}(S^{n+1}) $$
is an equivalence, and, as in the base case, this implies that $\DD_{2^{k+1}s}(S^n) \simeq \ast$.
\end{rmk}

Specializing Corollary~\ref{cor:EHPfiber} to the case of $i = 2^k$, and applying the equivalences (\ref{eq:AroneDwyer}), we recover fiber sequences of the same form as those discovered by Kuhn \cite[Prop.~A.7]{Kuhnloop} and Takayasu  \cite{Takayasu}.

\begin{cor}\label{cor:Takayasu}
There are fiber sequences:
$$
\Sigma^n {L(k-1)}_{2n+1} \xrightarrow{P} L(k)_n \xrightarrow{E} L(k)_{n+1}.  
$$
\end{cor}

\begin{rmk}
We are \emph{not} asserting that the fiber sequences of Corollary~\ref{cor:Takayasu} are equivalent to those of \cite{Kuhnloop} and \cite{Takayasu}, but it is likely that they are.  Indeed, Proposition~\ref{prop:HTakayasu} will demonstrate that both fiber sequences have the same homological behavior.  Our reason for using the $L(k)_n$ notation instead of the $\DD_{2^k}(S^n)$ notation is that the indexing in the fiber sequences is more compelling in the $L(k)_n$-notation.
\end{rmk}

\section{Homological behavior of the fiber sequences}\label{sec:HEHPfiberseq}

We wish to understand the homological behavior of the fiber sequences of Corollary~\ref{cor:Takayasu}.
Our first task is to identify the operations $\bar{Q}^j$ on $H_*(\DD_{2^k}(\Omega \Sigma))$ and $H_*(\DD_{2^k}(\Omega \Sigma \Sq))$.  We make some general observations in the next sequence of lemmas.  

We make use of the Arone-Ching chain rule \cite{AroneChing}:  if $F, G: \Top_* \rightarrow \Top_*$ are reduced finitary  homotopy functors, there is an equivalence in the homotopy category of symmetric sequences of spectra
$$ \partial_*(FG) \simeq B(\partial_* (F), \partial_*(\Id), \partial_*(G)). $$
The technology of Arone-Ching (see, for example, the discussion after Proposition~0.4 of \cite{AroneChing}) also gives a simple method to compute the derivatives of any finitary reduced homotopy functor $F: \Top_* \rightarrow \Top_*$ in terms of the dual derivatives $\partial^*(\Sigma^\infty F)$ of the associated functor $\Sigma^\infty F: \Top_* \rightarrow \Sp$: 
\begin{equation}\label{eq:ACformula}
 \partial_*(F) \simeq B(1, \mr{Comm}, \partial^*(\Sigma^\infty F))^\vee. 
\end{equation}
This is an equivalence of left $\partial_*(\Id)$-modules, under the left action of 
$$ \partial_*(\Id) \simeq B(1, \mr{Comm}, 1)^\vee $$
induced by the left coaction of $B(1, \mr{Comm}, 1)$ on $B(1, \mr{Comm}, \partial^*(\Sigma^\infty F))$. 


\begin{lem}\label{lem:Hsusp}
Under the isomorphisms
$$ \mr{susp}_* : H_*(\DD_i(F)(\Sigma X)) \xrightarrow{\cong} H_*(\DD_i(F\Sigma)(X)) $$
we have
$$ \bar{Q}^j \mr{susp}_*(x) = \mr{susp}_* (\bar{Q}^j x). $$
\end{lem}

\begin{proof}
The derivatives of the functor
$$ \Sigma: \Top_* \rightarrow \Top_* $$
are given by \cite[Examples~19.4]{AroneChing}
$$ \partial_i(\Sigma) \simeq \partial_i(\Id) \wedge S^i $$ 
(where $\Sigma_i$ acts on $S^i = (S^1)^{\wedge i}$ by permuting the factors).  The resulting equivalence of symmetric sequences
$$ \partial_*(\Sigma) \simeq \partial_*(\Id) \wedge S^* $$
is an equivalence of left $\partial_*(\Id)$-modules, where $\partial_*(\Id) \wedge S^*$ is given the left module structure induced from the left module structure of $\partial_*(\Id)$ on itself.

It follows that there are equivalences of left $\partial_*(\Id)$-modules
$$ \partial_*(F\Sigma) \simeq B(\partial_*(F), \partial_*(\Id), \partial_*(\Sigma)) \simeq \partial_*(F) \wedge S^*. $$
 Therefore,  
the following diagram commutes.
$$
\xymatrix{
\partial_2(\Id) \wedge_{h\Sigma_2} \DD_i(F\Sigma)(X)^{\wedge 2} \ar[r] \ar[d]_\simeq  
& \DD_{2i}(F\Sigma)(X) \ar[d]^\simeq \\
\partial_2(\Id) \wedge_{h\Sigma_2} \DD_i(F)(\Sigma X)^{\wedge 2} \ar[r]   
& \DD_{2i}(F)(\Sigma X)
}
$$
The result follows by applying homology to the above diagram.
\end{proof}

\begin{lem}\label{lem:Homega}
Under the isomorphisms
$$ \omega_*: H_*(\DD_i(\Omega F)(X)) \xrightarrow{\cong} H_*(\Sigma^{-1}\DD_i(F)(X)) $$
we have 
$$ \bar{Q}^j \omega_*(x) = \omega_*(\bar{Q}^j x). $$
\end{lem}

\begin{proof}
The proof follows the sames lines as the proof of Lemma~\ref{lem:Hsusp}, and relies on the observation that there is an equivalence of left $\partial_*(\Id)$-modules \cite[Examples~19.4]{AroneChing}
$$ \partial_*(\Omega F) \simeq F(S^1, \partial_*(F)), $$
where the left $\partial_*(\Id)$-structure on the right-hand side of this equivalence is induced by the left module structure on $\partial_*(F)$, together with the diagonal on $S^1$.
\end{proof}

\begin{lem}\label{lem:HSq}
Under the isomorphisms of Lemma~\ref{lem:Sq}
$$ \mr{sqrt}_*: H_*(\DD_{2i}(F\Sq)) \xrightarrow{\cong} H_*(\DD_{i}(F)) $$
we have 
$$ \bar{Q}^j \mr{sqrt}_*(x) = \mr{sqrt}_* (\bar{Q}^j x). $$
\end{lem}

\begin{proof}
The functor
$$ \Sigma^\infty \Sq: \Top_* \rightarrow \Sp $$
clearly has dual derivatives
$$ \partial^i(\Sigma^\infty \Sq) \simeq 
\begin{cases}
(\Sigma_2)_+, & i = 2, \\
\ast, & i \ne 2.
\end{cases}
$$
We shall denote this symmetric sequence $2_*$.  The left action of $\mr{Comm}$ is trivial: there is an equivalence of left $\mr{Comm}$-modules
$$ 2_* \simeq 1_* \circ 2_*. $$
Using (\ref{eq:ACformula}), we have
\begin{align*}
 \partial_{*}(\Sq)  
& \simeq B(1, \mr{Comm}, 2)^\vee \\
& \simeq B(1, \mr{Comm}, 1 \circ 2)^\vee \\
& \simeq B(1, \mr{Comm}, 1 )^\vee \circ 2_* \\
& \simeq \partial_*(\Id) \circ 2_*. 
\end{align*}
We deduce, using the chain rule, that there is an equivalence of left $\partial_*(\Id)$-modules
$$ \partial_*(F \Sq) \simeq \partial_*(F) \circ 2_* $$
where the left $\partial_*(\Id)$-module structure on the right-hand side is induced from the left $\partial_*(\Id)$-module structure on $\partial_*(F)$.
In particular
\begin{align*}
\partial_{2i}(F\Sq) & \simeq (\Sigma_{2i})_+ \wedge_{\Sigma_i \wr \Sigma_2^i} \partial_i(F) \wedge (\Sigma_2^i)_+ \\
& \simeq (\Sigma_{2i})_+ \wedge_{\Sigma_i} \partial_i(F).
\end{align*}
(This gives yet another computation of $\DD_{2i}(F\Sq)(X)$.)  We deduce that the following diagram commutes.
$$
\xymatrix{
\partial_2(\Id) \wedge_{h\Sigma_2} \DD_{2i}(F\Sq)(X)^{\wedge 2} \ar[r] \ar[d]_\simeq  
& \DD_{4i}(F\Sq)(X) \ar[d]^\simeq \\
\partial_2(\Id) \wedge_{h\Sigma_2} \DD_{i}(F)(X \wedge X)^{\wedge 2} \ar[r]   
& \DD_{2i}(F)(X \wedge X)
}
$$
The lemma follows from applying homology.
\end{proof}

Write
$$ H_*(L(k)_n) = \FF_2\{\sigma^{k-n}\bar{Q}^J \iota_n \: : \: \text{$J$ CU, $\abs{J} = k$, $e(J) \ge n$} \}. $$
It will be convenient to define
$$ [j_1, \ldots, j_k] := \sigma^{k-n}\bar{Q}^{j_1} \cdots \bar{Q}^{j_k} \iota_n \in H_*(L(k)_n). $$

\begin{prop}\label{prop:HTakayasu}
The fiber sequences in Corollary~\ref{cor:Takayasu} induce short exact sequences in homology, and we have
\begin{align*}
P_*([j_1, \ldots, j_{k-1}]) & = [j_1, \ldots, j_{k-1}, n], \\
E_*([j_1, \ldots, j_k]) & = 
\begin{cases}
0, & j_k = n, \\
[j_1, \ldots, j_k], & j_k > n.  
\end{cases}
\end{align*}
\end{prop}

\begin{proof}
Since on $\DD_1$ the natural transformation $E$ is an equivalence, the associated map in homology
$$
H_*(\DD_1(\Id)(S^n)) \xrightarrow[E_*]{} H_*(\DD_1(\Omega \Sigma)(S^n)) \xrightarrow[\omega_* \mr{susp}_*]{\cong} H_*(\Sigma^{-1} \DD_1(S^{n+1}))
$$
is given by:
$$ \omega_* \mr{susp}_* E_*(\iota_n) = \sigma^{-1} \iota_{n+1}. $$
By naturality of the operations $\bar{Q}^j$, together with Lemmas~\ref{lem:Hsusp} and \ref{lem:Homega}, 
we deduce that
\begin{align*}
\omega_* \mr{susp}_* E_*(\bar{Q}^{j_1} \cdots \bar{Q}^{j_k} \iota_n) 
& =  \bar{Q}^{j_1} \cdots \bar{Q}^{j_k} \omega_* \mr{susp}_* E_* \iota_n \\
& = \begin{cases}
0, & j_k = n, \\
\sigma^{-1}\bar{Q}^{j_1} \cdots \bar{Q}^{j_k} \iota_{n+1}, & j_k > n.  
\end{cases}
\end{align*}
(Here we are using the excess relation in $\bar{\mc{R}}_n$: $\bar{Q}^j \iota_{n+1} = 0$ for $j < n+1$.)

Define
$$ \mr{fib}_{E}: \Top_* \rightarrow \Top_* $$
by 
$$ \mr{fib}_E(X) := \mr{fiber}(X \xrightarrow{E} \Omega \Sigma X). $$
As the composite
$$ \Id \xrightarrow{E} \Omega \Sigma \xrightarrow{H} \Omega \Sigma \Sq $$
is \emph{naturally} null homotopic,  there is an induced natural transformation
$$ \mr{fib}_E \xrightarrow{\rho} \Omega^2 \Sigma \Sq. $$ 
The EHP sequence implies that this natural transformation gives equivalences
$$ \mr{fib}_E(S^n) \xrightarrow[\rho]{\simeq} \Omega^2 S^{2n+1} $$
and Lemma~\ref{lem:EHPfiber} implies that there are induced equivalences 
$$ \DD_{2^k}(\mr{fib}_E)(S^n) \xrightarrow[\rho]{\simeq} \DD_{2^k}(\Omega^2 \Sigma \Sq)(S^n) \xrightarrow[\omega^2 \mr{susp} \, \mr{sqrt}]{\simeq} \Sigma^{-2} \DD_{2^{k-1}}(S^{2n+1}) $$ 
on the layers.
As $\partial_2(\Id) \simeq S^{-1}$, in the case of $k = 1$, the fiber sequence
$$ \DD_{2}(\mr{fib}_E)(S^n) \xrightarrow{i} \DD_2(\Id)(S^n) \rightarrow \DD_2(\Omega \Sigma)(S^n) $$
is equivalent to the fiber sequence
$$
\begin{array}{ccccc}
\ul{S}^{2n-1} & \rightarrow & \Sigma^{-1} (\ul{S}^n)^{\wedge 2}_{h \Sigma_2} & \xrightarrow{E} & \Sigma^{-2} (\ul{S}^{n+1})^{\wedge 2}_{h\Sigma_2} \\
& & \Vert & & \Vert \\
& & \Sigma^{n-1}P_n^\infty & & \Sigma^{n-1}P_{n+1}^\infty
\end{array}
$$
Therefore we deduce
$$ P_*(\sigma^{-2} \iota_{2n+1}) = \bar{Q}^n \iota_n. $$
Again naturality of the operations $\bar{Q}^j$, combined with Lemmas~\ref{lem:Hsusp}, \ref{lem:Homega}, and \ref{lem:HSq}, implies
\begin{align*}
 P_*(\sigma^{-2} \bar{Q}^{j_1} \cdots \bar{Q}^{j_{k-1}} \iota_{2n+1})  
& = i_* \rho^{-1}_* \mr{sqrt}^{-1}_{*} s^{-1}_{*}\omega^{-2}_{*} \bar{Q}^{j_1} \cdots \bar{Q}^{j_{k-1}} \sigma^{-2} \iota_{2n+1} \\
& = \bar{Q}^{j_1} \cdots \bar{Q}^{j_{k-1}} i_* \rho^{-1}_* \mr{sqrt}^{-1}_{*} s^{-1}_{*}\omega^{-2}_{*}  \sigma^{-2}\iota_{2n+1} \\
& = \bar{Q}^{j_1} \cdots \bar{Q}^{j_{k-1}} P_*(\sigma^{-2}\iota_{2n+1}) \\
& = \bar{Q}^{j_1} \cdots \bar{Q}^{j_{k-1}} \bar{Q}^n \iota_n.
\end{align*}
It is clear from these formulas that the sequence
$$ 0 \rightarrow H_*(\Sigma^{n} L(k)_{2n+1}) \xrightarrow{P_*} H_*(L(k)_n) \xrightarrow{E_*} H_*(L(k)_{n+1}) \rightarrow 0. $$
is exact.
\end{proof}

\section{Transfinite Atiyah-Hirzebruch spectral sequences}\label{sec:TAHSS}

Define $L(k)_n^m$ to be the fiber
$$ L(k)_n^m := \mr{fiber}\left( L(k)_n \xrightarrow{E^{m-n+1}} L(k)_{m+1} \right). $$
Note that $L(1)_n^m \simeq P_n^m$.
The homology of these spectra is easily computed with Proposition~\ref{prop:HTakayasu}.

\begin{prop}
We have
$$ H_*(L(k)_n^m) := \FF_2\{ [J] \: : \: J \: \mr{CU}, \abs{J} = k, \: n \le e(J) \le m \}. $$
\end{prop}

The spectra $L(k)_n^m$ endow $L(k)_n$ with an increasing filtration:
\begin{gather*}
\cdots \rightarrow L(k)_n^m \rightarrow L(k)_n^{m+1} \rightarrow \cdots \rightarrow L(k)_n, \\
\varinjlim_m L(k)_n^m \simeq L(k)_n. 
\end{gather*}
The filtration quotients are given  by the cofiber sequences
$$ L(k)_n^{m-1} \rightarrow L(k)_n^m \rightarrow \Sigma^{m} L(k-1)_{2m+1}. $$
Thus, the $L(k)_n$ spectra are built out of the spectra $L(k-1)_{2m+1}$, and there are associated Atiyah-Hirzebruch-type spectral sequences
$$ E^1_{*,t} = \bigoplus_{m \ge n} \pi_t(\Sigma^{m}L(k-1)_{2m+1}) \Rightarrow \pi_t(L(k)_n). $$
These spectral sequences give an inductive means of computing $\pi_*(L(k)_n)$, starting with $\pi_*(L(0)) = \pi^s_*$.
One way to think about this is that there is a sequence of spectral sequences
\begin{equation*}
\bigoplus_{\stackem{(j_1, \ldots, j_k) \: CU}{j_k \ge n}} \pi_t(\ul{S}^{j_1+\cdots + j_k}) \Rightarrow 
\bigoplus_{\stackem{(j_2, \ldots, j_{k}) \: CU}{j_k \ge n}} \pi_t(\Sigma^{j_2+\cdots +j_k} L(1)_{2j_2+1}) \Rightarrow 
\cdots
\Rightarrow \pi_t(L(k)_n).
\end{equation*}

Alternatively, one can view this as a single transfinite spectral sequence (in the sense of \cite{Hu}). 
As in Section~\ref{sec:Groth}, let $\mc{G}(\omega^k)$ be the Grothendieck group of ordinals less than $\omega^k$:
$$ \mc{G}(\omega^k) = \{ j_1 + j_2 \omega + \cdots + j_k \omega^{k-1} \: : \: j_s \in \ZZ \}. $$ 
Then we define a degreewise finite $\mc{G}(\omega^k)$-indexed tower under $L(k)_n$ (Section~\ref{sec:towers}) 
$$
\{  [L(k)_n]_{j_1 + j_2 \omega + \cdots + j_k \omega^{k-1}} \} = \{ [L(k)_n]_{j_1, \ldots, j_k} \} 
$$
as follows.
\begin{align*}
[L(k)_n]_{0,\ldots, 0, j_k} & = 
\begin{cases}
L(k)_{j_k} & j_k > n, \\
L(k)_n, & \text{else},
\end{cases} 
\\
[L(k)_n]_{0, \ldots, 0, j_{k-1}, j_k} & =
\begin{cases}
\mr{cofiber} \left( \Sigma^{j_{k}} L(k-1)^{j_{k-1}-1}_{2j_k+1} \rightarrow [L(k)_{n}]_{0, \ldots, 0, j_k} \right), & 
j_k \ge n, \\
& j_{k-1} > 2j_{k}+1, \\
\\
F_{0, \ldots, 0,  j_k} L(k)_n, & \text{else},
\end{cases}
\\
& \vdots
\\
[L(k)_n]_{j_1, \ldots, j_k} & = 
\begin{cases}
\mr{cofiber} \left(  \Sigma^{j_2+\cdots+j_k} L(1)^{j_{1}-1}_{2j_2+1} \rightarrow [L(k)_{n}]_{0, j_2, \ldots, j_k} \right), & j_1 > 2j_{2}+1, \\
& j_s \ge 2j_{s+1}+1 \\
& \mr{for} \: 2 \le s \le k-1, \\
& j_k \ge n, \\
\\
[L(k)_n]_{0, j_2, \ldots, j_k}, & \text{else}.
\end{cases}
\end{align*}
We have
$$ H_*([L(k)_n]_{J}) = \FF_2 \{ [T] \: : \: T \: \text{CU}, \: \abs{T} = k, \: e(T) \ge n, \: T \ge J \}. $$
The filtration induced by this tower implies that cells $e^{\norm{J}}$, associated to CU sequences $J$ of length $k$, only attach to cells corresponding to lesser sequences.  The quotient complex $[L(k)_n]_{J}$ is obtained from $L(k)_n$ by collapsing out all cells corresponding to sequences $T < J$.

For CU sequences $J = (j_1, \ldots, j_k)$, $j_k \ge n$, and 
$$ \mu(J) := j_1 + j_2 \omega + \cdots + j_k \omega^{k-1} \in \mc{G}(\omega^k), $$
there are fiber sequences
$$ \ul{S}^{\norm{J}} \rightarrow [L(k)_n]_{\mu(J)} \rightarrow [L(k)_{n}]_{\mu(J)+1}. $$
For all other $\mu \in \mc{G}(\omega^k)$ we have $[L(k)_n]_\mu = [L(k)_n]_{\mu+1}$.
The $\mc{G}(\omega^{k})$-indexed transfinite spectral sequence associated to this tower (Section~\ref{sec:towers}) takes the form
$$ E^1_{t, \mu}(L(k)_n) \Rightarrow \pi_{t}(L(k)_n) $$
with
$$
E^1_{t, \mu} (L(k)_{n})
= 
\begin{cases}
\pi_t(\ul{S}^{\norm{J}}), & \mu = \mu(J), \: J \: \text{CU}, \: \abs{J} = k, \: e(J) \ge n, \\
0, & \text{else}.
\end{cases}
$$

We shall refer to this as the \emph{transfinite Atiyah-Hirzebruch spectral sequence} (TAHSS).  
As the only important $\mc{G}(\omega^k)$ indices for this spectral sequence are those of the form $\mu(J)$, we will write
$E^1_{t,J}(L(k)_n)$ for $E^1_{t,\mu(J)}(L(k)_n)$.
Elements in the $E^1_{t,J}$-term of the TAHSS will be denoted
$$ \alpha[J] $$
for $\alpha \in \pi_t(\ul{S}^{\norm{J}})$.  
Differentials in the TAHSS take the form
$$ d^{L(k)_n}_{\mu}(\alpha[J]) = \alpha'[J'] $$
where $J' < J$,  
$$ \mu = j_1-j'_1+(j_2-j_2')\omega+ \cdots (j_{k-1}-j'_{k-1})\omega^{k-1}, $$
and 
$$ \abs{\alpha'} + \norm{J'} = \abs{\alpha} + \norm{J} -1. $$
As the length $\mu$ of the above differential is somewhat cumbersome to write as an element of $\mc{G}(\omega^k)$, and is completely determined by $J$ and $J'$, it will typically be omitted from the notation.

Differentials in the TAHSS, as in any Atiyah-Hirzebruch-type spectral sequence, can be effectively computed from a sound understanding of the attaching map structure in the CW-spectrum $L(k)_n$.  The cells $e^{\norm{J}}$ associated to CU sequences $J$ are in bijective correspondence with the homology elements
$$ [J] \in H_*(L(k)_n). $$
Often, these attaching maps can be determined from the action of the dual Steenrod operations on
$$ H_*(L(k)_n) \cong \bar{\mc{R}}_n $$
given by the Nishida relations.  This will be our main technique for determining TAHSS differentials in the sample calculations of Chapter~\ref{sec:calculations}.

\section{Transfinite Goodwillie spectral sequence}\label{sec:TGSS}

Combining the Goodwillie spectral sequence with the TAHSS gives a sequence of spectral sequences
\begin{equation*}
\bigoplus_{k \ge 0} \bigoplus_{\stackem{(j_1, \ldots, j_k) \: CU}{j_k \ge n}} \pi_t(\ul{S}^{j_1+\cdots + j_k}) \Rightarrow  \cdots \Rightarrow \bigoplus_{k \ge 0} \pi_t(L(k)_n) \Rightarrow \pi_{t+n-k}(S^n).
\end{equation*}
which we again wish to regard as a transfinite spectral sequence: the transfinite Goodwillie spectral sequence (TGSS).
Since the TAHSS's are indexed on $\mc{G}(\omega^k)$ for varying $k$, we must index the TGSS on $\mc{G}(\omega^\alpha)$ for an ordinal $\alpha$ larger than $\omega^{\omega}$.

To accomplish this, define a degreewise finite $\mc{G}(\omega^{\omega+1})$-indexed tower $\{ [S^n]_\mu \}$
on $S^n$ as follows.  There are fiber sequences
$$ P_{2^k}(S^n) \rightarrow P_{2^{k-1}}(S^n) \rightarrow \Omega^{\infty} \Sigma^{n-k+1}L(k)_n. $$
We define
\begin{align*}
[S^n]_0 & := P_1(S^n), \\
[S^n]_\mu & := \ast, \quad \mu > 0.
\end{align*}
For $k \ge 1$, and $\mu$ of the form
$$ \mu = j_1+j_2\omega+\cdots+j_k\omega^{k-1}-k\omega^\omega, $$
we define $[S^n]_\mu$ to be the fiber
$$ [S^n]_{\mu} \rightarrow P_{2^{k-1}}(S^n) \rightarrow \Omega^\infty \Sigma^{n-k+1} [L(k)_n]_{j_1+\cdots+j_k \omega^{k-1}}. $$
In particular, we have
$$ [S^n]_{-k\omega^\omega} = P_{2^k}(S^n). $$
We fill in the rest by setting 
$$ [S^n]_\mu := P_{2^{k-1}}(S^n) $$
for $k \ge 1$ and either
$$ \omega^k -k\omega^\omega \le \mu  < -(k-1) \omega^{\omega} $$
or $\mu$ in the right subset of the $(-k\omega^\omega, \omega^k-k\omega^\omega)$-gap (see \cite[Def.~1]{Hu}).

For $k \ge 0$, $J = (j_1, \ldots, j_k)$ a CU sequence with $j_k \ge n$, and
$$ \mu[J] := j_1 + j_2\omega + \cdots + j_k \omega^{k-1} - k\omega^\omega, $$
there are fiber sequences
$$ Q S^{n-k+\norm{J}} \rightarrow [S^n]_{\mu[J]} \rightarrow [S^n]_{\mu[J]+1}. $$
For all other $\mu \in \mc{G}(\omega^{\omega+1})$, we have $[S^n]_{\mu} = [S^n]_{\mu+1}$.
The TGSS is the $\mc{G}(\omega^{\omega+1})$-indexed transfinite spectral sequence associated to this tower under $S^n$; it takes the form
$$ E^1_{t,\mu}(S^n) \Rightarrow \pi_t(S^n) $$
with
$$
E^1_{t,\mu}(S^n) =
\begin{cases}
\pi_{t}(\ul{S}^{n-\abs{J}+\norm{J}}), & \mu = \mu[J], \: J \: \mr{CU}, \: \abs{J} \ge 0, \: e(J) \ge n, \\
0, & \text{else}.
\end{cases}
$$ 
This spectral sequence
computes unstable homotopy groups of spheres from stable homotopy groups of spheres.

As the only relevant terms in this spectral sequence are those of the form $E^1_{t,\mu[J]}$ for $J$ a CU sequence with $e(J) \ge n$, we will often write
$E^1_{t,J}(S^n)$ for $E^1_{t,\mu[J]}(S^n).$
Elements in the $E^1_{t,J}$-term of the TGSS will be denoted
$$ \alpha[J] $$
for $\alpha \in \pi_t(\ul{S}^{n-\abs{J}+\norm{J}})$.
If $J = \emptyset$, the empty sequence of length zero, we shall simply write
$$ \alpha = \alpha[\emptyset] \in E^1_{t, \emptyset}(S^n). $$
Differentials in the TGSS spectral sequence take the form
$$ d^{S^n}_{\mu}(\alpha[J]) = \alpha'[J'] $$
where $J' < J$, where $\mu$ is given by
$$ \mu = j_1-j'_1+(j_2-j_2')\omega+ \cdots + (j_{k}-j'_{k})\omega^{k-1} -j'_{k+1}\omega^{k}-\cdots - j_{k'}\omega^{k'-1}+(k'-k)\omega^\omega, $$
and 
$$ \abs{\alpha'} + \norm{J'} - \abs{J'} = \abs{\alpha} + \norm{J} - \abs{J} - 1. $$
As the expression for $\mu$ is rather cumbersome, and completely determined by $J$ and $J'$, we will often omit it from the notation, expressing $d^{S^n}_\mu$ simply as $d^{S^n}$. 

\chapter{Goodwillie filtration and the $P$ map}\label{sec:Gfilt}

The purpose of this chapter is to contemplate the meaning of the Goodwillie filtration, and the nature of detecting elements in the TGSS.  The notion of Goodwillie filtration is reviewed in Section~\ref{sec:Goodfilt}.  Section~\ref{sec:DOI} introduces the notion of the degree of instability of an element, and assigns elements of the unstable stems stable names through a tracking of its ``lineage'' in the EHP sequence.  In Section~\ref{sec:EPTAHSS} we show the $E$ and $P$ maps give well behaved maps of TAHSS's.  In Section~\ref{sec:EPTGSS} we show that these maps of TAHSS's extend to maps of TGSS's.  In Section~\ref{sec:detectionTGSS} we observe that these maps of spectral sequences imply a close connection between Goodwillie filtration and degree of instability, and furthermore the lineage of unstable elements is often encoded in the names of detecting elements in the TGSS. In Section~\ref{sec:Whiteprod} we note that the relationship of the $P$ map to Whitehead products gives yet another interpretation of Goodwillie filtration.

\section{Goodwillie filtration}\label{sec:Goodfilt}

We shall say that a non-zero element $\beta$ in $\pi_t(S^n)$ is \emph{of Goodwillie filtration $2^k$} if its image in  $\pi_t(P_{2^{k-1}}(S^n))$ is null, but its image in $\pi_t(P_{2^{k}}(S^n))$ is non-zero.  

The elements of Goodwillie filtration $2^k$ are precisely the elements which are detected by elements of the form
$$ \alpha[j_1, \ldots, j_k] $$
in the TGSS.  

It is well known that the elements of Goodwillie filtration $1$ are precisely those homotopy elements which are stably non-trivial.  We shall refer to such homotopy elements as \emph{stable elements}.  The elements of Goodwillie filtration greater than $1$ are the unstable elements.  It stands to reason that the Goodwillie filtration is a measure of the degree of instability.  

\section{The genealogy of unstable elements}\label{sec:DOI}

The EHP sequence supplies another measure of instability, which we will now explain.
If $\beta \in \pi_t(S^n)$ is unstable and nonzero, let $r_1$ be minimal such that
$$ E^{r_1+1}(\beta) = 0 \in \pi_{t+r_1+1}(S^{n+r_1+1}) $$
(i.e. $\beta$ dies on $S^{n+r_1+1}$).
By exactness of the EHP sequence, there exists a (necessarily non-trivial) element $\alpha_1 \in \pi_{t+r_1+2}(S^{2(n+r_1)+1})$ such that
$$ E^{r_1} (\beta) = P(\alpha_1). $$
We will refer to such an element $\alpha_1$ as a \emph{child} of $\beta$ and write
$$ \beta \in E^{-r_1}P(\alpha_1). $$
There are two possibilities: $\alpha_1$ is either stable or unstable.  If $\alpha_1$ is unstable, then we can repeat the above procedure, and express
$$ \beta \in E^{-r_1}PE^{-r_2}P(\alpha_2) $$
for a child $\alpha_2$ of $\alpha_1$.
Continuing in this manner, if we have
$$ \beta \in E^{-r_1}P \cdots E^{-r_\ell}P (\alpha_\ell) $$
we shall say that $\alpha_\ell$ is a \emph{($\ell$th generation) descendant} of $\beta$.
 We shall say that $\beta$ is \emph{unstable of degree $k$} if $k$ is maximal such that there exists a $k$th generation descendant $\alpha_k$ of $\beta$: 
$$ \beta \in E^{-r_1}P \cdots E^{-r_k}P (\alpha_k). $$
The homotopy element $\alpha_k$ in the above expression is necessarily stable, and there is a zig-zag:
$$
\xymatrix@C+0em{
\alpha \in \pi_*^s &
\pi_*S^{2j_1+1} \ni \alpha_k \ar[l]^-{E^\infty} \ar@<-2em>[d]^{P}
\\
& \pi_* S^{j_1} \phantom{\ni \alpha_k} &
\pi_*S^{2j_2+1} \ni \alpha_{k-1} \ar[l]^-{E^{r_k}}  \ar@<-2em>[d]^{P}&
\\
&& \pi_*S^{j_2} \phantom{\ni \alpha_{k-1}}  &
\\
& 
&& \pi_*S^{j_k} \ar@{}[ul]|{\ddots} & \pi_*S^{n} \ni \beta
\ar[l]^-{E^{r_1}}
\\
} $$
We call such a sequence $(\alpha_1, \ldots, \alpha_k)$ a \emph{lineage} of $\beta$, and write
$$ \beta \in \alpha\bra{j_1, j_2, \ldots, j_k}. $$
Note that the sequence $(j_1, \ldots, j_k)$ is necessarily CU of excess greater than or equal to $n$.

We wish to compare the notion of degree of instability with the Goodwillie filtration, and relate detection in the TGSS with the notion of lineage.  To prepare for this analysis, we will need to understand the maps of TAHSS's and TGSS's induced by the maps $P$ and $E$.  This will be accomplished in the next two sections.

\section{Behavior of the E and P maps in the TAHSS}\label{sec:EPTAHSS}

\begin{lem}\label{lem:Pfilt1}
For all $\mu \in \mc{G}(\omega^{k-1})$, there are maps of towers
$$
\xymatrix{
\Sigma^n L(k-1)_{2n+1} \ar[r]^P \ar[d] &
L(k)_n \ar[d] 
\\
\Sigma^n[L(k-1)_{2n+1}]_{\mu} \ar[r] &
[L(k)_n]_{\mu+n\omega^{k-1}}  
}
$$
\end{lem}

\begin{proof}
The maps are produced by a downward induction on $s$:   Given maps
$$ \Sigma^n [L(k-1)_{2n+1}]_{j_{s+1}\omega^s+\cdots+j_{k-1} \omega^{k-2}} \rightarrow 
[L(k)_n]_{j_{s+1}\omega^s+\cdots+j_{k-1} \omega^{k-2}+n\omega^{k-1}} $$ 
we obtain an induced map of cofibers:
$$
\xymatrix{
\Sigma^{j_{s+1}+\cdots+j_{k-1}+n} L(s)^{j_s-1}_{2j_{s}+1} \ar@{=}[r] \ar[d] &
\Sigma^{j_{s+1}+\cdots+j_{k-1}+n} L(s)^{j_s-1}_{2j_{s}+1} \ar[d] 
\\
\Sigma^n[L(k-1)_{2n+1}]_{j_{s+1}\omega^s+\cdots+j_{k-1} \omega^{k-2}} \ar[d] \ar[r] &
[L(k)_n]_{j_{s+1}\omega^s+\cdots+j_{k-1} \omega^{k-2}+n\omega^{k-1}} \ar[d]
\\
\Sigma^n [L(k-1)_{2n+1}]_{j_{s}\omega^{s-1}+\cdots+j_{k-1} \omega^{k-2}} \ar@{.>}[r] &
[L(k)_n]_{j_{s}\omega^{s-1}+\cdots+j_{k-1} \omega^{k-2}+n\omega^{k-1}}
}
$$
\end{proof}

\begin{cor}\label{cor:PTAHSS}
The $P$ map induces a map of TAHSS's
$$
\xymatrix{
E^1_{t-n,J}(L(k-1)_{2n+1}) \ar@{=>}[r] \ar[d]_{P_*} &
\pi_t \Sigma^n L(k-1)_{2n+1} \ar[d]_P \\
E^1_{t,[J,n]}(L(k)_n) \ar@{=>}[r] &
\pi_t L(k)_n
}
$$
which on $E^1$-terms is given by
$$ P_*(\alpha[J]) = \alpha[J, n]. $$
\end{cor}

\begin{lem}\label{lem:Efilt1}
For all $\mu \in \mc{G}(\omega^{k})$, there are maps of towers
$$
\xymatrix{
L(k)_{n} \ar[r]^E \ar[d] &
L(k)_{n+1} \ar[d] 
\\
[L(k)_{n}]_{\mu} \ar[r] &
[L(k)_{n+1}]_{\mu}  
}
$$
\end{lem}

\begin{proof}
The proof is similar to Lemma~\ref{lem:Pfilt1}, except simpler.
\end{proof}

\begin{cor}\label{cor:ETAHSS}
The $E$ map induces a map of TAHSS's
$$
\xymatrix{
E^1_{t,J}(L(k)_{n}) \ar@{=>}[r] \ar[d]_{E_*} &
\pi_t L(k)_{n} \ar[d]_E \\
E^1_{t,J}(L(k)_{n+1}) \ar@{=>}[r] &
\pi_t L(k)_{n+1}
}
$$
which on $E^1$-terms is given by
$$ P_*(\alpha[J]) = 
\begin{cases}
\alpha[J], & e(J) \ge {n+1}, \\
0, & \text{else}.
\end{cases}
$$
\end{cor}

\begin{rmk}\label{rmk:TAHSS}
Corollaries~\ref{cor:PTAHSS} and \ref{cor:ETAHSS} give an inductive scheme for computing the TAHSS's $\{ E^\alpha_{*,*}(L(k)_n)\}$.  Suppose we have computed the TAHSS for $L(k) = L(k)_1$.
$$ \{ E^\alpha_{*,*}(L(k)) \}. $$ 
Then the map of spectral sequences induced by $E^{n-1}$
$$ E^{n-1}_*: \{ E^\alpha_{*,*}(L(k)) \} \rightarrow \{ E^\alpha_{*,*}(L(k)_n) \} $$
\emph{determines} the spectral sequence $\{ E^\alpha_{*,*}(L(k)_n) \}$ as it is just a truncation of the spectral sequence $\{ E^\alpha_{*,*}(L(k)) \}$.  i.e., the TAHSS for $L(k)_n$ is obtained from the TAHSS for $L(k)$ by performing the following actions:
\begin{itemize}
\item Set $E^1_{t,J} = 0$ for $e(J) < n$.
\item Set to zero all differentials of the form
$$ d^{L(k)} (\alpha[J]) = \alpha'[J'] $$
with $e(J') < n$.
\end{itemize}
The map $P$ faithfully embeds the TAHSS for $L(k-1)_{2m+1}$ into the TAHSS for $L(k)_m$.  i.e., the differentials
$$ d^{L(k)_m}: E^{*}_{*,[J,m]}(L(k)_m) \rightarrow E^{*}_{*,[J',m]}(L(k)_m) $$
are in bijective correspondence with the differentials
$$ d^{L(k-1)_{2m+1}}: E^{*}_{*,J}(L(k-1)_{2m+1}) \rightarrow E^{*}_{*,J'}(L(k-1)_{2m+1}). $$
Combining this with the suspension truncation, we deduce that the differentials in the TAHSS for $L(k)$ of the form
$$ d^{L(k)}(\alpha[J,m]) = \alpha'[J',m] $$
come from differentials
$$ d^{L(k-1)}(\alpha[J]) = \alpha'[J'] $$
in the TAHSS for $L(k-1)$.

Thus, if we want to compute the TAHSS for $L(k)_n$, we can start by computing the AHSS for $L(1)$.  By truncating this spectral sequence appropriately,  this gives us all of the differentials of the form
$$ d^{L(2)}(\alpha[j_1, j_2]) = \alpha'[j'_1, j_2] $$
in the TAHSS for $L(2)$.  We then must compute only the remaining longer differentials of the form
$$ d^{L(2)}(\alpha[j_1, j_2]) = \alpha'[j'_1, j'_2] $$
for $j'_2 < j_2$.
Having computed the TAHSS for $L(2)$, we now know all of the differentials in the TAHSS for $L(3)$ of the form
$$ d^{L(3)}(\alpha[j_1, j_2, j_3]) = \alpha'[j_1', j'_2, j_3], $$
and we are left with computing the differentials which lower $j_3$.
We continue in this manner until we get up to $L(k)$, and then truncate to get $L(k)_n$.  This procedure will be implemented to compute $\pi_t L(k)$ for $t \lesssim 20$ and $k \le 3$ in Chapter~\ref{sec:calculations}.
\end{rmk}

\section{Behavior of  the E and P maps in the TGSS}\label{sec:EPTGSS}

\begin{lem}\label{lem:Pfilt}
For all $k \ge 1$ and $\mu \in \mc{G}(\omega^{\omega})$, there are maps of towers
$$
\xymatrix{
\Omega^2 S^{2n+1} \ar[r]^P \ar[d] &
S^n \ar[d] 
\\
\Omega^2[S^{2n+1}]_{\mu-(k-1)\omega^\omega} \ar[r] &
[S^{2n+1}]_{\mu+n\omega^{k-1}-k\omega^\omega} 
}
$$
\end{lem}

\begin{proof}
Suppose that $k \ge 2$.
Corollary~\ref{cor:EHPfiber} implies that $P$ lifts to give compatible maps
$$ \Omega^2 P_{2^{k-1}}(S^{2n+1}) \rightarrow 
P_{2^k}(S^n). $$
Lemma~\ref{lem:Pfilt1} then produces maps which induce the desired maps on fibers.
$$
\xymatrix{
\Omega^2 [S^{2n+1}]_{\mu-(k-1)\omega^\omega} \ar@{.>}[r] \ar[d] &
[S^n]_{\mu+n\omega^{k-1}-k\omega^{omega}} \ar[d] 
\\
\Omega^2 P_{2^{k-2}}(S^{2n+1}) \ar[r] \ar[d] &
P_{2^{k-1}}(S^n) \ar[d] 
\\
\Omega^\infty \Sigma^{2n-k+1} [L(k-1)_{2n+1}]_{\mu} \ar[r] &
\Omega^\infty \Sigma^{n-k+1} [L(k)_n]_{\mu+n\omega^{k-1}}
}
$$
This argument certainly makes sense for $\mu \in \mc{G}(\omega^{k-1})$, but it also makes sense for $\mu \in \mc{G}(\omega^\omega)$ if we define 
$$ [L(k)_n]_\mu := \begin{cases}
L(k)_n,  & \mu \in \mc{G}(\omega^\omega)\backslash \mc{G}(\omega^{k-1}), \: \mu < 0, \\
\ast,  & \mu \in \mc{G}(\omega^\omega)\backslash \mc{G}(\omega^{k-1}), \: \mu > 0 \\
\end{cases}
$$
for $\mu \in \mc{G}(\omega^\omega) \backslash \mc{G}(\omega^{k-1})$.
For $k = 1$, the lemma is verified in a similar manner, but this fringe case is more trivial.  
\end{proof}

This map of $\mc{G}(\omega^{\omega+1})$-indexed towers of Lemma~\ref{lem:Pfilt} induces a map of TGSS's.

\begin{lem}\label{lem:PTGSS}
The $P$ map induces a map of TGSS's:
$$
\xymatrix{
E^1_{t+2, J}(S^{2n+1}) \ar@{=>}[r] \ar[d]_{P_*} &
\pi_{t+2}(S^{2n+1}) \ar[d]^P\\
E^1_{t, [J,n]}(S^{n}) \ar@{=>}[r] &
\pi_t(S^n)
}
$$
On $E^1$-terms, this map is described by
$$ P_*(\alpha[J]) = \alpha[J,n]. $$
In particular, the $P$ map takes an element of Goodwillie filtration $2^{k-1}$ to an element of Goodwillie filtration at least $2^k$. 
\end{lem}

Note that the transfinite indexing of the TGSS shifts under the $P$ map, as does the length of all potential non-trivial differentials. The reader need not let this be an issue of concern: the $P$ map does not change the \emph{order} of differentials.  Explicitly, the $P$ map sends a differential
$$ d^{S^{2n+1}}(\alpha[J]) = \alpha'[J'] $$
to a differential
$$ d^{S^n}(\alpha[J,n]) = \alpha'[J',n] $$

The analysis of the $E$ map is similar, but easier, using Corollary~\ref{cor:EHPfiber}.

\begin{lem}\label{lem:Efilt}
There is a map of $\mc{G}(\omega^{\omega+1})$-indexed towers
$$
\xymatrix{
S^n \ar[r]^E \ar[d] &
\Omega S^{n+1} \ar[d] 
\\
[S^n]_\mu \ar[r] &
\Omega [S^{n+1}]_{\mu}
}
$$
\end{lem}

\begin{lem}\label{lem:ETGSS}
The $E$ map induces a map of TGSS's:
$$
\xymatrix{
E^1_{t, J}(S^{n}) \ar@{=>}[r] \ar[d]_{E_*} &
\pi_{t}(S^{n}) \ar[d]^E\\
E^1_{t+1, J}(S^{n}) \ar@{=>}[r] &
\pi_{t+1}(S^{n+1})
}
$$
On $E^1$-terms, this map is described by
$$ E_*(\alpha[J]) = 
\begin{cases}
\alpha[J], & e(J) \ge {n+1}, \\
0, & \text{else}. 
\end{cases}
$$
In particular, the $E$ map takes an element of Goodwillie filtration $2^{k}$ to an element of Goodwillie filtration at least $2^k$.
\end{lem}

\section{Detection in the TGSS}\label{sec:detectionTGSS}

The relationship between Goodwillie filtration and degree of instability is explained in the following immediate consequence of Lemmas~\ref{lem:PTGSS} and \ref{lem:ETGSS}.

\begin{thm}\label{thm:TGSS}
Suppose that we have a TGSS element
$$ \alpha[j_1, \ldots, j_k] \in E^1_{*, J}(S^n). $$
Then on the level of TGSS $E_1$-terms, we have 
$$ \alpha[j_1, \ldots, j_k] \in E_*^{-\ell_k}P_*E_*^{-\ell_{k-1}}P_* \cdots E_*^{-\ell_1}P_* \alpha  $$
for $\alpha \in E^1_{*,\emptyset}(S^{2j_1+1})$ and
\begin{align*}
\ell_s & = j_s - (2j_{s+1}+1), \quad 1 \le s < k \\
\ell_k & = j_k - n.
\end{align*}
If, for all $s$, the elements 
$$ \alpha[j_1, \ldots, j_s] \in E^1_{*, [j_1,\cdots,j_s]}(S^{2j_s+1}) $$
are permanent cycles in the TGSS, converging to elements $\alpha_s \in \pi_*(S^{2j_s+1})$, then we have
$$ E^{\ell_s} (\alpha_{s+1}) \equiv P(\alpha_s) $$
modulo elements of higher Goodwillie filtration.
\end{thm}
 
Theorem~\ref{thm:TGSS} tells us that, \emph{under favorable circumstances}, if $\beta \in \pi_*(S^n)$ is detected by 
$$ \alpha[j_1, \ldots, j_k], $$
then 
$$ \beta \in E^{-\ell_k}PE^{-\ell_{k-1}}P \cdots E^{-\ell_1}P \td{\alpha} $$
where $\ell_s$ are as in Theorem~\ref{thm:TGSS}, and $\td{\alpha} \in \pi_*(S^{2j_1+1})$ stabilizes to $\alpha \in \pi_*^s$.  \emph{Under favorable circumstances}, the degree of instability of the element $\beta$ is equal to $k$, and $\beta$ has lineage
$$ \beta \in \alpha\bra{j_1, \ldots, j_k}.$$  
However, as we are dealing with maps of spectral sequences in opposing directions, these principles could potentially fail wildly.  

One might conjecture that every element of Goodwillie filtration $2^k$ has degree of instability equal to $k$.  However, this is easily seen to be false, as the following example illustrates.

\begin{ex}\label{ex:stupid}
This example implicitly uses the calculations of Chapter~\ref{sec:calculations}.  Define
$$ x = \br{\alpha_{6/3}[4]} + \br{1[11,5]} \in \pi_{18}(S^4) $$
where $\br{\alpha_{6/3}}[4]\in  \pi_{18}(S^4)$ is the unique element element detected by $\alpha_{6/3}[4]$ in the TGSS for $S^4$ with the property that
$$ E(\br{\alpha_{6/3}[4]}) = 0, $$
and $\br{1[11,5]}\in \pi_{18}(S^4)$ is the unique element detected by $1[11, 5]$ in the TGSS for $S^4$.  Then 
$x$ has Goodwillie filtration $2^1$, but degree of instability $2$.  Indeed 
$$ E(x) = P(P(1))$$
for $1 \in \pi_{23}(S^{23})$ and $x$ has lineage
$$ x \in 1\bra{11,5} $$
yet $x$ is detected by $\alpha_{6/3}[4]$ in the TGSS for $S^4$.
\end{ex}

The reader should rightfully feel that Example~\ref{ex:stupid} fails on a technicality: our choice of $x$ detected by $\alpha_{6/3}[4]$ could have been modified to have degree of instability equal to $1$.
A more reasonable question is therefore the following.

\begin{ques}\label{ques:TGSS}
Suppose that $x \in \pi_*(S^n)$ has Goodwillie filtration $2^k$, and is detected by $\alpha[J]$ in the TGSS for $S^n$.  Then does there exist an $x' \in \pi_*(S^n)$ which is equivalent to $x$ modulo Goodwillie filtration $2^{k+1}$, such that the degree of instability of $x'$ is $k$ and the lineage of $x'$ is given by
$$ x' \in \alpha\bra{J}? $$
\end{ques}

Question~\ref{ques:TGSS} does not always have an affirmative answer, as the following example illustrates.

\begin{ex}\label{ex:DOI}
Let $\br{\eta^2[13]} \in \pi_{27}(S^{13})$ be the unique element detected by $\eta^2[13]$ in the TGSS for $S^{13}$ (see the calculations of Chapter~\ref{sec:calculations}).  Then we have
$$ P(\br{\eta^2[13]}) = \br{\eta^2[13,6]} \in \pi_{25}(S^6) $$
where $\br{\eta^2[13,6]}$ is the (unique) element detected by $\eta^2[13,6]$ in the TGSS for $S^{6}$.  However, in Section~\ref{sec:bad} it is shown that the (unique) desuspension $\br{\alpha_{8/5}[5]}$ of $\br{\eta^2[13,6]}$ to $\pi_{24}(S^5)$ is detected by $\alpha_{8/5}[5]$ in the TGSS for $S^5$.  Thus 
$$ \br{\alpha_{8/5}[5]} \in E^{-1} P P\td{\eta^2}. $$
and 
$$ \br{\alpha_{8/5}[5]} \in \eta^2\bra{13,6}. $$
In particular, every element of $\pi_{24}(S^5)$ detected by $\alpha_{8/5}[5]$ has degree of instability equal to $2$.
\end{ex}

\begin{rmk}
By way of analogy, if $x \in \pi_*(S^n)$ is detected by $\alpha[J]$ in the TGSS for $S^n$, then the element $\alpha[J]$ should be regarded as the ``DNA'' of $x$.  Then the above examples illustrate that DNA tests are not $100\%$ effective at determining the lineage of an element.
\end{rmk}

Counterexamples to Question~\ref{ques:TGSS} are difficult to find, and represent non-trivial EHP phenomena.
Indeed, the calculations of Chapter~\ref{sec:calculations} demonstrate that Example~\ref{ex:DOI} is the \emph{only} counterexample to Question~\ref{ques:TGSS} in the Toda range.

Question~\ref{ques:TGSS} is clearly answered in the affirmative for $k = 0$.  It also holds for $k = 1$ in the metastable range, as the following proposition demonstrates.

\begin{prop}\label{prop:metastableGSS}
Suppose that $0 \ne \beta \in \pi_{t+m}(S^m)$ is detected by $\alpha[n]$ in the TGSS for $S^m$
and that $t \le 3n-2$.  Then 
$$ 0 \ne E^{n-m}\beta = P(\alpha) \in \pi_{t+n}(S^n) $$ 
under the isomorphism
$$ \alpha \in \pi_{t-n+1}^s \cong \pi_{t+n+2}(S^{2n+1}). $$
\end{prop}

\begin{proof}
The element $E^{n-m}(\beta)$ is detected by $\alpha[n]$ in the TGSS for $S^n$.  Using Remark~\ref{rmk:TAHSS} we see that every TGSS differential with target $\alpha[n]$ pulls back to the TGSS for $S^m$.  We deduce that $\alpha[n]$ is not the target of any differentials in the TGSS for $S^n$ and therefore $E^{n-m}\beta \ne 0$. 
Our conditions on $t$ ensure $\pi_{t-n+1}^s \cong \pi_{t+n+2}(S^{2n+1})$, and under this isomorphism the stable element $\alpha$ corresponds to an unstable element 
$$ \alpha \in \pi_{t+n+2}(S^{2n+1}) $$ 
detected by $\alpha[\emptyset]$ in the TGSS for $S^{2n+1}$.  In the TGSS for $S^n$, 
$$ P_*(\alpha[\emptyset]) = \alpha[n] $$
and thus $P(\alpha)$ and $E^{n-m}\beta$ agree modulo elements of Goodwillie filtration greater than $2$.  But our assumption on $t$ implies that there are no non-zero elements of $\pi_{t+n}(S^n)$ of Goodwillie filtration greater than $2$.  Thus $P(\alpha) = E^{n-m} \beta$, as desired.  
\end{proof}

\section{Relationship with Whitehead products}\label{sec:Whiteprod}

The $P$ map gets is name because of its close relationship with the Whitehead product, as explained in the following well known proposition (see \cite{Whitehead}).

\begin{prop}\label{prop:PWhitehead}
If $\alpha$ is
an element of 
$\pi_{i+n+1}(S^{2n+1})$, and there is an element $\td{\alpha} \in \pi_i(S^n)$ 
so that 
$\alpha = E^{n+1} \td{\alpha}$, then we have
$$ P(\alpha) = [\iota_n, \td{\alpha}]. $$
\end{prop}

Proposition~\ref{prop:PWhitehead}, together with Theorem~\ref{thm:TGSS}, suggests that 
$$ \alpha[j_1, \ldots, j_k]$$
is the prime candidate to detect
the 
iterated Whitehead product
$$ [\iota_{j_k}, [\iota_{j_{k-1}}, [\cdots [\iota_{j_1}, \td{\alpha}] \cdots ]
$$
where $\td{\alpha}$ is a suitable desuspension of $\alpha$.  (In fact, every
term of the Whitehead product must be desuspended appropriately for the
expression to make sense.)

\chapter{Goodwillie differentials and Hopf invariants}\label{sec:Ghopf}

In this chapter we prove theorems which can be used to determine differentials in the GSS.  In Section~\ref{sec:TEHPSS} we review the notion of the Hopf invariant (HI).  By constructing a transfinite refinement of the EHPSS (TEHPSS), we define a generalized Hopf invariant (GHI).  In Section~\ref{sec:metastable}, we recall the relationship between metastable homotopy and $P_2(X)$, and recall how the James-Hopf map and Kahn-Priddy theorem allow for the definition of another invariant, the stable Hopf invariant (SHI).  When the Hopf invariant is stable, it agrees with the stable Hopf invariant.  In Section~\ref{sec:GSSd1}, we show that the stable Hopf invariant can be used to compute $d_1$-differentials emanating from the zero line in the GSS.  In Section~\ref{sec:GSSdr}, we explain how generalized Hopf invariants can be used to compute $d_r$-differentials emanating from the zero line in the GSS.  In Section~\ref{sec:propdiffs}, we explain how the $E$ and $P$ maps allow us to propagate $d_r$-differentials emanating from the zero line of the GSS to arbitrary locations.  Section~\ref{sec:GWC} discusses the case of the Goodwilllie spectral sequence for $S^1$.  The Goodwillie-Whitehead conjecture predicts that this spectral sequence collapses at $E_2$ to give a proof of the Whitehead conjecture (Kuhn's theorem).  We explain how our $d_1$-differentials are compatible with this conjecture.  In Section~\ref{sec:exotic}, we discuss two classes of GSS differentials which do not arise as Hopf invariants: these are the geometric boundary effect differentials and the bizarre differentials.  The geometric boundary effect differentials come from the geometric boundary theorem.  The bizarre differentials come from the GSS for $S^1$.

\section{Hopf Invariants and the transfinite EHPSS}\label{sec:TEHPSS}

Consider the EHPSS:
$$ E^1_{m,t} = \pi_{t+m+1} S^{2m+1} \Rightarrow \pi_t^s $$
The \emph{Hopf invariant} $HI(\alpha)$ of an element $0 \ne \alpha \in \pi_t^s$ is the coset of elements $\beta \in \pi_{t+m+1}S^{2m+1}$ which detect $\alpha$ in the EHPSS.  For such a detecting element $\beta$ we write
$$ \beta \in HI(\alpha). $$

Explicitly, given $0 \ne \alpha \in \pi_t^s$, lift $\alpha$ to an unstable element $\td{\alpha} \in \pi_{t+m+1}(S^{m+1})$ for $m$ minimal.  Then the image of $\td{\alpha}$ under the H map
$$ \pi_{t+m+1}(S^{m+1}) \xrightarrow{H} \pi_{t+m+1}(S^{2m+1}) $$
satisfies
$$ H(\td{\alpha}) \in HI(\alpha). $$
Note that the minimality of $m$ implies that $H(\td{\alpha}) \ne 0$.
If we wish to emphasize the sphere of origin $S^{m+1}$ of such an element $\alpha$ with Hopf invariant $\beta \in \pi_t(S^{2m+1})$, we will  write 
$$ \beta(m) \in HI(\alpha) $$.

One can also speak of the Hopf invariant of an unstable element.
Given $0 \ne \alpha \in \pi_{t+n}(S^n)$, the Hopf invariant of $\alpha$ 
is the coset of elements in the $E^1$-term of the truncated EHPSS
$$ \bigoplus_{0 \le m < n} \pi_{t+m+1}(S^{2m+1}) \Rightarrow \pi_{t+n}(S^n) $$
which detect $\alpha$. 
Explicitly, lift $\alpha$ to a element $\td{\alpha} \in \pi_{t+m+1}(S^{m+1})$ for $m\le n-1$ minimal.  Then 
$$ 0 \ne H(\td{\alpha}) \in HI(\alpha) \subset \pi_{t+m+1}(S^{2m+1}). $$
Again, if $\beta  \in \pi_t(S^{2m+1})$ is a Hopf invariant of $\alpha$, and we wish to emphasize the sphere of origin $S^{m+1}$ of $\alpha$, we will write
$$ \beta(m) \in HI(\alpha). $$ 

Combining the EHPSS with the TGSS, there is a sequence of spectral sequences
\begin{equation*}
\bigoplus_{\stackem{k \ge 0}{m \ge 0}} \bigoplus_{\stackem{(j_1, \ldots, j_k) \: CU}{j_k \ge 2m+1}} \pi_{t-m+k}(\ul{S}^{j_1+\cdots + j_k}) \Rightarrow  \cdots \Rightarrow \bigoplus_{m \ge 0} \pi_{t+m+1}(S^{2m+1}) \Rightarrow \pi_t^s.
\end{equation*}
We wish to construct this sequence of spectral sequences as a single $\mc{G}(\omega^{\omega+2})$-indexed transfinite spectral sequence.  We will refer to the resulting transfinite spectral sequence as the \emph{transfinite EHP spectral sequence} (TEHPSS). 

The EHPSS arises from the filtration
$$ \cdots \rightarrow \Omega^m S^m \rightarrow \Omega^{m+1} S^{m+1} \rightarrow \cdots \rightarrow QS^0. $$ 
We wish to refine this filtration to a transfinite filtration indexed on $\mc{G}(\omega^{\omega+2})$.  To accomplish this, we define, for $m \ge 0$ and $\mu' \in \mc{G}(\omega^{\omega+1})$, 
$$ F_{\mu'+m\omega^{\omega+1}}QS^0 := \mr{fiber} \left(  
\Omega^{m+1}S^{m+1} \rightarrow \Omega^{m+1}S^{2m+1} \rightarrow \Omega^{m+1}[S^{2m+1}]_{\mu'+1} \right). $$
In particular, we have
$$ F_{m\omega^{\omega+1}}QS^0 = \Omega^{m+1} S^{m+1}. $$
For $J = (j_1, \ldots, j_k)$ define
\begin{equation*}
 \mu\bra{J,m} = j_1+j_2\omega+\cdots + j_k\omega^{k-1}-k\omega^\omega+m\omega^{\omega+1} \in \mc{G}(\omega^{\omega+2}).
\end{equation*}
For $m \ge 0$ and $J$ a CU sequence with $e(J) \ge 2m+1$, there are fiber sequences
$$ F_{\mu\bra{J,m}-1} QS^0 \rightarrow F_{\mu\bra{J,m}} QS^0 \rightarrow QS^{\norm{J}+m-\abs{J}}. $$
The TEHPSS is the $\mc{G}(\omega^{\omega+2})$-indexed spectral sequence associated to this filtration, and takes the form
$$ E^1_{t,\mu}  \Rightarrow \pi^s_{t} $$
where
$$
E^1_{t,\mu} = 
\begin{cases}
\pi_{t+\abs{J}}(\ul{S}^{\norm{J}+m}), 
& \mu = \mu\bra{J,m}, \: m \ge 0, \: J \: \mr{CU}, \: e(J) \ge 2m+1. \\ 
0 & \text{else}. 
\end{cases}
$$
This spectral sequence computes the stable homotopy groups of spheres from the stable homotopy groups of spheres.
We will denote $E^1_{t,\mu\bra{J,m}}$ by $E^1_{t, [J,m]}$.
Elements in $E^1_{*,[J,m]}$ given by  
$$ \gamma \in \pi_{t+\abs{J}}(\ul{S}^{\norm{J}+m}) $$
will be denoted
$$ \gamma[J,m]. $$

For $0 \ne \alpha \in \pi^s_{t}$ we define the \emph{generalized Hopf invariant} to be the coset
$$ GHI(\alpha) \subset \pi_{t+\abs{J}}(\ul{S}^{\norm{J}+m}) $$
of elements which detect $\alpha$ in the TEHPSS.  If
$$ \gamma[J,m] \in GHI(\alpha) $$
then $\alpha$ is born on $S^{m+1}$, with Hopf invariant which is detected by $\gamma[J]$ in the TGSS for $S^{2m+1}$.

\section{Stable Hopf invariants and metastable homotopy}\label{sec:metastable}

For $X \in \Top_*$, let
$$ JH: QX \rightarrow QX^{\wedge 2}_{h\Sigma_2} $$
denote the James-Hopf map.  It is adjoint to the map
$$ \Sigma^\infty QX \rightarrow \Sigma^{\infty} X^{\wedge 2}_{h\Sigma_2} $$
coming from the Snaith splitting.
The following theorem is well-known (see \cite{AroneMahowald}).

\begin{thm}\label{thm:metastable}
Let
$$ P_2(X) \rightarrow P_1(X) = QX $$
be the first map in the Goodwillie tower of the identity. 
Then the sequence
$$ P_2(X) \rightarrow QX \xrightarrow{JH} QX^{\wedge 2}_{h\Sigma_2}. $$ 
is a fiber sequence.
\end{thm}

Theorem~\ref{thm:metastable} follows from the fact that the map
$$ \pi_t (X) \rightarrow \pi_t(P_2(X)) $$
is an isomorphism for $t \le 3 \cdot \mr{conn}(X)$.  The homotopy groups $\pi_*(P_2(X))$ are classically known as the \emph{metastable homotopy groups of $X$}.  In the case where $X = S^n$, these maps take the form
$$ QS^n \xrightarrow{JH} Q P_n^\infty \simeq \Omega^\infty L(1)_n. $$ 

Suppose that $0 \ne \alpha \in \pi_t^s$ with $t \ge 1$ and take $n = 1$ (so that $L(1)_n = L(1) = \Sigma^\infty \RR P^\infty$).  The Kahn-Priddy theorem implies that $JH(\alpha) \ne 0$.  We define the \emph{stable Hopf invariant} 
$$ SHI(\alpha) \subset \pi_t(\ul{S}^m) $$
to be the coset of elements in the $E^1$-term of the AHSS
$$ E^1_{m,t}(L(1)) = \begin{cases}
\pi_t(\ul{S}^m), & m \ge 1, \\
0, & \text{else}
\end{cases}
 \Rightarrow \pi_t L(1) $$
which detect $JH(\alpha)$.  For $\beta \in SHI(\alpha) \subset \pi_t(\ul{S}^m)$ as above, we will sometimes write
$$ \beta[m] \in SHI(\alpha) $$
to emphasize the cell which detects $JH(\alpha)$.

The Snaith splitting gives restrictions of the James-Hopf map
$$ JH_m: \Omega^{m+1} S^{m+1} \rightarrow \Omega^\infty L(1)^m. $$
Kuhn \cite{KuhnJH} showed that these are compatible, in the sense that there is a map of fiber sequences
$$
\xymatrix{
\Omega^m S^m \ar[r]^E \ar[d]_{JH_{m-1}} &
\Omega^{m+1} S^{m+1} \ar[r]^H \ar[d]_{JH_m} & 
\Omega^{m+1} S^{2m+1} \ar[d]^{E^\infty} 
\\
\Omega^\infty L(1)^{m-1} \ar[r] &
\Omega^\infty L(1)^{m} \ar[r] &
QS^m  
}
$$
Thus there is a well-known map of spectral sequences
\begin{equation}\label{eq:Mahowald}
\xymatrix{
E^1_{m,t} = \pi_{t+m+1}(S^{2m+1}) \ar[d]_{E^\infty} \ar@{=>}[r] & 
\pi_t(QS^0) \ar[d]^{JH} 
\\
E^1_{m,t}(L(1)) = \pi_t(\ul{S}^m) \ar@{=>}[r] &
\pi_t L(1)
}
\end{equation}
from the EHPSS to the AHSS for $L(1)$.  We deduce the following lemma.

\begin{lem}\label{lem:SHI}
Suppose that $t \ge 1$ and $0 \ne \alpha \in \pi^s_{t}$ has Hopf invariant $\beta \in \pi_{t+m+1}(S^{2m+1})$.  Then one of the following two possibilities occurs.
\begin{enumerate}
\item $\beta$ is stable, and
$$ E^\infty \beta \in SHI(\alpha) \subset \pi_t(\ul{S}^m). $$

\item $\beta$ is unstable, and
$$ SHI(\alpha) \subset \pi_{t}(\ul{S}^{m'}), \quad \text{for $m' < m$.} $$ 
\end{enumerate}
\end{lem}

\section{Goodwillie $d_1$ differentials and stable Hopf invariants}\label{sec:GSSd1}

In this section we will observe that the stable Hopf invariant may be used to compute $d_1$-differentials in the GSS.  
The reader is reminded that the TGSS is set up so that GSS $d_1$-differentials correspond to TGSS differentials
$$ d^{S^n}: E^*_{*,J}(S^n) \rightarrow E^*_{*,J'}(S^n) $$
for $\abs{J'} = \abs{J}+1$.

The following theorem gives a complete computation of the GSS $d_1$-differentials emanating from the $0$-line of the GSS.

\begin{thm}\label{thm:GSSd1}
Suppose that $t \ge 1$ and we have a nontrivial element of the $E^1$-term of the TGSS for $S^n$:
$$ \alpha \in \pi_t(\ul{S}^n) = E^1_{t,\emptyset}(S^n) $$
with $\beta[m] \in SHI(\alpha)$.
If $m \ge n$, there is a non-trivial TGSS differential
$$ d^{S^n}(\alpha) = \beta[m]. $$
Otherwise, $\alpha$ is in the kernel of all of the differentials of the form 
$$ d^{S^n}: E^*_{*,\emptyset}(S^n) \rightarrow E^*_{*, [j]}(S^n) $$
for $j \ge n$.
\end{thm}

\begin{proof}
Theorem~\ref{thm:metastable} implies that in the TGSS for $S^1$ there is a non-trivial differential
$$ d^{S^1}(\alpha) = \beta[m]. $$
Consider the map of TGSS's induced by the suspension (Lemma~\ref{lem:ETGSS})
$$ E_*^{n-1}: \{ E^*_{*,*}(S^1) \} \rightarrow \{ E^*_{*+n-1,*}(S^n) \}. $$

Suppose that $m \ge n$.  Then there is a differential
\begin{equation}\label{eq:GSSd1}
 d^{S^n}(\alpha) = \beta[m] 
\end{equation}
in the TGSS for $S^n$.  Suppose that $\beta[m]$ is the target of a shorter $d^{S^n}$-differential.  The only possibility is that $\beta[m]$ is the target of a differential in the AHSS for $L(1)_n$.  Since the AHSS for $L(1)_n$ is a truncation of the AHSS for $L(1)_1$, this would imply that $\beta[m]$ is the target of a differential in the AHSS for $L(1)_1$.  This is impossible, as $\beta[m]$ detects $JH(\alpha) \ne 0$.  We conclude that the differential (\ref{eq:GSSd1}) must be non-trivial.

Suppose that $m < n$.  Then we conclude that $\alpha$ is in the kernel of the differential
$$ d^{S^n}: E^*_{*,\emptyset}(S^n) \rightarrow E^*_{*,[m]}(S^n). $$
In particular, $\alpha$ is in the kernel of the shorter differentials 
$$ d^{S^n}: E^*_{*,\emptyset}(S^n) \rightarrow E^*_{*, [j]}(S^n), \quad j \ge n. $$
\end{proof}

\section{Higher Goodwillie differentials and unstable Hopf invariants}\label{sec:GSSdr}

In this section we will explain how $d_r$-differentials, for $r > 1$, emanating from the $0$-line of the GSS correspond to unstable Hopf invariants.  
The reader is reminded that the TGSS is set up so that GSS $d_r$-differentials correspond to TGSS differentials
$$ d^{S^n}: E^*_{*,J}(S^n) \rightarrow E^*_{*,J'}(S^n) $$
for $\abs{J'} = \abs{J}+r$.

Suppose that  we are given a TGSS element
$$ 0 \ne \alpha \in \pi_t(\ul{S}^n) = E^1_{t,\emptyset}(S^n) $$
and that $\alpha$ is born on the $(m+1)$-sphere with unstable Hopf invariant.  According to Lemma~\ref{lem:SHI}, this implies that $\beta[m'] \in SHI(\alpha)$ for $m' < m$.  Then, by Theorem~\ref{thm:GSSd1}, if $n \le m'$, there is a non-trivial TGSS differential
$$ d^{S^n}(\alpha) = \beta[m']. $$
If $m' < n \le m$, then the element $\alpha$ must support a non-trivial differential in the GSS, and by Theorem~\ref{thm:GSSd1}, the GSS differential must be a $d_r$-differential for $r > 1$.  The following theorem tells us that often this longer differential is given by the (unstable) Hopf invariant of $\alpha$, and the length of this differential is determined by the degree of instability of this Hopf invariant.

\begin{thm}\label{thm:GSSdr}
Suppose that $\alpha \in \pi_t^s$ is born on $S^{m+1}$.
Then in the TGSS for $S^m$ there is a differential
$$ d^{S^m} (\alpha) = \beta[J,m] $$
where 
$$ \beta[J,m] \in GHI(\alpha). $$
In the TGSS for $S^n$, with $n < m$, either the element $\alpha$ supports a shorter differential, or
$$ d^{S^n}(\alpha) = \beta[J,m]. $$
\end{thm}

\begin{proof}
Lemmas~\ref{lem:Pfilt} and \ref{lem:Efilt} allow us to apply Lemma~\ref{lem:GBT} to the TGSS's of the fiber sequence
$$ \Omega^2 S^{m+1} \xrightarrow{H} \Omega^2 S^{2m+1} \xrightarrow{P} S^m \xrightarrow{E} \Omega S^{m+1}. $$
By assumption, $\alpha$ is a permanent cycle in the TGSS for $S^{m+1}$.  By the proof of Lemma~\ref{lem:GBT}, we are in Case (5).  Thus there exists a lift $\td{\alpha} \in \pi_{t+m+1}(S^{m+1})$ of $\alpha$ so that 
$$
d^{S^m}(\alpha) = \beta[J, m] $$
and either $\beta[J]$ detects $H(\td{\alpha})$ or it is the target of a longer differential in the TGSS for $S^{2m+1}$.  However, any longer differentials in the spectral sequence have source in a zero group.  We conclude that 
$$ \beta[J,m] \in GHI(\alpha). $$
The second part of the theorem is deduced from Lemma~\ref{lem:ETGSS}.
\end{proof}

\section{Propagating differentials with the $P$ and $E$ maps}\label{sec:propdiffs}

The $P$ and $E$ maps allow us to propagate the $0$-line GSS $d_r$-differentials of Theorems~\ref{thm:GSSd1} and \ref{thm:GSSdr} to give a plethora of $d_r$-differentials throughout the GSS.

\begin{prop}\label{prop:EGSS}
Suppose that $\ell > 0$ and in the TGSS for $S^n$ there is a non-trivial differential
$$ d^{S^n} (\alpha[J]) = \beta[J']. $$
If $n < m$ then one of the following possibilities occurs in the TGSS for $S^m$.
\begin{enumerate}
\item Either $e(J') < m$  or $\beta[J']$ is the target of a non-trivial differential 
$$ d^{S^m}(\gamma[I]) = \beta[J'] $$ 
for $I < J$, $\abs{I} \le \abs{J}-1$ and $\alpha[J]$ is in the kernel of the differential 
$$ d^{S^m}: E^*_{*,J}(S^m) \rightarrow E^*_{*,J'}(S^m). $$
\item There is a non-trivial differential
$$ d^{S^m} (\alpha[J]) = \beta[J']. $$
\end{enumerate}
\end{prop}

\begin{proof}
The proposition follows immediately from Lemma~\ref{lem:ETGSS}.  The fact that the length of the sequence $I$ in Case (1) must be less than $\abs{J'}$ follows from the fact that if $\abs{I} = \abs{J'}$ then there is a differential
$$ d^{S^n}(\gamma[I]) = \beta[J'] $$ 
in the TGSS for $S^n$ (see Remark~\ref{rmk:TAHSS}).  Since $I < J$, this would violate our assumption that the differential 
$$ d^{S^n} (\alpha[J]) = \beta[J']. $$
was non-trivial.
\end{proof}

\begin{prop}\label{prop:-EGSS}
Suppose that $\ell > 0$ and in the TGSS for $S^n$ there is a non-trivial differential
$$ d^{S^n} (\alpha[J]) = \beta[J']. $$
If $m < n$ then one of the following possibilities occurs in the TGSS for $S^m$.
\begin{enumerate}
\item The element $\alpha[J]$ supports a shorter non-trivial $d^{S^m}$ differential
$$ d^{S^m}(\alpha[J]) = \gamma[I] $$
for $I > J'$.
\item There is a non-trivial differential
$$ d^{S^m} (\alpha[J]) = \beta[J']. $$
\end{enumerate}
\end{prop}

\begin{proof}
This proposition follows immediately from Lemma~\ref{lem:ETGSS}.
\end{proof}

\begin{prop}\label{prop:PGSS}
Suppose that $\ell > 0$ and in the TGSS for $S^{2n+1}$ there is a non-trivial differential
$$ d^{S^{2n+1}} (\alpha[J]) = \beta[J']. $$
Then one of the following possibilities occurs in the TGSS for $S^{n}$.
\begin{enumerate}
\item The element $\beta[J',n]$ is the target of a non-trivial differential 
$$ d^{S^n} (\gamma[I]) \rightarrow \beta[J',n] $$
for $I > [J,n]$, $\abs{I} < \abs{J'}+1$
and $\alpha[J,n]$ is in the kernel of the differential
$$ d^{S^n}: E^*_{*,[J,n]} \rightarrow E^*_{*, [J',n]}. $$
\item There is a non-trivial differential
$$ d^{S^n} (\alpha[J, n]) = \beta[J', n]. $$
\end{enumerate}
\end{prop}

\begin{proof}
The proposition follows immediately from Lemma~\ref{lem:PTGSS}.  The fact that the sequence $I$ in Case (1) satisfies $\abs{I} < \abs{J'}+1$ follows from the fact that otherwise this differential comes from an differential in the TGSS for $S^{2n+1}$ (see Remark~\ref{rmk:TAHSS}), and this would violate our assumption that the differential 
$$ d^{S^{2n+1}} (\alpha[J]) = \beta[J']. $$
was non-trivial. 
\end{proof}

Combining Propositions~\ref{prop:-EGSS} and \ref{prop:PGSS} gives the following corollary.

\begin{cor}\label{cor:propdiffs}
Suppose that $m \le 2n+1$, $m' \le n$,  and in the TGSS for $S^{m}$ there is a non-trivial differential
$$ d^{S^m} (\alpha[J]) = \beta[J'] $$
for $\abs{J'} > \abs{J}$.
Then in the TGSS for $S^{m'}$, one of the following occurs.
\begin{enumerate}
\item The element $\beta[J']$ is the target of a $d^{S^{2n+1}}$ differential
$$ d^{S^{2n+1}}(\gamma[I]) = \beta[J'] $$
for $I > J$ and $\abs{I} < \abs{J'}$.
\item The element $\alpha[J, n]$ supports a non-trivial $d^{S^{m'}}$ differential 
$$ d^{S^{m'}}(\alpha[J,n]) = \gamma'[I'] $$
for $I' < [J',n]$.
\item There is a differential
$$ d^{S^{m'}}(\alpha[J, n]) = \beta[J',n]. $$
\end{enumerate} 
\end{cor}

The significance of Corollary~\ref{cor:propdiffs} is that differentials emanating from the $k$-line of the GSS can be propagated to differentials emanating from the $(k+1)$-line in the GSS.  As Theorems~\ref{thm:GSSd1} and \ref{thm:GSSdr} give us many differentials emanating from the $0$-line, these ``Hopf invariant'' differentials populate the entire GSS through the mechanism of Corollary~\ref{cor:propdiffs}.

\section{Goodwillie-Whitehead conjecture}\label{sec:GWC}

Consider the GSS for $S^1$:
$$ E_1^{k,t}(S^1) \cong \pi_{t+k-1}(L(k)) \Rightarrow \pi_t(S^1). $$
The appearance of a spectral sequence with $E_1$-term given by the homotopy of the $L(k)$-spectra, converging to $\ZZ_{(2)}$, bears a striking resemblance to the the $2$-primary Whitehead conjecture, as originally proved by Kuhn \cite{KuhnWhitehead}.

In his proof of the $2$-primary Whitehead conjecture, Kuhn formed a Kahn-Priddy sequence
$$ S^1 \leftrightarrows \Omega^\infty \Sigma L(0) 
\begin{array}{c} \partial_1 \\ \leftrightarrows \\ \delta_0 \end{array}
\Omega^\infty \Sigma L(1) 
\begin{array}{c} \partial_2 \\ \leftrightarrows \\ \delta_1 \end{array}
\Omega^\infty \Sigma L(2) 
\begin{array}{c} \partial_3 \\ \leftrightarrows \\ \delta_2 \end{array}
\cdots
$$
where the maps $\partial_k$ are the infinite loop space maps induced by the composites
$$ \Sigma L(k) \simeq \Sigma^{-k} \mr{Sp}_{2^k}(S)/\mr{Sp}_{2^{k-1}}(S) \xrightarrow{\partial} \Sigma^{-k+1} \mr{Sp}_{2^{k-1}}(S)/\mr{Sp}_{2^{k-2}}(S) \simeq \Sigma L(k-1)  $$
and the maps $\delta_k$ are given by the composites
$$ \Omega^\infty \Sigma L(k) \rightarrow Q (S^1)^{\wedge 2^k}_{h\Sigma_2^{\wr k}} \xrightarrow{JH} Q (S^1)^{\wedge 2^{k+1}}_{h\Sigma_2^{\wr (k+1)}} \rightarrow \Omega^\infty \Sigma L(k+1). $$
Kuhn showed that the sum
$$ \partial_{k+1}\delta_k+\delta_{k-1}\partial_k $$
is a self-equivalence of $L(k)$.  Paired with the observation that $\partial^2 \simeq 0$, this establishes that the differentials $\partial_k$ give an acyclic chain comlex, where $d_k$ serves as the contracting homotopy.

The $k$-invariants of the Goodwillie tower for $S^n$ give maps
$$ \delta^{S^n}_k: \Omega^\infty \Sigma^{n-k} L(k)_n \rightarrow \Omega^\infty \Sigma^{n-k} L(k+1)_n. $$
The following calculus versions of the Whitehead conjecture have been postulated and studied by Arone, Dwyer, Kuhn, Lesh, and Mahowald (see \cite{AroneLesh}, \cite{AroneDwyerLesh}).

\begin{conj}[Goodwillie-Whitehead conjecture: strong form]\label{conj:SGWC}
We have
\begin{equation}\label{eq:SGWC}
 \Omega^k \delta_k = \delta^{S^1}_k. 
\end{equation}
\end{conj}

\begin{rmk}
Actually, N.~Kuhn has proposed an even stronger conjecture.  Both $\delta_k$ and $\delta^{S^1}_k$ extend to natural transformations between functors from vector spaces to spaces (see the proof of Lemma~\ref{lem:EHPfiber}).  Kuhn conjectures that (\ref{eq:SGWC}) holds on the level of these natural transformations.  Arone, Dwyer, and Lesh have proven that $\delta^{S^V}_k$ admits a $k$-fold delooping as a natural transformation.
\end{rmk}

Theorem~\ref{thm:metastable} affirms the $k = 0$ version of this conjecture.  The compatibility of the Goodwillie tower with $E$ and $P$ given by Corollary~\ref{cor:EHPfiber} implies the following diagrams commute.
$$
\xymatrix{
\Omega^\infty \Sigma^{n-k} L(k)_n \ar[r]^{\delta^{S^n}_k} \ar[d]_E & 
\Omega^\infty \Sigma^{n-k} L(k+1)_n \ar[d]^E 
\\
\Omega^\infty \Sigma^{n-k} L(k)_{n+1} \ar[r]_{\delta^{S^{n+1}}_k} &
\Omega^\infty \Sigma^{n-k} L(k+1)_{n+1}
}
$$
$$
\xymatrix{
\Omega^\infty \Sigma^{2n-k} L(k-1)_{2n+1} \ar[r]^{\delta^{S^{2n+1}}_{k-1}} \ar[d]_P & 
\Omega^\infty \Sigma^{2n-k} L(k)_{2n+1} \ar[d]^P 
\\
\Omega^\infty \Sigma^{n-k} L(k)_{n} \ar[r]_{\delta^{S^{n}}_k} &
\Omega^\infty \Sigma^{n-k} L(k+1)_{n}
}
$$
Together with Theorem~\ref{thm:GSSd1}, these imply the following slight strengthening of Corollary~\ref{cor:propdiffs} in the context of GSS $d_1$-differentials.

\begin{thm}\label{thm:higherd1}
Suppose that $\alpha \in \pi_*^s$ with stable Hopf invariant
$$ \beta[m] \in SHI(\alpha). $$
Suppose that $\alpha[j_1, \ldots, j_k]$ is a permanent cycle in the TAHSS for $L(k)_n$ and $m \ge 2j_1+1$.  Then in the TGSS for $S^n$, we have
$$ d^{S^n}(\alpha[j_1, \ldots, j_k]) = \beta[m, j_1, \ldots, j_k]. $$
\end{thm}

Theorem~\ref{thm:GSSd1} can be summarized by the slogan: 
$$ d^{GSS}_1(\alpha[J]) = SHI(\alpha)[J] + \text{lower terms}. $$
This is consistent with Conjecture~\ref{conj:SGWC}.

Conjecture~\ref{conj:SGWC} implies the following weaker conjecture.

\begin{conj}[Goodwillie-Whitehead conjecture: weak form]\label{conj:WGWC}
The GSS for $S^1$ collapses at the $E_2$-page.
\end{conj}

The low dimensional calculations of the GSS for $S^1$ in Chapter~\ref{sec:calculations} are consistent with Conjecture~\ref{conj:WGWC}.
 
\section{Exotic Goodwillie differentials}\label{sec:exotic}

We shall see in Chapter~\ref{sec:calculations} that 
Theorems~\ref{thm:GSSd1}, \ref{thm:GSSdr}, \ref{thm:higherd1} and Corollary~\ref{cor:propdiffs} account for most, but not all, of the GSS differentials in our sample computations.  We think of these differentials that are given by stable Hopf invariants and Hopf invariants as ``typical''.

The remaining differentials we regard as ``exotic''.  All of the exotic differentials which appear in our low dimensional calculations are instances of two different phenomena:

\begin{itemize}
\item \ul{Geometric Boundary Effect:} such exotic differentials can be deduced from TAHSS differentials and non-exotic GSS differentials when Lemma~\ref{lem:GBT} is applied to the EHP sequence.
\vspace{1em}

\item \ul{Bizarre:} these differentials show up without explanation in the GSS for $S^1$ --- the only reason the author knows they are present is because he already knows $\pi_*(S^1)$, and these differentials are necessary to get the right answer.  These differentials in the GSS for $S^1$ induce bizarre differentials throughout the GSS for $S^n$ for various $n$ through the mechanisms of  Proposition~\ref{prop:EGSS} and Corollary~\ref{cor:propdiffs}. 
\end{itemize}

The geometric boundary effect differentials are a consequence of the following theorem.  

\begin{thm}\label{thm:GBE}
Suppose that there is a non-trivial TGSS differential
$$ d^{S^n}(\alpha[J]) = \beta[J', n] $$
and there are non-trivial differentials
\begin{align*}
d^{S^{2n+1}}(\alpha[J']) &  = \gamma[T'], \\
d^{S^{n}}(\delta[T]) & = -\gamma[T', n].
\end{align*}
Assume the following technical condition:

\begin{tabular}{c p{.7\textwidth}}
$(\ast)$ & 
For all non-trivial differentials
$$ d^{S^{2n+1}} (\tau[I]) = \tau'[I'] $$
with 
$$ T' < I' < I < J' \quad \text{and} \quad \abs{I} < \abs{I}' $$
the induced differential under the $P$ map
$$ d^{S^{n}} (\tau[I,n]) = \tau'[I',n] $$
is non-trivial.
\end{tabular}

Then there is a differential
$$ d^{S^{n+1}}(\alpha[J]) = \delta[T]. $$
\end{thm}

\begin{proof}
Apply Lemma~\ref{lem:GBT} to the sequence 
$$ \Omega^2 S^{n+1} \xrightarrow{H} \Omega^2 S^{2n+1} \xrightarrow{P} S^n \xrightarrow{E} \Omega S^{n+1}. $$
Since
$$ E_* \alpha[J', n] = 0 $$
we are not in Case (1).  Since $E_*$ and $P_*$ induce short exact sequences on the level of GSS $E^1$-terms, $H_* = 0$, and therefore we cannot be in Case (2).  Since $\alpha[J']$ is assumed to support a non-trivial differential, we are in Case (3).  Therefore, Lemma~\ref{lem:GBT}(3) supplies us with elements $x''$ and $y'' = \delta[T]$ so that
\begin{align*}
d^{S^{2n+1}}(\beta[J']) & = x'', \\
d^{S^n}(y'') & = P_* x'', \\
d^{S^{n+1}} (\alpha[J]) & = E_* y''.
\end{align*} 
The theorem would be proven if we knew that $x'' = \gamma[T']$.
Suppose not.  Then $x'' = \tau'[I']$ with $I' > T'$ and $[I',n] < T$, and $\tau'[I']$ must be the target of a non-trivial differential 
$$ d^{S^{2n+1}} (\tau[I]) = \tau'[I'] $$
with $I' < I < J'$.  The non-triviality condition on $x''$ in Lemma~\ref{lem:GBT}(3) implies $[I,n] > T$.  Thus the induced differential under the $P$ map
$$ d^{S^n}(\tau[I,n]) = \tau[I',n] $$
is trivial, since $\tau[I',n]$ was already killed by the shorter $d^{S^n}$ differential supported by $\delta[T]$.  
We must therefore have $\abs{I} < \abs{I'}$, since TGSS differentials $d^{S^{2n+1}}_\alpha$ with $\alpha \in \mc{G}(\omega^\omega)$ are mapped \emph{faithfully} to TGSS $d^{S^n}$ differentials under the $P$ map (see Remark~\ref{rmk:TAHSS}). 
But then we are in violation of condition $(*)$.  We conclude that $x'' = \gamma[T']$, as desired.
\end{proof}

\begin{ex}
We give an example of an exotic GSS differential which comes from the geometric boundary effect.
The element $\epsilon\eta[4]$ is a permanent cycle in the TAHSS for $L(1)$.
We have
$$ \epsilon[1] \in SHI(\epsilon\eta) $$
but this does not help us compute $d^{S^2}(\epsilon\eta[4])$ since the sequence $(1,4)$ is not CU. 
Instead we apply Theorem~\ref{thm:GBE}.  In the TGSS for $S^1$ there is a differential
$$ d^{S^1}(\epsilon \eta[4]) = \alpha_{6/3}[1] $$
coming from a differential in the AHSS for $L(1)$.
In the TGSS for $S^3$ there is a differential
$$ d^{S^3}(\alpha_{6/3}) = 8\sigma[4] $$
coming from the fact that $8\sigma \in HI(\alpha_{6/3})$.  In the TGSS for $S^1$ there is a differential
$$ d^{S^1}(\eta^3[8,2]) = 8\sigma[4,1] $$
coming from a differential in the TAHSS for $L(2)$.  Condition $(\ast)$ of Theorem~\ref{thm:GBE} is verified by checking (see Table~\ref{tab:L(1)AHSS}) that there are no elements of $\pi_{11}(L(1)_3)$ detected on the $3$-cell which could support an intervening differential.
We deduce that there is an exotic TGSS differential
$$ d^{S^2}(\epsilon \eta[4]) = \eta^3[8,2]. $$
Tables~\ref{tab:S2GSS}--\ref{tab:S6GSS} show several exotic GSS $d_1$ differentials which arise from the geometric boundary effect (in these tables, these differentials are labeled with a $(\ast)$).
\end{ex}

\begin{ex}
The computations of Chapter~\ref{sec:calculations} show that there are precisely four bizarre differentials in the GSS for $S^1$ through the $20$-stem (the first three of these appear in Table~\ref{tab:S1GSS}, denoted with dotted arrows, the last is just beyond the range displayed in the table):
\begin{align*}
d^{S^1}(\theta_3[4]) & = 1[15,3], \\
d^{S^1}(\theta_3[4,1]) & = 1[15,3,1], \\
d^{S^1}(\nu \kappa[2]) & = \kappa[4,1], \\
d^{S^1}(\theta_3[8]) & = 1[15,7].
\end{align*}
(Note that the second of these differentials follows from the first by Corollary~\ref{cor:propdiffs}.)  Proposition~\ref{prop:EGSS} can be used to produce bizarre differentials
\begin{align*}
d^{S^n}(\theta_3[4]) & = 1[15,3], \\
d^{S^n}(\theta_3[8]) & = 1[15,7].
\end{align*}
in the GSS for $S^n$ for various $n$.
\end{ex}

We shall see in Chapter~\ref{sec:EHPdiff} that exotic GSS differentials can induce exotic EHPSS differentials.

\chapter{EHPSS differentials}\label{sec:EHPdiff}

In this chapter we explain how to lift differentials from the TGSS to produce EHPSS differentials.  In Section~\ref{sec:TEHPname}, we discuss two naming conventions for elements of the $E^1$-term of the EHPSS: one coming from the TEHPSS, the other coming from the notion of lineage, and explain the relationship between these two naming conventions.  Section~\ref{sec:Hmap} proves some theorems which allow for the computation of the $H$ map by means of the TGSS.  These theorems are then used in Section~\ref{sec:TEHPdiffs} to determine TEHPSS differentials from TGSS differentials.  This method is remarkably robust, and accounts for all but one differential in the EHPSS through the Toda range.  This rogue differential is discussed in Section~\ref{sec:bad}.

\section{EHPSS naming conventions}\label{sec:TEHPname}

As explained in Section~\ref{sec:TEHPSS}, the TEHPSS allows us to refer to elements in the $E^1$-term of the EHPSS with the notation
$$ x[N,n] $$
where $x$ is an element of the stable stems, and $N = (n_1, n_2, \ldots,
n_s)$ is a CU sequence with $e(N) \ge 2n+1$.  
The construction of the TEHPSS gives the following interpretation of this notation.
One can use the TAHSS
followed by the GSS to compute the unstable homotopy groups that give the
$E^1$-term of the EHPSS.
$$ 
\xymatrix@C+1.5em{
\pi_{t-n+s}(\ul{S}^{\norm{N}})
\ar@{=>}[r]_-{TAHSS} &
\pi_{t-n+s}(L(s)_{2n+1}) \ar@{=>}[r]_-{GSS} &
\pi_{t+n+1}(S^{2n+1}) 
}$$
Then an element 
$$ \alpha \in \pi_{t-n+1}(S^{2n+1})$$ 
in the $E^1$-term of the EHPSS corresponds to the TEHPSS element 
$$ x[N,n] \in \pi_{t-n+s}(\ul{S}^{\norm{N}}) $$ 
if $x[N]$ detects $\alpha$ in the TGSS for $S^{2n+1}$.

Recall that the EHPSS may be truncated to give the $S^k$-EHPSS:
$$ E^1_{m,t}(k) = 
\begin{cases}
\pi_{t+m+1}(S^{2m+1}), & 0 \le m < k,  \\
0, & \text{else.}
\end{cases}
\Rightarrow
\pi_{t+k}(S^k)
$$
The notions of lineage in Section~\ref{sec:DOI} give rise to a different naming convention for elements of the EHPSS $E^1$-term which only involves the EHPSS and the $S^k$-EHPSS for various $k$.  In this alternative naming convention we name elements of the EHPSS $E^1$-term with the notation
$$ x\bra{N}[n] \in \pi_{t+n+1}(S^{2n+1}) $$
to indicate an element of $\pi_{t+n+1}(S^{2n+1})$ with lineage $x\bra{N}$. 
The lineage of such an element can be traced using the EHPSS.  If $\abs{N} = 0$, then $x\bra{n}$ has a nontrivial
image in the stable stems, and we have, under the stabilization map
\begin{gather*}
\pi_{k+n+1}(S^{2n+1}) \xrightarrow{E^\infty} \pi^s_{k-n} \\
x[n] \mapsto x
\end{gather*}
Otherwise, $\abs{N} > 0$ and $x\bra{N}[n]$ is unstable, and is detected by an element of the
EHPSS which is the target of a differential.  Then there is a zig-zag
$$
\xymatrix@C+.5em{
x\bra{n_1, \ldots, n_s}[n] &
y_1\bra{J_1}[m_1] \ar@{=>}[l]_-{S^{2n+1}}^-{\mit{EHPSS}}
\\
& x\bra{n_1, \ldots, n_{s-1}}[n_s] \ar[u]_{d^{\mit{EHP}}} &
y_2\bra{J_2}[m_2] \ar@{=>}[l]_-{S^{2n_s+1}}^-{\mit{EHPSS}} &
\\
&& \hole \ar[u]_{d^{\mit{EHP}}} &
\\
& 
&& \hole \ar@{}[ul]|{\ddots} & y_s\bra{J_s} 
\ar@{=>}[l]_-{S^{2n_2+1}}^-{\mit{EHPSS}} &
\\
& 
&&& x[n_1] \ar[u]_{d^{\mit{EHP}}}
} $$
where $x[n_1]$ is stable.  

\begin{rmk}\label{rmk:TEHPname}
Theorem~\ref{thm:TGSS} leads one to expect a correspondence between EHPSS elements
\begin{equation}\label{eq:TEHPname}
 x\bra{N}[n] \leftrightarrow x[N,n].
\end{equation}
In general, the correspondence (\ref{eq:TEHPname}) is more of a slogan than a theorem.  Nevertheless, there is only one example in the Toda range of an an element (modulo higher Goodwillie filtration) whose lineage differs from the name of the element which detects it in the TGSS, in the 19-stem (see Example~\ref{ex:DOI}).  This produces in the EHPSS $E^1$-term a violation of (\ref{eq:TEHPname})
$$ \eta^2\bra{13,6}[2] \leftrightarrow \alpha_{8/5}[5,2]. $$
(This discrepancy lies in the $21$-stem of the EHPSS, and hence lies outside the range of our EHPSS calculations.)  The calculations of Chapter~\ref{sec:calculations} demonstrate that there exists a set of generators of the $E^1$-term of the EHPSS through the $20$-stem for which (\ref{eq:TEHPname}) holds.
Using Proposition~\ref{prop:metastableGSS}, we see that (\ref{eq:TEHPname}) always holds for elements of the EHPSS $E^1$-term $\pi_{t+2n+1}(S^{2n+1})$ for $t \le 6n+1$ (metastable range).
\end{rmk}

\begin{ex}
In $k=15$ there is an element in the EHPSS $E^1$-term
$$\nu\bra{9,4}[1] \in E^1_{1,15} = \pi_{17}(S^3). $$  
In the TGSS for $S^3$ this element is detected by
the element 
$$\nu[9,4] \in \pi_{16}(\ul{S}^{13}).$$ 
On the other hand, there is a
there is the following zig-zag.
$$
\xymatrix@C+1.5em{
\nu[9,4,1] &
\sigma\eta[5,2] \ar@{=>}[l]_-{S^3}^-{\mit{EHPSS}} 
\\
& \nu[9,4] \ar[u]_{d^{\mit{EHP}}} &
\nu^2[5] \ar@{=>}[l]_-{S^9}^-{\mit{EHPSS}}
\\
&& \nu[9] \ar[u]_{d^{\mit{EHP}}}
} $$
Thus $\nu\bra{9,4}[1] \leftrightarrow \nu[9,4,1]$.
\end{ex}

\section{Using the TGSS to compute the $H$ map}\label{sec:Hmap}

As demonstrated by Lemmas~\ref{lem:ETGSS} and \ref{lem:PTGSS}, the $E$ and $P$ maps are typically easy to understand in the TGSS.  In fact, we have seen that the maps of spectral sequences $E_*$ and $P_*$ induces short exact sequences of TGSS $E^1$-terms.  Therefore, the $H$ map \emph{always} is zero on the level of TGSS $E^1$-terms.  The following lemma explains how to use the geometric boundary theorem to compute the $H$ map in terms of TGSS elements.

\begin{thm}\label{thm:Hmap}
Suppose that $\alpha[J] \in E^1_{t,J}(S^{n+1})$ is a non-trivial permanent cycle in the TGSS for $S^{n+1}$.
Suppose there is a TGSS differential
$$  d^{S^n}(\alpha[J]) = \beta[J',n]. $$
Then there exists an element $x \in \pi_t(S^{n+1})$ detected by $\alpha[J]$ so that one of the following two alternatives holds.
\begin{enumerate}
\item The element $\beta[J']$ is the target of a differential in the TGSS for $S^{2n+1}$, and $H(x)$ is either zero or detected in the TGSS by $\gamma[I]$ for $I < J'$.
\item The element $\beta[J']$ detects $H(x)$ in the TGSS for $S^{2n+1}$.
\end{enumerate}
\end{thm}

\begin{proof}
The theorem follows from applying Lemma~\ref{lem:GBT} to the fiber sequence
$$ \Omega^2 S^{n+1} \xrightarrow{H} \Omega^2 S^{2n+1} \xrightarrow{P} S^n \xrightarrow{E} \Omega S^{n+1}. $$
Since $\alpha[J]$ is assumed to be a permanent cycle in the TGSS for $S^{n+1}$, we are in Case (5).
\end{proof}

The above theorem is good when we want to compute $H(x)$ for some element $x$ detected by $\alpha[J]$ in the TGSS, but is insufficient if we actually need to compute $H(x)$ for a \emph{specific} $x$.  We therefore also present the following variant of the above theorem.

\begin{thm}\label{thm:Hmap2}
Suppose that $x \in \pi_{t+1}(S^{m+1})$ is detected by $\alpha[J,n]$ in the TGSS for $S^{m+1}$, and suppose that $\beta[J']$ detects $H(x)$ in the TGSS for $S^{2m+1}$.  Then one of the following two alternatives holds.
\begin{enumerate}
\item There is a non-trivial TGSS differential
$$ d^{S^m}(\alpha[J,n]) = \beta[J',m]. $$
\item  There is a TGSS differential
$$ d^{S^m}(\gamma[I,\ell]) = \beta[J',m] $$
for $[I,\ell] < [J,n]$, and $\alpha[J,n]$ is in the kernel of the TGSS differential
$$ d^{S^{m}}: E^*_{*, [J,n]} \rightarrow E^*_{*,[J',m]}. $$
\end{enumerate}
\end{thm}

\begin{proof}
The theorem follows from applying Lemma~\ref{lem:GBT2} to the fiber sequence
$$ \Omega^2 S^{m+1} \xrightarrow{H} \Omega^2 S^{2m+1} \xrightarrow{P} S^m \xrightarrow{E} S^{m+1}. $$
\end{proof}

\section{TEHPSS differentials from TGSS differentials}\label{sec:TEHPdiffs}

Theorem~\ref{thm:Hmap2} allows us to deduce many TEHPSS differentials from TGSS differentials, as we will now explain.

Differentials in the EHPSS are computed as follows.  Suppose we are given an EHPSS $E^1$-element 
$$ \alpha \in \pi_{t+n+1}(S^{2n+1}) = E^1_{n,t} $$
and we wish to compute its EHPSS differential.  Take $P(\alpha) \in \pi_{t+n-1}(S^n)$ and desuspend it to its sphere of origin
$$ \gamma \in \pi_{t+m}(S^{m+1}), \quad E^{n-m-1}\gamma = P(\alpha). $$
Then
$$ d^{\mit{EHP}}_{n-m}(\alpha) = H(\gamma) \in \pi_{t+m}(S^{2m+1}) $$

Differentials in the TEHPSS take the form
\begin{equation}\label{eq:TEHPdiffs}
d^{\mit{TEHP}}: E^*_{*,[J,n]} \rightarrow E^*_{*,[J',m]} 
\end{equation}
for $m < n$, or $n = m$ and $J' < J$.
Differentials (\ref{eq:TEHPdiffs}) the case of $n = m$ all arise from TGSS differentials: the presence of a TEHPSS differential 
$$ d^{\mit{TEHP}}(\alpha[J,n]) = \beta[J',n] $$
is equivalent to the presence of a TGSS differential
$$ d^{S^{2n+1}}(\alpha[J]) = \beta[J']. $$

Differentials (\ref{eq:TEHPdiffs}) for $m < n$ are essentially EHPSS differentials: the presence of a TEHPSS differential
$$ d^{\mit{TEHP}}(\alpha[J,n]) = \beta[J',m], \quad m < n $$
is equivalent to the presence of an EHPSS differential
$$ d^{\mit{EHP}} (x) = y, $$
where $\alpha[J]$ detects $x$ in the TGSS for $S^{2n+1}$ and $\beta[J']$ detects $y$ in the TGSS for $S^{2m+1}$.

\begin{thm}\label{thm:TEHPdiffs}
Suppose the following conditions hold. 
\begin{enumerate}
\item The element $\alpha[J]$ is permanent cycle in the TGSS for $S^{2n+1}$.
\item For $m < n$, there is a TGSS differential
$$ d^{S^m}(\alpha[J,n]) = \beta[J',m]. $$
\item The element $\alpha[J,n]$ is a permanent cycle in the TGSS for $S^{m+1}$. 
\item For $m \le \ell < n$ there do not exist any non-trivial TGSS differentials
$$ d^{S^{\ell}}(\gamma'[I', \ell']) = \gamma[I,\ell] $$
for $\ell < \ell'$ and $[J',m]< [I,\ell] < [I',\ell'] < [J,n]$ such that $\gamma'[I',\ell']$ is a not a permanent cycle in the TGSS for $S^{\ell+1}$.
\end{enumerate}
Then there is a TEHPSS differential
$$ d^{\mit{TEHP}} \alpha[J,n] = \beta[J',m]. $$
\end{thm}

\begin{proof}
We prove the theorem by induction on the length of the differential in condition (2).  By (1), there exists an element $x \in \pi_*(S^{2n+1})$ which is detected by $\alpha[J]$ in the TGSS.  The element $P(x) \in \pi_*(S^n)$ is detected by $\alpha[J,n]$ in the TGSS for $S^n$.  By our inductive hypothesis (applied by allowing the ``$\beta$'' in (2) to be trivial) there no non-trivial TEHPSS differentials
$$ d^{\mit{TEHP}}(\alpha[J,n]) = \beta'[J'', m'] $$
for $m' > m$, and so there exists a desuspension $y$ of $P(x)$:
$$ y \in \pi_*(S^{m+1}) \xrightarrow{E^{n-m+1}} \pi_{*+n-m+1}(S^n) \ni P(x). $$
By condition (3), $y$ is detected by $\alpha[J,n]$ in the TGSS for $S^{m+1}$.
We now use Theorem~\ref{thm:Hmap2} to compute $H(y)$.  Suppose that we are in Case (2) of Theorem~\ref{thm:Hmap2}.  Then $H(y)$ is detected in the TGSS for $S^{2m+1}$ by $\gamma[I]$ for $I > J'$, and there is a non-trivial TGSS differential
$$ d^{S^m}(\gamma'[I',\ell]) = \gamma[I,m]. $$
for $[I',\ell] < [J,n]$.  Now, if $\ell = m$, then the differential pulls back under $P$ to give a differential
$$ d^{S^{2m+1}}(\gamma'[I']) = \gamma[I] $$
and this means $\gamma[I]$ does not detect anything in the TGSS for $S^{2m+1}$.  Otherwise, $\ell > m$, and (4) implies that $\gamma'[I', \ell]$ is a permanent cycle in the TGSS for $S^{m+1}$.  Then the inductive hypothesis implies that there is a differential
$$ d^{\mit{TEHP}}(\gamma'[I',\ell]) = \gamma[I,m]. $$
In particular, $\gamma[I,m]$ cannot detect a Hopf invariant.  We conclude that we must be in Case (1) of Theorem~\ref{thm:Hmap2}.  Then $H(y)$ is detected by $\beta[J,n]$, and there is a TEHPSS differential
$$ d^{\mit{TEHP}}: \alpha[J,n] = \beta[J',m] $$
or $\beta[J',m]$ is the target of a shorter TGSS differential, and $\alpha[J,n]$ is in the kernel of the TEHPSS differential
$$ d^{\mit{TEHP}}: E^*_{*,[J,n]} \rightarrow  E^*_{*, [J',m]}. $$
\end{proof}

\begin{rmk}
Condition (4) of Theorem~\ref{thm:TEHPdiffs} can be labor intensive to verify, but there is a shortcut to checking it.  Namely, if (\ref{eq:TEHPname}) holds for all of the relevant TEHPSS elements, then suppose that there is a TGSS differential
$$ d^{S^\ell}(\gamma'[I',\ell']) = \gamma[I,\ell] $$
where $\gamma'[I',\ell']$ is not a permanent cycle in the TGSS for $S^{\ell+1}$.  Then the differential cannot lift to a TEHPSS differential
$$ d^{\mit{TEHP}}(\gamma'[I',\ell']) = \gamma[I,\ell], $$
for if it did, then $\gamma[I,\ell]$ would be a non-trivial permanent cycle in the $S^{\ell+1}$-EHPSS, and converge to an element detected by $\gamma'[I',\ell']$ in the TGSS for $S^{\ell+1}$.  So the set of intervening TGSS differentials that violate condition (4) is a subset of the set of TGSS differentials as above, with $\ell < \ell'$ which do not lift to the TEHPSS. 
\end{rmk}

\begin{ex}
The most typical manner in which Theorem~\ref{thm:TEHPdiffs} is used is to ``lift'' differentials from the TAHSS's for the $L(k)$ spectra in a manner which generalizes the manner Mahowald and others lift differentials from the AHSS to the EHPSS (using (\ref{eq:Mahowald})).   
For example, there is a differential in the $S^1$-TGSS (see
Table~\ref{tab:L(2)AHSS}, $k=7$)
$$ d^{S^1}(1[5,2]) = \eta[4,1]. $$
Theorem~\ref{thm:TEHPdiffs} can be used to deduce the TEHPSS differential
$$ d^{\mit{TEHP}}(1[5,2]) = \eta[4,1]. $$
\end{ex}

\begin{ex}
Theorem~\ref{thm:TEHPdiffs} can also be used to obtain TEHPSS differentials from exotic TGSS differentials coming from the geometric boundary effect (see Section~\ref{sec:exotic}).  
Consider the geometric boundary effect TGSS differential
$$ d^{S^2}(\eta\epsilon[4]) = \eta^3[8,2]. $$
This differential arises from applying Theorem~\ref{thm:GBE} to the system of TGSS differentials:
$$
\xymatrix@C+1em{ 
\eta\epsilon[4] \ar@{|->}[d]_{d^{S^1}} &
\eta^3[8,2] \ar@{|->}[d]^{d^{S^1}} 
\\
\alpha_{6/3}[1]  &
8\sigma[4,1] 
\\
\alpha_{6/3} \ar@{|->}[r]_{d^{S^3}} \ar@{|->}[u]^{P_*} &
8\sigma[4] \ar@{|->}[u]_{P_*}
} $$
Theorem~\ref{thm:TEHPdiffs} gives a TEHPSS differential
$$ d^{\mit{TEHP}}(\eta\epsilon[4]) = \eta^3[8,2]. $$
\end{ex}

\begin{ex}
Theorem~\ref{thm:TEHPdiffs} even produces bizarre TEHPSS differentials from bizarre TGSS differentials (see Section~\ref{sec:exotic}).
Consider the bizarre differential
$$ d^{GSS}(\theta_3[4]) = 1[15,3] $$
in the TGSS for $S^1$.
Theorem~\ref{thm:TEHPdiffs} gives a corresponding 
TEHPSS differential
$$ d^{\mit{TEHP}}(\theta_3[4]) = 1[15,3]. $$
\end{ex}

\section{A bad differential}\label{sec:bad}

The calculations of Chapter~\ref{sec:calculations} demonstrate that through the $20$-stem, \emph{with one exception}, all of the differentials in the TEHPSS can be obtained by lifting TGSS differentials through the application of Theorem~\ref{thm:TEHPdiffs}.  We discuss this exception, which is the root of the counterexamples in Example~\ref{ex:DOI} and Remark~\ref{rmk:TEHPname}.  This ``rogue differential'' is given by the following lemma. 

\begin{lem}\label{lem:bad}
There is a non-trivial TEHPSS differential
$$ d^{\mit{TEHP}}(\eta^2[13,6]) = \eta\alpha_{8/5}[3] $$
\end{lem}

\begin{proof}
The differentials
\begin{align*}
d^{S^5}(\eta^2[13,6]) & = 4\nu[12,5], \\
d^{S^{11}}(4 \nu[12]) & = 0
\end{align*}
imply, using Theorem~\ref{thm:TEHPdiffs}, that $\eta^2[13,6]$ is in the kernel of the TEHPSS differentials
$$ d^{\mit{TEHP}}: E^*_{*,[6,2]} \rightarrow E^*_{*,[J,5]} $$
for all $J$.  This implies that if $\br{\eta^2[13]} \in \pi_{27}(S^{13})$ is the unique element which is detected by $\eta^2[13]$ in the TGSS for $S^{13}$, then $P(\br{\eta^2[13]})$ must desuspend to $S^4$.

Now, on the level of TGSS $E^1$-terms, we have
$$ P_*(\eta^2[13]) = \eta^2[13,6] \in E^1_{25,[13,6]}(S^6). $$
Consulting Table~\ref{tab:S6GSS}, we see that $\eta^2[13,6]$ is a non-trivial permanent cycle in the TGSS for $S^6$.  Therefore, we conclude that $P(\br{\eta^2[13]}) \ne 0$.  Therefore, there must exist a non-zero element $x \in \pi_{23}(S^4)$ whose double suspension is $P(\br{\eta^2[13]})$.  In the TGSS the element $x$ must be detected by an element with transfinite Goodwillie filtration greater than or equal to $\mu{[13,6]}$, but $x$ must have Goodwillie filtration greater than $1$ since $E^3(x) = 0$.  Consulting Table~\ref{tab:S4GSS}, there are two possibilities of TGSS elements which can detect $x$:
$$ \eta \alpha_{8/5}[4] \quad \text{and} \quad \alpha_{8/5}[5]. $$

Now, if $x$ were detected by $\eta \alpha_{8/5}[4]$, then since
$$ E_*(\eta \alpha_{8/5}[4]) = 0 \in E^1_{24,[4]}(S^5), $$
the element $E(x) \in \pi_{24}S^{5}$ must be detected by an element of transfinite Goodwillie filtration less than $\mu{[4]}$.  Consulting Table~\ref{tab:S5GSS}, we see that the only such elements in the TGSS for $S^5$ are stable.  Since $E(x)$ is not stable, we have arrived at a contradiction, and we deduce that $x$ must be detected by $\alpha_{8/5}[5]$ in the TGSS for $S^4$.

We now compute $H(x)$.  Using Theorem~\ref{thm:Hmap}, the TGSS differential
$$ d^{S^3}(\alpha_{8/5}[5]) = \eta \alpha_{8/5}[3] $$
allows us to conclude that there exists an element $x' \in \pi_{23}(S^4)$ detected by $\alpha_{8/5}[5]$ in the TGSS for $S^4$ such that $H(x)$ is detected by $\eta \alpha_{8/5}$ in the TGSS for $S^7$.  Tables~\ref{tab:S4GSS} and \ref{tab:S5GSS} reveal that $E(x) = E(x')$.  Thus $x'$ is also a desuspension of $P\br{\eta^2[13]}$.  We have therefore shown that there is an EHPSS differential
$$ d^{\mit{EHP}}(\br{\eta^2[13]}) = H(x'). $$
Since $\br{\eta^2[13]}$ is detected in the TGSS by $\eta^2[13]$, and $H(x')$ is detected in the TGSS by $\eta\alpha_{8/5}$, there is a TEHPSS differential
$$ d^{\mit{TEHP}}(\eta^2[13,6]) = \eta\alpha_{8/5}[3]. $$
\end{proof}

\chapter{Calculations in the $2$-primary Toda range}\label{sec:calculations}

In this chapter we illustrate the computational effectiveness of our methods by completely computing the TAHSS for $L(k)$ for $1 \le k \le 3$ (Tables~\ref{tab:L(1)AHSS}--\ref{tab:L(3)AHSS}), the TGSS for $S^n$ for $1 \le n \le 6$ (Tables~\ref{tab:S1GSS}--\ref{tab:S6GSS}), and the TEHPSS (Table~\ref{tab:EHPSS}), in the Toda range (i.e. up to and including the 19-stem).  

We summarize the methods and notations of our computations in the first several sections of this chapter.
Section~\ref{sec:AHSScalc} describes the inductive low dimensional computations of the TAHSS's for the $L(k)$ spectra.  This data will
be used to give the $E_1$-terms of the TGSS's, and to give the differentials
in the EHPSS.  The TAHSS computations feed into Sections~\ref{sec:Kuhncalc} and \ref{sec:GSScalc}, where we compute the TGSS for $S^n$, $1 \le n \le 6$ in the Toda range.
In Section~\ref{sec:EHPSScalc} we describe the computation the TEHPSS in the Toda range.  Note that since $\pi_{n+*}(S^n)$ is metastable for $* \le 19$ and $n \ge 7$, we do not bother to compute the TGSS for $S^n$ for $n \ge 7$: in the Toda range it can be completely read off of the TEHPSS.

The methodology of these computations can be summarized as follows.
\begin{enumerate}
\item Inductively compute the TAHSS for $L(k)$, facilitated by an understanding of the dual action of the Steenrod algebra on the homology of the $L(k)$-spectra, and an excellent understanding of the structure of the stable homotopy ring in low dimensions.
\item Feed the TAHSS computations into the TGSS for $S^n$ for various $n$, starting with $n = 1$.  Use stable Hopf invariant computations to fill in most of the differentials for the TGSS for $S^1$, and then deduce bizarre differentials using our knowledge of $\pi_*(S^1)$.  Then fill in the differentials for the TGSS for $S^n$ inductively on $n$, using the $E$ map to map differentials from the TGSS for $S^{n-1}$ to the TGSS for $S^n$.  The remaining differentials come from the geometric boundary effect (Theorem~\ref{thm:GBE}) or unstable Hopf invariants (Theorem~\ref{thm:GSSdr}).
\item The unstable Hopf invariants, and to some degree the stable Hopf invariants, require a concurrent computation of the TEHPSS.  Differentials in the TEHPSS are lifted from the TAHSS and TGSS using Theorem~\ref{thm:TEHPdiffs}.   
\end{enumerate}

\section{AHSS calculations}\label{sec:AHSScalc}
The TAHSS's for $\pi_*(L(s))$ in the Toda range for $1 \le s \le 3$ are given in
Tables~\ref{tab:L(1)AHSS}, \ref{tab:L(2)AHSS}, and \ref{tab:L(3)AHSS}.
As the TAHSS for $L(k)$ completely determines the $E^\alpha$-pages of the TAHSS for $L(k+1)$ for $\alpha \in \mc{G}(\omega^k)$, we really only need to start the TAHSS for $L(k)$ at the $E^{\omega^{k}}$-pages of the TAHSS's for $L(k)_{2m+1}$ for various $m$ (see Remark~\ref{rmk:TAHSS}).

The terms are given in terms of an $\FF_2$-basis of the associated graded
with respect to the $2$-adic filtration.  The notation 
$$ x(2^m)[n_1, \ldots ,n_s]$$ 
represents the span of the elements 
$$ \{x[N], 2x[N], 4x[N], \ldots, 2^{m-1}[N]\}$$
where $N$ is a CU sequence of length $s$, and $x$ is
an element of the stable stems.  
If $m = 1$, then the ``$(2^m)$'' in the notation is omitted.

We list all elements which are
not the targets of differentials.  If an element supports a non-trivial 
differential,
then the differential is represented by an arrow which points toward the
target.

We use standard terminology for the generators of the stable stems, most of
which originates with Toda \cite{Toda}.  The generators of the
$v_1$-periodic homotopy groups are given using a variation on the notation
of Ravenel \cite{Ravenel}.  The element $\alpha_{k/l}$ is a $v_1$-periodic
element of order $2^l$ in the $(2k-1)$-stem.

The most difficult of these computations is the computation of
$\pi_*(L(1))$ given in Table~\ref{tab:L(1)AHSS}.  These computations are effectively contained in \cite{Mahowaldmeta}.  The differentials on the
$v_1$-periodic elements are mostly deduced from the $J$-homology AHSS for
$P^\infty$, as computed in \cite{MahowaldJ}.  

The remaining differentials
may be deduced in this range from the dual action of the Steenrod algebra on 
$H^*(L(1))$.  
For instance, in the $k=5$ list, the formula
$$ \Sq^2_* \bar{Q}^4 = \bar{Q}^2 $$
implies the AHSS differential
$$ d_2(\eta[4]) = \eta\cdot\eta[2] = \eta^2[2] $$
and in $k=12$, the formula
$$ \Sq^2_*Sq^1_* \bar{Q}^6 = \bar{Q}^3 $$
implies the AHSS differential
$$ d_3(\nu^2[6]) = \bra{\eta, 2, \nu^2}[3] = \epsilon[3]. $$

The TAHSS's for $L(k)$, $k = 2,3$ in the Toda range are actually much less complicated.  The Steenrod
operations on $H_*(L(s))$ are given by the Nishida relations.  There are
subtleties associated with the transfinite AHSS filtration.  
We give an example that nicely illustrates how to interpret things.
We look at the TAHSS for $L(2)$, $k=19$. The Nishida relations, together with the relations amongst the $\bar{Q}^j$-operations, give
$$ \Sq^4_* \bar{Q}^9\bar{Q}^4 = \bar{Q}^7\bar{Q}^2 + \bar{Q}^5\bar{Q}^4 = \bar{Q}^7\bar{Q}^2. $$
We would initially think that this should give an AHSS differential
$$ d^{L(2)}(\nu^2[9,4]) = \nu^3 [7,2] = (\sigma\eta^2+\epsilon\eta)[7,2] $$
except that $(\sigma\eta^2+\epsilon\eta)[7]$ is null in $\pi_*(L(1)_7)$.  
In the AHSS for $L(1)$,
$(\sigma\eta^2+\epsilon\eta)[7]$ is killed by the differential $d^{L(1)}((\sigma\eta+\epsilon)[9])$ induced
from
$$ Sq^2_*\bar{Q}^9 = \bar{Q}^7. $$
However, something more is going on in $L(2)$.  The Nishida relations give
$$ Sq^2_*\bar{Q}^9\bar{Q}^2 = \bar{Q}^8\bar{Q}^1 + \bar{Q}^7\bar{Q}^2 $$
which translate into the fact that in $L(2)$, we have
$$ d^{L(2)}((\sigma\eta+\epsilon)[9,2]) = (\sigma\eta^2+\epsilon\eta)[8,1] + (\sigma\eta^2+\epsilon\eta)[7,2] $$
so the terms in the right-hand side are equated.  Thus we actually have
$$ d^{L(2)}(\nu^2[9,4]) = (\sigma\eta^2+\epsilon\eta)[7,2] = (\sigma\eta^2+\epsilon\eta)[8,1]. $$

\section{Calculation of the GSS for $S^1$}\label{sec:Kuhncalc}  The TGSS for $\pi_*(S^1)$ is given in
Table~\ref{tab:S1GSS}.  
We know that $\pi_*(S^1)$ is concentrated in degree $1$, and
consistent with the Goodwillie-Whitehead conjecture, the only differentials it has
are are GSS $d_1$'s.

The $E^\alpha$ terms for $\alpha \in \mc{G}(\omega^\omega)$ may be read off of Tables~\ref{tab:L(1)AHSS},
\ref{tab:L(2)AHSS} and \ref{tab:L(3)AHSS}.  The other $L(k)$'s are too
highly connected to contribute anything in our range of computation.
The $d_1$'s originating in the first column of Table~\ref{tab:S1GSS} are
exactly the stable Hopf invariants, by Theorem~\ref{thm:GSSd1}.  For example, we have $SHI(\eta) = 1$
and 
$$d_1(\eta) = 1[1].$$  
Most of the other $d_1$'s are also given by stable Hopf invariants, using Theorem~\ref{thm:higherd1}.  The
ones that are not are the bizarre differentials: these are indicated by dashed arrows in Table~\ref{tab:S1GSS}.
They are necessary to make the
spectral sequence acyclic.

Actually, in this range, these stable Hopf invariants are rather easy to
deduce from the Hopf invariants of $\eta$, $\nu$, and $\sigma$ and the
attaching map structure in $L(1)$.  We also used the relation of Hopf
invariants to root invariants \cite{MahowaldRavenel}.

\section{GSS calculations}\label{sec:GSScalc}
Tables~\ref{tab:S2GSS}, \ref{tab:S3GSS}, \ref{tab:S4GSS}, \ref{tab:S5GSS}, and \ref{tab:S6GSS} display calculations of the TGSS for $S^n$ for $2 \le n \le 6$.  In the Toda range, the TGSS for $S^n$ can be completely read off of the TEHPSS for $n \ge 7$ (metastable range). 

As in Section~\ref{sec:Kuhncalc}, the $E^\alpha$ terms for $\alpha \in \mc{G}(\omega^\omega)$ may be read off of Tables~\ref{tab:L(1)AHSS},
\ref{tab:L(2)AHSS} and \ref{tab:L(3)AHSS}.
The unmarked differentials all arise from the application of
Theorems~\ref{thm:higherd1} and \ref{thm:GSSdr}.  The geometric boundary effect differentials deduced from Theorem~\ref{thm:GBE} are denoted with a $(*)$. The bizarre differentials induced from the TGSS for $S^1$ (using Proposition~\ref{prop:EGSS}) are denoted with a $(**)$.

Note that there are precisely two longer differentials in the TGSS for $S^2$ obtained from Theorem~\ref{thm:GSSdr}: we must use our TEHPSS calculations to compute these generalized Hopf invariants.
For example, in the TGSS for $S^2$
in $n=16$ there is a TGSS differential
$$ d^{S^2}(\eta\alpha_{8/5}) = \eta^2[13,2]. $$
This arises from the fact that the unstable Hopf invariant
$HI(\eta\alpha_{8/5})$ is detected by $\eta^2[13]$ in the
TGSS for $S^5$.  
The stable element $\eta\alpha_{8/5}$ has an \emph{unstable}
Hopf invariant, and this is precisely the mechanism which results in such a longer
TGSS differential.

\section{Calculation of the EHPSS}\label{sec:EHPSScalc}
Table~\ref{tab:EHPSS} displays the TEHPSS in the Toda range.
The $E^\alpha$-page of the TEHPSS can be read off of Tables~\ref{tab:L(1)AHSS}, \ref{tab:L(2)AHSS}, \ref{tab:L(3)AHSS}, \ref{tab:S1GSS}, \ref{tab:S3GSS}, and \ref{tab:S5GSS}.
Table~\ref{tab:EHPSS} consists of lists of the non-trivial permanent cycles in the $E^{\omega^{\omega+1}}$-pages of the TGSS's for the various $S^{2m+1}$.
If an element is the target of a
TEHPSS differential, then this is displayed with an left arrow followed by the
source of the differential.  If an element is not the target of a
differential, then it is boxed, and a double right arrow gives the name of
the element of the stable stems that it detects in the TEHPSS.  One can read
read off the (unstable) 
Hopf invariants of the elements of the stable stems.

The elements in the $E^1_{t,N}$-term of the TEHPSS are given with the notation
$$ x(2^m)[n_1, \ldots, n_s,n] $$
where $x$ is an element of the stable stems, and $N = (n_1, n_2, \ldots,
n_s,n)$ is a CU sequence.  The parenthetical
$(2^m)$ is the order of the element, and this notation is omitted if $m=1$.
The meaning of this notation is completely explained in Section~\ref{sec:TEHPname}.

With the exception of the one bad differential (marked with a $(**)$) described in Section~\ref{sec:bad}, all of the TEHPSS differentials in the Toda range come from lifting a TGSS differential using Theorem~\ref{thm:TEHPdiffs}.
The differentials with unmarked arrows all arise from lifting differentials from the TAHSS's in Tables~\ref{tab:L(1)AHSS}, \ref{tab:L(2)AHSS}, and \ref{tab:L(3)AHSS}.  Differentials obtained from lifting geometric boundary effect differentials are labeled with a $(*)$, and differentials obtained from lifting bizarre differentials are labeled with a $(***)$.

\section{Tables of computations}

This section consists of the actual tables containing the $2$-primary Toda range computations discussed in the previous sections.
\begin{verse}
Tables~\ref{tab:L(1)AHSS}--\ref{tab:L(3)AHSS}: The AHSS for $\pi_k(L(k))$, $1 \le k \le 3$ \\
Table~\ref{tab:EHPSS}: The EHPSS \\
Tables~\ref{tab:S1GSS}--\ref{tab:S6GSS}: The GSS for $\pi_{t+n}(S^n)$, $1 \le n \le 6$ 
\end{verse}


\subsection{The AHSS for $\pi_k(L(1))$}\label{tab:L(1)AHSS} $\quad$

\begin{tabular}{lll}
$\begin{array}{l}
\underline{k=1} \\
1(\infty)[1] 
\end{array}$
&
$\begin{array}{l}
\underline{k=2} \\
\eta[1] \\
1(\infty)[2] \rightarrow 2(\infty)[1]
\end{array}$
&
$\begin{array}{l}
\underline{k=3} \\
\eta^2[1] \\
\eta[2] \\
1[3]
\end{array}$
\\ \\
$\begin{array}{l}
\underline{k=4} \\
\nu[1] \\
1(\infty)[4] \rightarrow 2(\infty)[3]
\end{array}$
&
$\begin{array}{l}
\underline{k=5} \\
\nu(4)[2] \rightarrow 2\nu(4)[1] \\
\eta[4] \rightarrow \eta^2[2] \\
1[5] \rightarrow \eta[3]
\end{array}$
&
$\begin{array}{l}
\underline{k=6} \\
\nu[3] \\
\eta^2[4] \rightarrow 4\nu[2] \\
\eta[5] \rightarrow \eta^2[3] \\
1(\infty)[6] \rightarrow 2(\infty)[5]
\end{array}$
\\ \\
$\begin{array}{l}
\underline{k=7} \\
\nu^2[1] \\
\nu(4)[4] \rightarrow 2\nu(4)[3] \\
4\nu[4] \\
\eta^2[5] \\
\eta [6] \\
1[7] 
\end{array}$
&
$\begin{array}{l}
\underline{k=8} \\
\sigma[1] \\
\nu^2[2] \\
\nu[5] \\
1(\infty)[8] \rightarrow 2(\infty)[7]
\end{array}$
&
$\begin{array}{l}
\underline{k=9} \\
\sigma\eta[1] \\
\epsilon[1] \\
\sigma(8)[2] \rightarrow 2\sigma(8)[1] \\
8\sigma[2] \\
\nu^2[3] \\
\nu(4)[6] \rightarrow 2\nu(4)[5] \\
\eta[8] \rightarrow \eta^2[6] \\
1[9] \rightarrow \eta[7]
\end{array}$
\\ \\
$\begin{array}{l}
\underline{k=10} \\
\sigma\eta^2[1] \\
\alpha_5[1] \\
\sigma\eta[2] \\
\sigma[3] \\
\eta^2[8] \rightarrow 4\nu[6] \\
\eta[9] \rightarrow \eta^2[7] \\
1(\infty)[10] \rightarrow 2(\infty)[9]
\end{array}$
&
$\begin{array}{l}
\underline{k=11} \\
\eta\alpha_5[1] \\
\alpha_5[2] \\
\sigma(8)[4] \rightarrow 2\sigma(8)[3] \\
8\sigma[4] \\
\nu(4)[8] \rightarrow 2\nu(4)[7] \\
4\nu[8] \rightarrow \epsilon\eta[1] \\
\eta^2[9] \rightarrow \epsilon[2] \\
\eta[10] \rightarrow \nu^2[4] \\
1[11] \rightarrow \nu[7]
\end{array}$
&
$\begin{array}{l}
\underline{k=12} \\
\sigma\eta[4] \rightarrow \sigma\eta^2[2] \\
\epsilon[4] \rightarrow \epsilon\eta[2] \\
\sigma[5] \rightarrow \sigma\eta[3] \\
\nu^2[6] \rightarrow \epsilon[3] \\
\nu[9] \rightarrow \nu^2[5] \\
1(\infty)[12] \rightarrow 2(\infty)[11]
\end{array}$
\\ \\
$\begin{array}{l}
\underline{k=13} \\
\alpha_{6/3}(4)[2] \rightarrow \alpha_{6/2}(4)[1] \\
\stackem{\epsilon\eta[4]}{\sigma\eta^2[4]} \bigg\} \rightarrow \alpha_{6/3}[1] \\
\alpha_5[4] \rightarrow \alpha_5\eta[2] \\
\sigma\eta[5] \rightarrow \sigma\eta^2[3] \\
\epsilon[5] \rightarrow \epsilon\eta[3] \\
\sigma(8)[6] \rightarrow 2\sigma(8)[5] \\
8\sigma[6] \rightarrow \alpha_5[3] \\
\nu(4)[10] \rightarrow 2\nu(4)[9] \\
\eta[12] \rightarrow \eta^2[10] \\
1[13] \rightarrow \eta[11]
\end{array}$
&
$\begin{array}{l}
\underline{k=14} \\
\alpha_5\eta[4] \rightarrow \alpha_6[2] \\
\alpha_5[5] \rightarrow \alpha_5\eta[3] \\
\epsilon[6] \\
\sigma[7] \\
\nu^2[8] \rightarrow (\sigma\eta^2+\epsilon\eta)[4] \\
\nu[11] \rightarrow \nu^2[7] \\
\eta^2[12] \rightarrow 4\nu[10] \\
\eta[13] \rightarrow \eta^2[11] \\
1(\infty)[14] \rightarrow 2(\infty)[13] 
\end{array}$
&
$\begin{array}{l}
\underline{k=15} \\
\theta_3[1] \\
\kappa[1] \\
\alpha_{6/3}(4)[4] \rightarrow \alpha_{6/2}(4)[3] \\
\alpha_6[4] \\
\alpha_5\eta[5] \\
\alpha_5[6] \\
\sigma(8)[8] \rightarrow 2\sigma(8)[7] \\
8\sigma[8] \\
\nu^2[9] \rightarrow (\sigma\eta^2+\epsilon\eta)[5] \\
\nu(4)[12] \rightarrow 2\nu(4)[11] \\
4\nu[12] \\
\eta^2[13] \rightarrow \alpha_{6/3}[3] \\
\eta[14] \rightarrow \epsilon\eta[5] \\
1[15] \rightarrow \sigma\eta[6]
\end{array}$
\\
\end{tabular}

\begin{tabular}{lll}
{\it Table~\ref{tab:L(1)AHSS}, cont'd}
\\ \\
$\begin{array}{l}
\underline{k=16} \\
\alpha_{8/5}[1] \\
\theta_3[2] \\
\kappa[2] \\
\sigma\eta[8] \rightarrow \sigma\eta^2[6] \\
\epsilon[8] \rightarrow \epsilon\eta[6] \\
\sigma[9] \rightarrow \sigma\eta[7] \\
\nu^2[10] \rightarrow \epsilon[7] \\
\nu[13] \\
1(\infty)[16] \rightarrow 2(\infty)[15]
\end{array}$
&
$\begin{array}{l}
\underline{k=17} \\
\eta_4[1] \\
\alpha_{8/5}\eta[1] \\
\alpha_{8/5}(16)[2] \rightarrow \alpha_{8/4}(16)[1] \\
\alpha_{8}[2] \\
\theta_3[3] \\
\kappa[3] \\
\alpha_{6/3}(4)[6] \rightarrow \alpha_{6/2}(4)[5] \\
\stackem{\sigma\eta^2[8]}{\epsilon\eta[8]} \bigg\} \rightarrow
\alpha_{6/3}[5] \\
(\sigma\eta^2+\epsilon\eta)[8] \rightarrow \kappa\eta[1] \\
\alpha_5[8] \rightarrow \alpha_5\eta[6] \\
\sigma\eta[9] \rightarrow \sigma\eta^2[7] \\
\epsilon[9] \rightarrow \epsilon\eta[7] \\
\sigma(8)[10] \rightarrow 2\sigma(8)[9] \\
8\sigma[10] \rightarrow \alpha_5[7] \\
\nu^2[11] \\
\nu(4)[14] \rightarrow 2\nu(4)[13] \\
\eta[16] \rightarrow \eta^2[14] \\
1[17] \rightarrow \eta[15]
\end{array}$
&
$\begin{array}{l}
\underline{k=18} \\
\eta_4\eta[1] \\
\alpha_9[1] \\
\eta_4[2] \\
\theta_3[4] \\
\kappa[4] \rightarrow \kappa\eta[2] \\
\alpha_5\eta[8] \rightarrow \alpha_6[6] \\
\sigma\eta^2[9] \\
\alpha_5[9] \rightarrow \alpha_5\eta[7] \\
\sigma\eta[10] \\
\sigma[11] \\
\eta^2[16] \rightarrow 4\nu[14] \\
\eta[17] \rightarrow \eta^2[15] \\
1(\infty)[18] \rightarrow 2(\infty)[17] 
\end{array}$
\\ \\
$\begin{array}{l}
\underline{k=19} \\
\nu^*[1] \\
\alpha_9\eta[1] \\
\kappa\nu[2] \\
\alpha_9[2] \\
\alpha_{8/5}(16)[4] \rightarrow \alpha_{8/4}(16)[3] \\
16\alpha_{8/5}[4] \\
\kappa\eta[4] \rightarrow \kappa\nu[1] \\
\theta_3[5] \\
\kappa[5] \rightarrow \kappa\eta[3] \\
\alpha_{6/3}(4)[8] \rightarrow \alpha_{6/2}(4)[7] \\
\alpha_6[8] \rightarrow \alpha_{8/5}\eta^2[1] \\
\alpha_5\eta[9] \rightarrow \alpha_{8/5}\eta[2] \\
\alpha_5[10] \rightarrow \alpha_{8/5}[3] \\
\sigma(8)[12] \rightarrow 2\sigma(8)[11] \\
8\sigma[12] \rightarrow \alpha_{6/3}[7] \\
\nu(4)[16] \rightarrow 2\nu(4)[15] \\
4\nu[16] \rightarrow \epsilon\eta[9] \\
\eta^2[17] \rightarrow \epsilon[10] \\
\eta[18] \rightarrow \nu^2[12] \\
1[19] \rightarrow \nu[15]
\end{array}$
&
$\begin{array}{l}
\underline{k=20} \\
\br{\sigma}[1] \\
\nu^*(4)[2] \rightarrow 2\nu^*(4)[1] \\
\kappa\nu[3] \\
\eta_4[4] \rightarrow \eta_4\eta[2] \\
\alpha_{8/5}\eta[4] \rightarrow \alpha_{8/5}\eta^2[2] \\
\alpha_{8/5}[5] \rightarrow \alpha_{8/5}\eta[3] \\
\kappa\eta[5] \\
\theta_3[6] \rightarrow \eta_4[3] \\
\kappa[6] \\
\sigma\eta[12] \rightarrow \sigma\eta^2[10] \\
\epsilon[12] \rightarrow \epsilon\eta[10] \\
\sigma[13] \rightarrow \sigma\eta[11] \\
\nu^2[14] \rightarrow \epsilon[11] \\
\nu[17] \rightarrow \nu^2[13] \\
1(\infty)[20] \rightarrow 2(\infty)[19]
\end{array}$
&
$\begin{array}{l} 
\underline{k=21} \\
\br{\kappa}[1] \\
\br{\sigma}[2] \\
\alpha_{10/3}(4)[2] \rightarrow \alpha_{10/2}(4)[1] \\
\nu^*[3] \\
\eta_4\eta[4] \rightarrow 4\nu^*[2] \\
\alpha_{8/5}\eta^2[4] \rightarrow \alpha_{10/3}[1] \\
\kappa\nu[4] \\
\alpha_9[4] \rightarrow \alpha_9\eta[2] \\
\eta_4[5] \rightarrow \eta_4\eta[3] \\
\alpha_{8/5}\eta[5] \rightarrow \alpha_{8/5}\eta^2[3] \\
\alpha_{8/5}(16)[6] \rightarrow \alpha_{8/4}(16)[5] \\
\alpha_8[6] \rightarrow \alpha_9[3] \\
\theta_3[7] \\
\alpha_{6/3}(4)[10] \rightarrow \alpha_{6/2}(4)[9] \\
\stackem{\sigma\eta^2[12]}{\epsilon\eta[12]}\bigg\} \rightarrow
\alpha_{6/3}[9] \\
\alpha_5[12] \rightarrow \alpha_5\eta[10] \\
\sigma\eta[13] \rightarrow \sigma\eta^2[11] \\
\epsilon[13] \rightarrow \epsilon\eta[11] \\
\sigma(8)[14] \rightarrow 2\sigma(8)[13] \\
8\sigma[14] \rightarrow \alpha_5[11] \\
\nu(4)[18] \rightarrow 2\nu(4)[17] \\
\eta[20] \rightarrow \eta^2[18] \\
1[21] \rightarrow \eta[19] 
\end{array}$
\\
\end{tabular}

\begin{tabular}{ll}
{\it Table~\ref{tab:L(1)AHSS}, cont'd} \\
\\ \\
$\begin{array}{l}
\underline{k=22} \\
\sigma^3[1] \\
\br{\kappa}\eta[1] \\
\br{\kappa}(4)[2] \rightarrow 2\br{\kappa}(4)[1] \\
4\br{\kappa}[2] \\
\br{\sigma}[3] \\
\nu^*(4)[4] \rightarrow 2\nu^*(4)[3] \\
4\nu^*[4] \\
\alpha_9\eta[4] \rightarrow \alpha_{10}[2] \\
\eta_4\eta[5] \\
\alpha_9[5] \rightarrow \alpha_9\eta[3] \\
\eta_4[6] \\
\theta_3[8] \\
\kappa[8] \rightarrow \kappa\eta[6] \\
\alpha_5\eta[12] \rightarrow \alpha_6[10] \\
\alpha_5[13] \rightarrow \alpha_5\eta[11] \\
\epsilon[14] \rightarrow \kappa[7] \\
\nu^2[16] \rightarrow (\sigma\eta^2+\epsilon\eta)[12] \\
\nu[19] \rightarrow \nu^2[15] \\
\eta^2[20] \rightarrow 4\nu[18] \\
\eta[21] \rightarrow \eta^2[19] \\
1(\infty)[22] \rightarrow 2(\infty)[21] 
\end{array}$
&
$\begin{array}{l}
\underline{k=23\text{(outgoing diffs only)}} \\
\alpha_{10/3}(4)[4] \rightarrow \alpha_{10/2}(4)[3] \\
\alpha_{8/5}(16)[8] \rightarrow \alpha_{8/4}(16)[7] \\
\kappa\eta[8] \rightarrow \kappa\nu[5] \\
\kappa[9] \rightarrow \kappa\eta[7] \\
\alpha_{6/3}(4)[12] \rightarrow \alpha_{6/2}(4)[11] \\
\alpha_6[12] \rightarrow \alpha_{10/3}[3] \\
\alpha_5\eta[13] \rightarrow \alpha_{8/5}\eta^2[5] \\
\alpha_5[14] \rightarrow \alpha_{8/5}\eta[6] \\
\sigma(8)[16] \rightarrow 2\sigma(8)[15] \\
8\sigma[16] \rightarrow \alpha_{8/5}[7] \\
\nu^2[17] \rightarrow \sigma\eta^2[13] \\
\nu(4)[20] \rightarrow 2\nu(4)[19] \\
4\nu[20] \rightarrow \alpha_{6/3}[11] \\
\eta^2[21] \rightarrow \epsilon\eta[13] \\
\eta[22] \rightarrow \sigma\eta[14] \\
1[23] \rightarrow \sigma[15] 
\end{array}$
\\
\end{tabular}

\subsection{The AHSS for $\pi_k(L(2))$}\label{tab:L(2)AHSS} $\quad$

\begin{tabular}{lll}
$\begin{array}{l}
\underline{k=4} \\
1[3,1]
\end{array}$
&
$\begin{array}{l}
\underline{k=7} \\
\nu[3,1] \\
1[5,2] \rightarrow \eta[4,1]
\end{array}$
&
$\begin{array}{l}
\underline{k=8} \\
1[7,1] \\
\eta[5,2] \rightarrow \eta^2[4,1]
\end{array}$
\\ \\
$\begin{array}{l}
\underline{k=9} \\
\nu[5,1] \\
\eta^2[5,2] \rightarrow \eta^3[4,1] \\
\eta[6,2] \rightarrow \eta^2[5,1] \\
1[7,2] \rightarrow \eta[6,1]
\end{array}$
&
$\begin{array}{l} 
\underline{k=10} \\
\nu^2[3,1] \\
\nu[5,2] \\
1[7,3] 
\end{array}$
&
$\begin{array}{l}
\underline{k=11} \\
\sigma[3,1]
\end{array}$
\\ \\
$\begin{array}{l}
\underline{k=13} \\
\eta^3[8,2] \rightarrow 8\sigma[4,1] \\
\eta^2[9,2] \rightarrow \eta^3[8,1] \\
\eta[10,2] \rightarrow \eta^2[9,1] \\
1[9,4] \rightarrow \eta[8,3]
\end{array}$
&
$\begin{array}{l}
\underline{k=14} \\
\sigma[5,2] \rightarrow \sigma\eta[4,1] \\
\nu^2[6,2] \rightarrow \epsilon[4,1] \\
\eta[9,4] \rightarrow \eta^2[8,3] 
\end{array}$
&
$\begin{array}{l} 
\underline{k = 15} \\
\sigma[7,1] \\
\epsilon[5,2] \rightarrow \epsilon\eta[4,1] \\
8\sigma[6,2] \rightarrow \alpha_5[4,1] \\
\eta^2[9,4] \rightarrow \eta^3[8,3] \\
\eta[10,4] \rightarrow \eta^2[9,3] \\
1[11,4] \rightarrow \eta[10,3]
\end{array}$
\\
\end{tabular}

\begin{tabular}{lll}
{\it Table~\ref{tab:L(2)AHSS}, cont'd} 
\\ \\
$\begin{array}{l}
\underline{k=16} \\
\alpha_5[5,2] \rightarrow \eta\alpha_5[4,1] \\
\epsilon[6,2] \\
\sigma[7,2] \\
\nu^2[8,2] \rightarrow \epsilon[6,1] \\
\nu[9,4] \rightarrow \sigma\eta[5,2] \\
1[11,5] \rightarrow \nu[9,3] 
\end{array}$
&
$\begin{array}{l}
\underline{k=17} \\
\nu[13,1] \\
\eta\alpha_5[5,2] \rightarrow \alpha_6[4,1] \\
\alpha_5[6,2] \rightarrow \eta\alpha_5[5,1] \\
8\sigma[8,2] \rightarrow \alpha_5[6,1] \\
\eta^3[12,2] \rightarrow 8\sigma[8,1] \\
\eta^2[13,2] \rightarrow \eta^3[12,1] \\
\sigma[7,3]
\end{array}$
&
$\begin{array}{l}
\underline{k=18} \\
\theta_3[3,1] \\
\nu^2[11,1] \\
\nu[13,2] \\
1[15,3] \\
\nu[11,4] \rightarrow \nu^2[8,3]
\end{array}$
\\ \\
$\begin{array}{l}
\underline{k=19} \\
\theta_3[4,1] \\
\kappa[4,1] \\
\sigma[11,1] \\
(\sigma\eta^2+\epsilon\eta)[8,2] \rightarrow \kappa[3,1] \\
\nu^2[9,4] \rightarrow (\sigma\eta^2+\epsilon\eta)[8,1] \\
\eta^3[12,4] \rightarrow 8\sigma[8,3] \\
\eta^2[13,4] \rightarrow \eta^3[12,3] \\
\eta[14,4] \rightarrow \eta^2[13,3] \\
1[15,4] \rightarrow \eta[14,3] \\
\nu[11,5] \rightarrow \nu^2[9,3] \\
1[13,6] \rightarrow \eta[12,5]
\end{array}$
&
$\begin{array}{l}
\underline{k=20} \\
\theta_3[5,1] \\
\sigma\eta[10,2] \rightarrow \sigma\eta^2[9,1] \\
\sigma[11,2] \rightarrow \sigma\eta[10,1] \\
\sigma[9,4] \rightarrow \sigma\eta[8,3] \\
\nu^2[10,4] \rightarrow \epsilon[8,3] \\
\nu[13,4] \rightarrow \nu^2[11,2] \\
1[15,5] \rightarrow \nu[13,3] \\
\eta[13,6] \rightarrow \eta^2[12,5]
\end{array}$
&
$\begin{array}{l}
\underline{k=21} \\
\kappa[6,1] \\
\theta_3[5,2] \\
\kappa[5,2] \rightarrow \kappa\eta[4,1] \\
\alpha_6[8,2] \rightarrow \alpha_8[4,1] \\
\eta\alpha_5[9,2] \rightarrow \alpha_6[8,1] \\
\alpha_5[10,2] \rightarrow \eta\alpha_5[9,1] \\
\sigma[11,3] \\
\sigma\eta[9,4] \rightarrow \sigma\eta^2[8,3] \\
\epsilon[9,4] \rightarrow \epsilon\eta[8,3] \\
8\sigma[10,4] \rightarrow \alpha_5[8,3] \\
\nu^2[11,4] \rightarrow \sigma\eta^2[9,2] \\
\nu[13,5] \rightarrow \nu^2[11,3] \\
\eta^2[13,6] \rightarrow 4\nu[12,5] \\
\eta[14,6] \rightarrow \eta^2[13,5] \\
1[15,6] \rightarrow \eta[14,5]
\end{array}$
\\ \\
$\begin{array}{l}
\underline{k=22} \\
\nu^*[3,1] \\
\nu\kappa[4,1] \\
\theta_3[7,1] \\
\alpha_{8/5}[5,2] \rightarrow \eta\alpha_{8/5}[4,1] \\
\kappa\eta[5,2] \rightarrow \nu\kappa[3,1] \\
\theta_3[6,2] \rightarrow \eta_4[4,1] \\
\kappa[6,2] \rightarrow \kappa\eta[5,1] \\
\sigma\eta^2[9,4] \\
\alpha_5[9,4] \rightarrow \eta\alpha_5[8,3] \\
\sigma\eta[10,4] \rightarrow \sigma\eta^2[9,3] \\
\sigma[11,4] \rightarrow \sigma\eta[10,3] \\
\nu^2[11,5] \\
\nu[13,6] \\
1[15,7]
\end{array}$
&
$\begin{array}{l}
\underline{k=23 \text{(outgoing diffs only)}} \\
\eta_4[5,2] \rightarrow \eta\eta_4[4,1] \\
\eta\alpha_{8/5}[5,2] \rightarrow \eta^2\alpha_{8/5}[4,1] \\
\alpha_8[6,2] \rightarrow \alpha_9[4,1] \\
\eta\alpha_5[9,4] \rightarrow \alpha_6[8,3] \\
\alpha_5[10,4] \rightarrow \eta\alpha_5[9,3] \\
8\sigma[12,4] \rightarrow \alpha_5[10,3]
\end{array}$
\\
\end{tabular}
\vfill

\subsection{The AHSS for $\pi_k(L(3))$}\label{tab:L(3)AHSS} $\quad$

\begin{tabular}{lll}
$\begin{array}{l}
\underline{k=11} \\
1[7,3,1]
\end{array}$
&
$\begin{array}{l}
\underline{k=18} \\
\sigma[7,3,1] \\
1[11,5,2] \rightarrow \nu[9,4,1]
\end{array}$
&
$\begin{array}{l}
\underline{k=19} \\
1[15,3,1] 
\end{array}$
\\ \\
$\begin{array}{l}
\underline{k=21} \\
\nu[11,5,2] \rightarrow \nu^2[9,4,1]
\end{array}$
&
$\begin{array}{l}
\underline{k=22} \\
\sigma[11,3,1] \\
1[15,5,2] \rightarrow \nu[13,4,1] 
\end{array}$
&
$\begin{array}{l} 
\underline{k=23} \\
1[15,7,1] \\
\nu[13,5,2] \rightarrow \nu^2[11,4,1] 
\end{array}$
\\
\end{tabular}

\subsection{The EHPSS}\label{tab:EHPSS} $\quad$

\begin{tabular}{lll}
$\begin{array}{l}
\underline{k=0} \\
\boxed{1(\infty)[0]} \Rightarrow 1(\infty) 
\end{array}$
&
$\begin{array}{l}
\underline{k=1} \\
\boxed{1[1]} \Rightarrow \eta \\
2(\infty)[1] \leftarrow 1(\infty)[2] 
\end{array}$
&
$\begin{array}{l}
\underline{k=2} \\
\boxed{\eta[1]} \Rightarrow \eta^2
\end{array}$
\\ \\
$\begin{array}{l}
\underline{k=3} \\
\boxed{\eta^2[1]} \Rightarrow 4\nu \\
\boxed{\eta[2]} \Rightarrow 2\nu \\
\boxed{1[3]} \Rightarrow \nu \\
2(\infty)[3] \leftarrow 1(\infty)[4]
\end{array}$
&
$\begin{array}{l}
\underline{k=4} \\
2\nu(4)[1] \leftarrow \nu(4)[2] \\
\eta^2[2] \leftarrow \eta[4] \\
\eta[3] \leftarrow 1[5]
\end{array}$
&
$\begin{array}{l}
\underline{k=5} \\
\eta[4,1] \leftarrow 1[5,2] \\
4\nu[2] \leftarrow \eta^2[4] \\
\eta^2[3] \leftarrow \eta[5] \\
2(\infty)[5] \leftarrow 1(\infty)[6]
\end{array}$
\\ \\
$\begin{array}{l}
\underline{k=6} \\
\eta^2[4,1] \leftarrow \eta[5,2] \\
\boxed{\nu[3]} \Rightarrow \nu^2 \\
2\nu(4)[3] \leftarrow \nu(4)[4]
\end{array}$
&
$\begin{array}{l}
\underline{k=7} \\
\boxed{4\nu[4]} \Rightarrow 8\sigma \\
\boxed{\eta^2[5]} \Rightarrow 4\sigma \\
\boxed{\eta[6]} \Rightarrow 2\sigma \\
\boxed{1[7]} \Rightarrow \sigma \\
2(\infty)[7] \leftarrow 1(\infty)[8] 
\end{array}$
&
$\begin{array}{l}
\underline{k=8} \\
\boxed{\nu^2[2]} \Rightarrow \epsilon \\
\boxed{\nu[5]} \Rightarrow \sigma\eta \\
2\nu(4)[5] \leftarrow \nu(4)[6] \\
\eta^2[6] \leftarrow \eta[8] \\
\eta[7] \leftarrow 1[9]
\end{array}$
\\ \\
$\begin{array}{l}
\underline{k=9} \\
\boxed{\epsilon[1]} \Rightarrow \epsilon\eta \\
\boxed{8\sigma[2]} \Rightarrow \alpha_5 \\
\boxed{\nu^2[3]} \Rightarrow \sigma\eta^2 \\
4\nu[6] \leftarrow \eta^2[8] \\
\eta^2[7] \leftarrow \eta[9] \\
2(\infty)[9] \leftarrow 1(\infty)[10] 
\end{array}$
&
$\begin{array}{l}
\underline{k=10} \\
\eta\epsilon[1] \leftarrow 4\nu[8] \\
\boxed{\alpha_5[1]} \Rightarrow \eta\alpha_5 \\
\epsilon[2] \leftarrow \eta^2[9] \\
2\sigma(8)[3] \leftarrow \sigma(8)[4] \\
\nu^2[4] \leftarrow \eta[10] \\
\nu[7] \leftarrow 1[11] \\
2\nu(4)[7] \leftarrow \nu(4)[8]
\end{array}$
&
$\begin{array}{l}
\underline{k=11} \\
\boxed{\eta\alpha_5[1]} \Rightarrow \alpha_{6} \\
\eta^3[8,1] \leftarrow \eta^2[9,2] \\
\eta^2[9,1] \leftarrow \eta[10,2] \\
\epsilon\eta[2] \leftarrow \epsilon[4] \\
\boxed{\alpha_5[2]} \Rightarrow \alpha_{6/2} \\
\sigma\eta^2[2] \leftarrow \sigma\eta[4] \\
\epsilon[3] \leftarrow \nu^2[6] \\
\sigma\eta[3] \leftarrow \sigma[5] \\
\eta[8,3] \leftarrow 1[9,4] \\
\boxed{8\sigma[4]} \Rightarrow \alpha_{6/3} \\
\nu^2[5] \leftarrow \nu[9] \\
2(\infty)[11] \leftarrow 1(\infty)[12] 
\end{array}$
\\
\end{tabular}

\begin{tabular}{lll}
{\it Table~\ref{tab:EHPSS}, cont'd}
\\ \\
$\begin{array}{l}
\underline{k=12} \\
\alpha_{6/2}(4)[1] \leftarrow \alpha_{6/3}(4)[2] \\
\epsilon[4,1] \leftarrow \nu^2[6,2] \\
\sigma\eta[4,1] \leftarrow \sigma[5,2] \\
\eta^3[8,2] \leftarrow \bigg\{ \stackem{\sigma\eta^2[4]}{\epsilon\eta[4]} \: (*) \\
\eta\alpha_5[2] \leftarrow \alpha_5[4] \\
\epsilon\eta[3] \leftarrow \epsilon[5] \\
\alpha_5[3] \leftarrow 8\sigma[6] \\
\sigma\eta^2[3] \leftarrow \sigma\eta[5] \\
\eta^2[8,3] \leftarrow \eta[9,4] \\
2\sigma(8)[5] \leftarrow \sigma(8)[6] \\
2\nu(4)[9] \leftarrow \nu(4)[10] \\
\eta^2[10] \leftarrow \eta[12] \\
\eta[11] \leftarrow 1[13]
\end{array}$
&
$\begin{array}{l}
\underline{k=13} \\
\epsilon\eta[4,1] \leftarrow \epsilon[5,2] \\
\alpha_5[4,1] \leftarrow 8\sigma[6,2] \\
\alpha_6[2] \leftarrow \eta\alpha_5[4] \\
\eta^3[8,3] \leftarrow \eta^2[9,4] \\
\eta\alpha_5[3] \leftarrow \alpha_5[5] \\
\eta^2[9,3] \leftarrow \eta[10,4] \\
\eta[10,3] \leftarrow 1[11,4] \\
(\sigma\eta^2+\epsilon\eta)[4] \leftarrow \nu^2[8] \\
\nu^2[7] \leftarrow \nu[11] \\
4\nu[10] \leftarrow \eta^2[12] \\
\eta^2[11] \leftarrow \eta[13] \\
2(\infty)[13] \leftarrow 1(\infty)[14]
\end{array}$
&
$\begin{array}{l}
\underline{k=14} \\
\eta\alpha_5[4,1] \leftarrow \alpha_5[5,2] \\
\sigma\eta[5,2] \leftarrow \nu[9,4] \\
\alpha_{6/3}[3] \leftarrow \eta^2[13] \\
\alpha_{6/2}(4)[3] \leftarrow \alpha_{6/3}(4)[4] \\
\nu[9,3] \leftarrow 1[11,5] \\
(\sigma\eta^2+\epsilon\eta)[5] \leftarrow \nu^2[9] \\
\epsilon\eta[5] \leftarrow \eta[14] \\
\sigma\eta[6] \leftarrow 1[15] \\
\boxed{\epsilon[6]} \Rightarrow \kappa \\
\boxed{\sigma[7]} \Rightarrow \theta_3 \\
2\sigma(8)[7] \leftarrow \sigma(8)[8] \\
2\nu(4)[11] \leftarrow \nu(4)[12]
\end{array}$
\\ \\
$\begin{array}{l}
\underline{k=15} \\
\nu[9,4,1] \leftarrow 1[11,5,2] \\
\boxed{\nu^2[8,2]} \Rightarrow \kappa\eta \\
\boxed{\alpha_6[4]} \Rightarrow \alpha_8 \\
\boxed{\eta\alpha_5[5]} \Rightarrow \alpha_{8/2} \\
\sigma\eta^2[6] \leftarrow \sigma\eta[8] \\
\epsilon\eta[6] \leftarrow \epsilon[8] \\
\boxed{\alpha_5[6]} \Rightarrow \alpha_{8/3} \\
\sigma\eta[7] \leftarrow \sigma[9] \\
\epsilon[7] \leftarrow \nu^2[10] \\
\boxed{8\sigma[8]} \Rightarrow \alpha_{8/4} \\
\boxed{4\nu[12]} \Rightarrow \alpha_{8/5} \\
2(\infty)[15] \leftarrow 1(\infty)[16] 
\end{array}$
&
$\begin{array}{l}
\underline{k=16} \\
\kappa\eta[1] \leftarrow (\sigma\eta^2+\epsilon\eta)[8] \\
\boxed{\eta^2[13,2]} \Rightarrow \eta\alpha_{8/5} \\
\nu^2[8,3] \leftarrow \nu[11,4] \\
\alpha_{6/3}[5] \leftarrow \bigg\{
\stackem{\sigma\eta^2[8]}{\epsilon\eta[8]} \\
\eta\alpha_5[6] \leftarrow \alpha_5[8] \\
\sigma\eta^2[7] \leftarrow \sigma\eta[9] \\
\epsilon\eta[7] \leftarrow \epsilon[9] \\
\alpha_5[7] \leftarrow 8\sigma[10] \\
2\sigma(8)[9] \leftarrow \sigma(8)[10] \\
\boxed{\nu[13]} \Rightarrow \eta_4 \\
2\nu(4)[13] \leftarrow \nu(4)[14] \\
\eta^2[14] \leftarrow \eta[16] \\
\eta[15] \leftarrow 1[17] 
\end{array}$
&
$\begin{array}{l} 
\underline{k=17} \\
\boxed{\eta\alpha_{8/5}[1]} \Rightarrow \eta^2\alpha_{8/5} \\
(\sigma\eta^2+\epsilon\eta)[8,1] \leftarrow \nu^2[9,4] \\
\kappa\eta[2] \leftarrow \kappa[4] \\
\boxed{\alpha_8[2]} \Rightarrow \alpha_9 \\
\eta^2[13,3] \leftarrow \eta[14,4] \\
\nu^2[9,3] \leftarrow \nu[11,5] \\
\eta[14.3] \leftarrow 1[15,4] \\
1[15,3] \leftarrow \theta_3[4] \: (***) \\
\boxed{\kappa[3]} \Rightarrow \kappa\nu \\
\eta[12,5] \leftarrow 1[13,6] \\
\alpha_6[6] \leftarrow \eta\alpha_5[8] \\
\eta\alpha_5[7] \leftarrow \alpha_5[9] \\
\boxed{\nu^2[11]} \Rightarrow \eta\eta_4 \\
4\nu[14] \leftarrow \eta^2[16] \\
\eta^2[15] \leftarrow \eta[17] \\
2(\infty)[17] \leftarrow 1(\infty)[18] 
\end{array}$
\\
\end{tabular}

\begin{tabular}{ll}
{\it Table~\ref{tab:EHPSS}, cont'd} 
\\ \\
$\begin{array}{l}
\underline{k=18} \\
\eta^2\alpha_{8/5}[1] \leftarrow \alpha_6[8] \\
\boxed{\alpha_9[1]} \Rightarrow \eta\alpha_9 \\
\nu^2[9,4,1] \leftarrow \nu[11,5,2] \\
\kappa[4,1] \leftarrow \kappa\nu[2] \: (***) \\
(\sigma\eta^2+\epsilon\eta)[8,2] \leftarrow \kappa\eta[4] \: (*) \\
\eta\alpha_{8/5}[2] \leftarrow \eta\alpha_5[9] \\
\kappa\eta[3] \leftarrow \kappa[5] \\
\alpha_{8/3}(8)[3] \leftarrow \alpha_{8/4}(8)[4] \\
\sigma\eta[8,3] \leftarrow \sigma[9,4] \\
\epsilon[8,3] \leftarrow \nu^2[10,4] \\
\eta^2[13,4] \leftarrow \alpha_5[10] \: (*) \\
\eta^2[12,5] \leftarrow \eta[13,6] \\
\alpha_{6/3}[7] \leftarrow 8\sigma[12] \\
\alpha_{6/2}(4)[7] \leftarrow \alpha_{6/3}(4)[8] \\
\boxed{\sigma\eta^2[9]} \Rightarrow 4\nu^* \\
\epsilon\eta[9] \leftarrow 4\nu[16] \\
\boxed{\sigma\eta[10]} \Rightarrow 2\nu^* \\
\epsilon[10] \leftarrow \eta^2[17] \\
\boxed{\sigma[11]} \Rightarrow \nu^* \\
\nu^2[12] \leftarrow \eta[18] \\
\nu[15] \leftarrow 1[19] \\
2\nu(4)[15] \leftarrow \nu(4)[16]
\end{array}$
&
$\begin{array}{l}
\underline{k=19} \\
\boxed{\eta\alpha_9[1]} \Rightarrow \alpha_{10} \\
\kappa\eta[4,1] \leftarrow \kappa[5,2] \\
\alpha_6[8,1] \leftarrow \eta\alpha_5[9,2] \\
\eta\alpha_5[9,1] \leftarrow \alpha_5[10,2] \\
\eta^2\alpha_{8/5}[2] \leftarrow \eta\alpha_{8/5}[4] \\
\boxed{\alpha_9[2]} \Rightarrow \alpha_{10/2} \\
\alpha_5[8,3] \leftarrow 8\sigma[10,4] \\
\eta\alpha_{8/5}[3] \leftarrow \eta^2[13,6] \: (**) \\
\sigma\eta^2[8,3] \leftarrow \sigma\eta[9,4] \\
\epsilon\eta[8,3] \leftarrow \epsilon[9,4] \\
\boxed{\alpha_8[4]} \Rightarrow \alpha_{10/3} \\
\boxed{\theta_3[5]} \Rightarrow \br{\sigma} \\
\eta^2[13,5] \leftarrow \eta[14,6] \\
\eta[14,5] \leftarrow 1[15,6] \\
1[15,5] \leftarrow \theta_3[6] \: (*) \\
\sigma\eta^2[10] \leftarrow \sigma\eta[12] \\
\epsilon\eta[10] \leftarrow \epsilon[12] \\
\sigma\eta[11] \leftarrow \sigma[13] \\
\epsilon[11] \leftarrow \nu^2[14] \\
\nu^2[13] \leftarrow \nu[17] \\
2(\infty)[19] \leftarrow 1(\infty)[20]
\end{array}$
\\
\end{tabular}
\vfill

\subsection{The GSS for $\pi_{n+1}(S^1)$}\label{tab:S1GSS}  $\quad$

$$
\xymatrix@R-2.7em{
\underline{n} & 
\underline{\pi_n(L(0))} &
\underline{\pi_{n-1}(L(1))} &
\underline{\pi_{n-2}(L(2))} &
\underline{\pi_{n-3}(L(3))} 
\\
0 &
1(\infty) & 1[1] && \\
\ar@{.}[rrrrr] &&&&&
\\
1 & \eta \ar '[]+R [uur]+L & \eta[1] && \\
\ar@{.}[rrrrr] &&&&&
\\
2 & \eta^2 \ar '[]+R [uur]+L & \eta^2[1] & & \\
  && \eta[2] && \\
  && 1[3] & 1[3,1] & \\
\ar@{.}[rrrrr] &&&&&
\\
3 & 4\nu \ar '[]+R [uuuur]+L & \nu[1] \ar '[]+R [ruu]+L && \\
  & 2\nu \ar '[]+R [uuuur]+L &&& \\
  & \nu \ar '[]+R [uuuur]+L &&& \\
\ar@{.}[rrrrr] &&&&&
\\
4 &&&& \\
\ar@{.}[rrrrr] &&&&& 
\\
5 & & \nu[3] & \nu[3,1] & \\
\ar@{.}[rrrrr] &&&&&
\\
6 & \nu^2 \ar '[]+R [uur]+L & \nu^2[1] \ar '[]+R [uur]+L && \\
  && \eta^3[4] && \\
  && \eta^2[5] && \\
  && \eta[6] && \\
  && 1[7] & 1[7,1] & \\
\ar@{.}[rrrrr] &&&&& 
\\
7 & 8\sigma \ar '[]+R [uuuuur]+L & \sigma[1] \ar '[]+R [uur]+L & &\\
  & 4\sigma \ar '[]+R [uuuuur]+L &&& \\
  & 2\sigma \ar '[]+R [uuuuur]+L & \nu^2[2] && \\
  & \sigma \ar '[]+R [uuuuur]+L & \nu[5] & \nu[5,1] & \\
\ar@{.}[rrrrr] &&&&&
\\
8 & \epsilon \ar '[]+R [ruuu]+L & \sigma\eta[1] \ar '[]+R [ruu]+L && \\
  & \sigma\eta \ar '[]+R [ruuu]+L & \epsilon[1] & \nu^2[3,1] & \\
  && \nu^2[3] & \nu[5,2] & \\
  && 8\sigma[2] & 1[7,3] & 1[7,3,1] \\
\ar@{.}[rrrrr] &&&&&
\\
9 & \epsilon\eta \ar '[]+R [ruuuu]+L & \sigma\eta^2[1] \ar '[]+R
  [ruuuu]+L & \sigma[3,1] \ar '[]+R [ruu]+L \\
  & \sigma\eta^2 \ar '[]+R [ruuuu]+L & \sigma\eta[2] \ar '[]+R [ruuuu]+L &&
  \\
  & \alpha_5 \ar '[]+R [ruuuu]+L & \sigma[3] \ar '[]+R [ruuuu]+L && \\
  && \alpha_5[1] && \\
\ar@{.}[rrrrr] &&&&&
\\
10 & \eta\alpha_5 \ar '[]+R [ruu]+L & \eta\alpha_5[1] && \\
   && \alpha_5[2] && \\
   && 8\sigma[4] && \\
\ar@{.}[rrrrr] &&&&&
\\
11 & \alpha_6 \ar '[]+R [ruuuu]+L &&& \\
   & \alpha_{6/2} \ar '[]+R [ruuuu]+L &&& \\
   & \alpha_{6/3} \ar '[]+R [ruuuu]+L &&& \\
\ar@{.}[rrrrr] &&&&&
\\
12 &&&& \\
\ar@{.}[rrrrr] &&&&&
\\
13 && \epsilon[6] & \sigma[7,1] & \\
   && \sigma[7] && \\
\ar@{.}[rrrrr] &&&&&
\\
   & \kappa \ar '[]+R [ruuu]+L & \theta_3[1] \ar '[]+R [ruuu]+L && \\
   & \theta_3 \ar '[]+R [ruuu]+L &&& \\
} $$ 

$$
\xymatrix@R-2.7em{
\underline{n} & 
\underline{\pi_n(L(0))} &
\underline{\pi_{n-1}(L(1))} &
\underline{\pi_{n-2}(L(2))} &
\underline{\pi_{n-3}(L(3))} 
\\
   && \epsilon[6] & & \\
   && \sigma[7] & \sigma[7,1] & \\
\ar@{.}[rrrrr] &&&&&
\\
14 & \kappa \ar '[]+R [ruuu]+L & \theta_3[1] \ar '[]+R [ruu]+L & & \\
   & \theta_3 \ar'[]+R [ruuu]+L & \kappa[1] & & \\
   && \alpha_6[4] && \\
   && \eta\alpha_5[5] && \\
   && \alpha_5[6] && \\
   && 8\sigma[8] & \sigma[7,2] & \\
   && \eta^3[12] & \epsilon[6,2] & \\
\ar@{.}[rrrrr] &&&&&
\\
15 & \eta\kappa \ar '[]+R [ruuuuuuu]+L & \theta_3[2] \ar '[]+R [ruuu]+L & & \\
   & \alpha_8 \ar '[]+R [ruuuuuuu]+L & \kappa[2] \ar '[]+R [ruuu]+L & & \\
   & \alpha_{8/2} \ar '[]+R [ruuuuuuu]+L & && \\
   & \alpha_{8/3} \ar '[]+R [ruuuuuuu]+L & && \\
   & \alpha_{8/4} \ar '[]+R [ruuuuuuu]+L & \nu[13] & \sigma[7,3] & \\
   & \alpha_{8/5} \ar '[]+R [ruuuuuuu]+L & \alpha_{8/5}[1] & \nu[13,1] & 
   \sigma[7,3,1] \\
\ar@{.}[rrrrr] &&&&&
\\
16 & \eta_4 \ar '[]+R [ruuu]+L & \theta_3[3] \ar '[]+R [ruuu]+L & 
    \theta_3[3,1] \ar '[]+R [ruu]+L & \\
   & \eta\alpha_{8/5} \ar '[]+R [ruuu]+L & \eta_4[1] \ar '[]+R [ruuu]+L & & \\
   && \nu^2[11] & & \\
   && \eta\alpha_{8/5}[1] & \nu^2[11,1] & \\
   && \kappa[3] & \nu[13,2] & \\
   && \alpha_8[2] & 1[15,3] & 1[15,3,1] \\
\ar@{.}[rrrrr] &&&&&
\\
17 & \eta\eta_4 \ar '[]+R [ruuuuu]+L & \eta\eta_4[1] \ar '[]+R [ruuuu]+L &
    \theta_3[4,1] \ar@{-->} '[]+R [ruu]+L & \\
   & \eta^2\alpha_{8/5} \ar '[]+R [ruuuuu]+L & \eta_4[2] \ar '[]+R [ruuuu]+L 
    && \\
   & \nu\kappa \ar '[]+R [ruuuuu]+L & \theta_3[4] \ar@{-->} '[]+R [ruuuu]+L&& \\
   & \alpha_9 \ar '[]+R [ruuuuu]+L & \sigma\eta^2[9] && \\
   && \sigma\eta[10] && \\
   && \sigma[11] & \kappa[4,1] & \\
   && \alpha_9[1] & \sigma[11,1] & \\
\ar@{.}[rrrrr] &&&&&
\\
18 & 4\nu^* \ar '[]+R [ruuuuu]+L & \nu\kappa[2] \ar@{-->} '[]+R [ruuu]+L && \\
   & 2\nu^* \ar '[]+R [ruuuuu]+L & \nu^*[1] \ar '[]+R [ruuu]+L && \\
   & \nu^* \ar '[]+R [ruuuuu]+L & \theta_3[5] && \\
   & \eta\alpha_9 \ar '[]+R [ruuuuu]+L & \eta\alpha_9[1] && \\
   && \alpha_9[2] && \\
   && \alpha_8[4] & \theta_3[5,1] & \\
\ar@{.}[rrrrr] &&&&&
\\
19 & \br{\sigma} \ar '[]+R [ruuuuu]+L & \br{\sigma}[1] \ar '[]+R [ruu]+L && \\
   & \alpha_{10} \ar '[]+R [ruuuuu]+L & \kappa\nu[3] & \kappa[6,1] & \\
   & \alpha_{10/2} \ar '[]+R [ruuuuu]+L & \kappa\eta[5] & \theta_3[5,2] & \\
   & \alpha_{10/3} \ar '[]+R [ruuuuu]+L & \kappa[6] & \sigma[11,3] &
   \sigma[11,3,1] \\
\ar@{.}[rrrrr] &&&&&
\\
   & 4\br{\kappa} \ar '[]+R [ruuuu]+L & \br{\kappa}[1] \ar '[]+R [ruuuu]+L
   & \nu^*[3,1] \ar '[]+R [ruu]+L & \\
   & 2\br{\kappa} \ar '[]+R [ruuuu]+L & \br{\sigma}[2] \ar '[]+R [ruuuu]+L && \\
   & \br{\kappa} \ar '[]+R [ruuuu]+L & \nu^*[3]\ar '[]+R [ruuuu]+L &&
}$$

\subsection{The GSS for $\pi_{n+2}(S^2)$}\label{tab:S2GSS} $\quad$

$$
\xymatrix@R-2.7em{
\underline{n} & 
\underline{\pi_n(L(0)_2)} &
\underline{\pi_{n-1}(L(1)_2)} &
\underline{\pi_{n-2}(L(2)_2)} &
\underline{\pi_{n-3}(L(3)_2)} 
\\
0 & 1(\infty) &&& \\
\ar@{.}[rrrrr] &&&&&
\\
1 & \eta & 1(\infty)[2] && \\
\ar@{.}[rrrrr] &&&&&
\\
2 && \eta[2] && \\
  & \eta^2 & 1[3] && \\
\ar@{.}[rrrrr] &&&&& 
\\
3 & 2\nu \ar '[]+R [ruuu]+L &&& \\
  & \nu \ar'[]+R [ruuu]+L &&& \\
  & 4\nu &&& \\
\ar@{.}[rrrrr] &&&&&
\\
4 && 2\nu[2] && \\
  && \nu[2] && \\
\ar@{.}[rrrrr] &&&&&
\\
5 && \nu[3] & 1[5,2] & \\
\ar@{.}[rrrrr] &&&&&
\\
6 & \nu^2 \ar '[]+R [ruu] & \eta^3[4] && \\
  && \eta^2[5] && \\
  && \eta[6] && \\
  && 1[7] & \eta[5,2] & \\
\ar@{.}[rrrrr] &&&&&
\\
7 & 8\sigma \ar '[]+R [ruuuuu]+L &&& \\
  & 4\sigma \ar '[]+R [ruuuuu]+L && \eta^2[5,2] & \\
  & 2\sigma \ar '[]+R [ruuuuu]+L & \nu^2[2] & \eta[6,2] & \\
  & \sigma \ar '[]+R [ruuuuu]+L & \nu[5] & 1[7,2] & \\
\ar@{.}[rrrrr] &&&&&
\\
8 & \epsilon \ar '[]+R [ruuu]+L & 4\sigma[2] \ar '[]+R [ruuuu]+L && \\
  & \sigma\eta \ar '[]+R [ruuu]+L & 2\sigma[2] \ar '[]+R [ruuuu]+L && \\
  && \sigma[2] \ar '[]+R [ruuuu]+L && \\
  && 8\sigma[2] & \nu[5,2] & \\
  && \nu^2[3] & 1[7,3] & \\
\ar@{.}[rrrrr] &&&&&
\\
9 & \alpha_5 \ar '[]+R [ruuu]+L & \sigma\eta[2] \ar '[]+R [ruuu]+L && \\
  & \sigma\eta^2 \ar '[]+R [ruuu]+L & \sigma[3] \ar '[]+R [ruuu]+L && \\
  & \epsilon\eta &&& \\
\ar@{.}[rrrrr] &&&&&
\\
10 && 4\nu[8] && \\
   && \alpha_5[2] && \\
   & \eta\alpha_5 & 8\sigma[4] && \\
\ar@{.}[rrrrr] &&&&&
\\
11 & \alpha_{6/2} \ar '[]+R [ruuu]+L && \eta^2[9,2] & \\
   & \alpha_{6/3} \ar '[]+R [ruuu]+L && \eta[10,2] & \\
   & \alpha_6 && \eta^3[8,2] & \\
\ar@{.}[rrrrr] &&&&&
\\
12 && \epsilon\eta[4] \ar '[]+R [ruu]+L|{(*)} && \\
   && \alpha_{6/2}[2] & \sigma[5,2] & \\
   && \alpha_{6/3}[2] & \nu^2[6,2] & \\
} $$ 

$$
\xymatrix@R-2.7em{
\ar@{.}[rrrrr] &&&&&
\\
13 && \epsilon[6] & \epsilon[5,2] & \\
   && \sigma[7] & 8\sigma[6,2] & \\
\ar@{.}[rrrrr] &&&&&
\\
14 & \kappa \ar '[]+R [ruuu]+L & \alpha_6[4] & \alpha_5[5,2] & \\
   & \theta_3 \ar '[]+R [ruuu]+L & \eta\alpha_5[5] && \\
   && \alpha_5[6] & \epsilon[6,2] & \\
   && 8\sigma[8] & \sigma[7,2] & \\
   && 4\nu[12] & \nu^2[8,2] & \\
\ar@{.}[rrrrr] &&&&&
\\
15 & \alpha_8 \ar '[]+R [ruuuuuu]+L & \kappa[2] \ar '[]+R [ruuuu]+L &
    \eta\alpha_5[5,2] & \\
   & \alpha_{8/2} \ar '[]+R [ruuuuuu]+L & \theta_3[2] \ar '[]+R [ruuuu]+L &
    \alpha_5[6,2] & \\
   & \alpha_{8/3} \ar '[]+R [ruuuuuu]+L && 8\sigma[8,2] & \\
   & \alpha_{8/4} \ar '[]+R [ruuuuuu]+L && \eta^3[12,2] & \\
   & \alpha_{8/5} \ar '[]+R [ruuuuuu]+L && \sigma[7,3] & \\
   & \kappa\eta \ar '[]+R [rruuuuuuu]+L & \nu[13] & \eta^2[13,2] & 
    1[11,5,2] \\
\ar@{.}[rrrrr] &&&&&
\\
16 & \eta_4 \ar '[]+R [ruu]+L & \alpha_{8/2}[2] \ar '[]+R [ruuuuuuu]+L && \\
   & \eta\alpha_{8/5} \ar '[]+R
    [rruuu]+L
    & \alpha_{8/3}[2] \ar '[]+R [ruuuuuuu]+L && \\
   && \alpha_{8/4}[2] \ar '[]+R [ruuuuuuu]+L && \\
   && \alpha_{8/5}[2] \ar '[]+R [ruuuuuuu]+L && \\
   && \theta_3[3] \ar '[]+R [ruuuuuuu]+L && \\
   && (\sigma\eta^2+\epsilon\eta)[8] && \\
   && \alpha_8[2] && \\
   && \kappa[3] & \nu[13,2] & \\
   && \nu^2[11] & 1[15,3] & \\
\ar@{.}[rrrrr] &&&&&
\\
17 & \alpha_9 \ar '[]+R [ruuuu]+L & \eta_4[2] \ar '[]+R [ruuu]+L &
    \nu^2[9,4] & \\
   & \nu\kappa \ar '[]+R [ruuuu]+L & \theta_3[4] \ar '[]+R [ruuu]+L|{(**)} && \\
   & \eta\eta_4 \ar '[]+R [ruuuu]+L & \sigma\eta^2[9] && \\
   && \sigma\eta[10] && \\
   & \eta^2\alpha_{8/5} & \sigma[11] & (\sigma\eta^2+\epsilon\eta)[8,2] & \\
\ar@{.}[rrrrr] &&&&&
\\
18 & 4\nu^* \ar '[]+R [ruuuu]+L & \kappa\eta[4] \ar '[]+R [ruu]+L|{(*)} && \\
   & 2\nu^* \ar '[]+R [ruuuu]+L & \nu\kappa[2] && \\
   & \nu^* \ar '[]+R [ruuuu]+L & \alpha_6[8] && \\
   && \alpha_9[2] && \\
   && \alpha_8[4] & \sigma\eta[10,2] & \\
   & \eta\alpha_9 & \theta_3[5] & \sigma[11,2] & \nu[11,5,2] \\
\ar@{.}[rrrrr] &&&&&
\\
19 & \alpha_{10/2} \ar '[]+R [ruuuu]+L & 2\nu^*[2] \ar '[]+R [ruuu]+L &
    \kappa[5,2] & \\
   & \alpha_{10/3} \ar '[]+R [ruuuu]+L & \nu^*[2] \ar '[]+R [ruuu]+L &
    \eta\alpha_5[9,2] & \\
   & \br{\sigma} \ar '[]+R [ruuuu]+L & \nu\kappa[3] & \alpha_5[10,2] & \\
   && \kappa\eta[5] & \alpha_6[8,2] & \\
   & \alpha_{10} & \kappa[6] & \sigma[11,3] & 1[15,5,2] \\
\ar@{.}[rrrrr] &&&&&
\\
   & 4\br{\kappa} \ar '[]+R [ruuuu]+L & \eta^2\alpha_{8/5}[4] \ar '[]+R 
    [ruuu]+L|{(*)} & \theta_3[6,2] \ar '[]+R [ruu]+L|{(*)} & \\
   & 2\br{\kappa} \ar '[]+R [ruuuu]+L & \nu^*[3] \ar '[]+R [ruuu]+L && \\
   & \br{\kappa} \ar '[]+R [ruuuu]+L &&&
}$$

\subsection{The GSS for $\pi_{n+3}(S^3)$}\label{tab:S3GSS} $\quad$

$$
\xymatrix@R-2.7em{
\underline{n} & 
\underline{\pi_n(L(0)_3)} &
\underline{\pi_{n-1}(L(1)_3)} &
\underline{\pi_{n-2}(L(2)_3)} 
\\
0 & 1(\infty) && \\
\ar@{.}[rrrr] &&&&
\\
1 & \eta & & \\
\ar@{.}[rrrr] &&&&
\\
2 & \eta^2 & 1[3] && \\
\ar@{.}[rrrr] &&&& 
\\
3 & \nu \ar '[]+R [ruu]+L && \\
  & 4\nu  && \\
  & 2\nu && \\
\ar@{.}[rrrr] &&&&
\\
4 && \eta[4] & \\
\ar@{.}[rrrr] &&&&
\\
5 && \eta^2[4] &  \\
  & & \nu[3] & \\
\ar@{.}[rrrr] &&&&
\\
6 & \nu^2 \ar '[]+R [ruu] & \eta^3[4] & \\
  && \eta^2[5] & \\
  && \eta[6] & \\
  && 1[7] & \\
\ar@{.}[rrrr] &&&&
\\
7 & 8\sigma \ar '[]+R [ruuuuu]+L && \\
  & 4\sigma \ar '[]+R [ruuuuu]+L && \\
  & 2\sigma \ar '[]+R [ruuuuu]+L && \\
  & \sigma \ar '[]+R [ruuuuu]+L & \nu[5] & \\
\ar@{.}[rrrr] &&&&
\\
8 & \sigma\eta \ar '[]+R [ruu]+L && \\
  & \epsilon  & \nu^2[3] & 1[7,3]  \\
\ar@{.}[rrrr] &&&&
\\
9 & \sigma\eta^2 \ar '[]+R [ruu]+L & \sigma[3] \ar '[]+R [ruu]+L & \\
  &  \alpha_5 && \\
  & \epsilon\eta && \\
\ar@{.}[rrrr] &&&&
\\
10 && 4\nu[8] & \\
   && \eta^2[9] & \\
   & \eta\alpha_5 & 8\sigma[4] & \\
\ar@{.}[rrrr] &&&&
\\
11 & \alpha_{6/3} \ar '[]+R [ruu]+L && \\
   & \alpha_{6} & \sigma\eta[4] &  \\
   & \alpha_{6/2} & \epsilon[4] &  \\
\ar@{.}[rrrr] &&&&
\\
12 && \epsilon\eta[4] & \\
   && \alpha_{5}[4] & \\
} $$ 
\vspace{3in}

$$
\xymatrix@R-2.7em{
\ar@{.}[rrrr] &&&&
\\
13 
   && \eta\alpha_5[4] \\
   && \epsilon[6] & & \\
   && \sigma[7] & & \\
\ar@{.}[rrrr] &&&&
\\
14 & \kappa \ar '[]+R [ruuu]+L & \alpha_6[4] & \nu[9,4]  \\
   & \theta_3 \ar '[]+R [ruuu]+L & \eta\alpha_5[5] & \\
   && \alpha_5[6] &  \\
   && 8\sigma[8] &  \\
   && 4\nu[12] &  \\
\ar@{.}[rrrr] &&&&
\\
15 & \alpha_8 \ar '[]+R [ruuuuuu]+L &  & \\
   & \alpha_{8/2} \ar '[]+R [ruuuuuu]+L & & \\
   & \alpha_{8/3} \ar '[]+R [ruuuuuu]+L &&  \\
   & \alpha_{8/4} \ar '[]+R [ruuuuuu]+L &&  \\
   & \alpha_{8/5} \ar '[]+R [ruuuuuu]+L &&  \\
   & \kappa\eta  & \nu[13] & \sigma[7,3]  \\
\ar@{.}[rrrr] &&&&
\\
16 & \eta_4 \ar '[]+R [ruu]+L & \theta_3[3] \ar '[]+R [ruu]+L & \\
   & \eta\alpha_{8/5} & (\sigma\eta^2+\epsilon\eta)[8] & \\
   && \kappa[3] &  \\
   && \nu^2[11] & 1[15,3]  \\
\ar@{.}[rrrr] &&&&
\\
17 & \nu\kappa \ar '[]+R [ruuu]+L & \theta_3[4] \ar '[]+R [ruu]+L|{(**)} & \nu^2[9,4]  \\
   & \eta\eta_4 \ar '[]+R [ruuu]+L & \kappa[4] & \\
   &  & \sigma\eta^2[9] & \\
   &\alpha_9 & \sigma\eta[10] & \\
   & \eta^2\alpha_{8/5} & \sigma[11] &  \\
\ar@{.}[rrrr] &&&&
\\
18 & 4\nu^* \ar '[]+R [ruuuu]+L & \kappa\eta[4] & \\
   & 2\nu^* \ar '[]+R [ruuuu]+L & \alpha_6[8] & \\
   & \nu^* \ar '[]+R [ruuuu]+L & \eta\alpha_5 [9] & \\
   && \alpha_8[4] & \\
   & \eta\alpha_9 & \theta_3[5] & \nu[13,4] \\
\ar@{.}[rrrr] &&&&
\\
19 & \alpha_{10/3} \ar '[]+R [ruuu]+L & \eta_4[4] \ar '[]+R [ruu]+L & \\
   & \br{\sigma} \ar '[]+R [ruuu]+L & \eta\alpha_{8/5}[4] & \\
   && \nu\kappa[3] &  \\
   & \alpha_{10} & \kappa\eta[5] & \sigma[11,3] \\
   & \alpha_{10/2} & \kappa[6] & \nu^2[11,4] \\
\ar@{.}[rrrr] &&&&
\\
   & 4\br{\kappa} \ar '[]+R [ruuuu]+L & \nu^*[3] \ar '[]+R  [ruuu]+L &  \\
   & 2\br{\kappa} \ar '[]+R [ruuuu]+L &  \eta\eta_4[4] \ar '[]+R [ruuu]+L & \\
   & \br{\kappa} \ar '[]+R [ruuuu]+L &&
}$$
\vfill

\subsection{The GSS for $\pi_{n+4}(S^4)$}\label{tab:S4GSS} $\quad$

$$
\xymatrix@R-2.7em{
\underline{n} & 
\underline{\pi_n(L(0)_4)} &
\underline{\pi_{n-1}(L(1)_4)} &
\underline{\pi_{n-2}(L(2)_4)} 
\\
0 & 1(\infty) && \\
\ar@{.}[rrrr] &&&&
\\
1 & \eta & & \\
\ar@{.}[rrrr] &&&&
\\
2 & \eta^2 & && \\
\ar@{.}[rrrr] &&&& 
\\
3 & 4\nu  && \\
  & 2\nu  & 1(\infty)[4] & \\
  & \nu && \\
\ar@{.}[rrrr] &&&&
\\
4 && \eta[4] & \\
 && 1[5] & \\
\ar@{.}[rrrr] &&&&
\\
5 && \eta^2[4] &  \\
  & & \eta[5] & \\
\ar@{.}[rrrr] &&&&
\\
6 & \nu^2 & 2\nu[4] & \\
&& \nu[4] & \\
&& \eta^3[4] & \\
  && \eta^2[5] & \\
  && \eta[6] & \\
  && 1[7] & \\
\ar@{.}[rrrr] &&&&
\\
7 & 8\sigma \ar '[]+R [ruuuuu]+L && \\
  & 4\sigma \ar '[]+R [ruuuuu]+L && \\
  & 2\sigma \ar '[]+R [ruuuuu]+L && \\
  & \sigma \ar '[]+R [ruuuuu]+L & \nu[5] & \\
\ar@{.}[rrrr] &&&&
\\
8 & \sigma\eta \ar '[]+R [ruu]+L && \\
  & \epsilon  & &   \\
\ar@{.}[rrrr] &&&&
\\
9 & \sigma\eta^2  &  & \\
  &  \alpha_5 && \\
  & \epsilon\eta && \\
\ar@{.}[rrrr] &&&&
\\
10 && 4\sigma[4] & \\
&& 2\sigma[4] & \\
&& \sigma[4] & \\
&& 4\nu[8] & \\
   && \eta^2[9] & \\
   & \eta\alpha_5 & 8\sigma[4] & \\
\ar@{.}[rrrr] &&&&
\\
11 & \alpha_{6/3} \ar '[]+R [ruu]+L & \sigma\eta[4] & \\
  && \epsilon[4] & \\
   & \alpha_{6} & \sigma[5] &  \\
   & \alpha_{6/2} & \nu^2[6] &  \\
\ar@{.}[rrrr] &&&&
\\
12 && \epsilon\eta[4] & \eta[9,4] \\
   && \alpha_{5}[4] & \\
  && \sigma\eta[5] & \\
&& \epsilon[5] & \\
&& 8\sigma[6] & \\
} $$ 
\vspace{0in}

$$
\xymatrix@R-2.7em{
\ar@{.}[rrrr] &&&&
\\
13 
   && \eta\alpha_5[4] & \eta^2[9,4] \\
  && \alpha_5[5] & \eta[10,4] \\
   && \epsilon[6] &  1[11,4] \\
   && \sigma[7] &  \\
\ar@{.}[rrrr] &&&&
\\
14 & \kappa \ar '[]+R [ruuu]+L & \alpha_{6/2}[4] & \nu[9,4]  \\
   & \theta_3 \ar '[]+R [ruuu]+L &  \alpha_{6/3}[4] & 1[11,5] \\
&& \eta^2[13] & \\
&&  \alpha_6[4] & \\
&& \eta\alpha_5[5] & \\
   && \alpha_5[6] &  \\
   && 8\sigma[8] &  \\
   && 4\nu[12] &  \\
\ar@{.}[rrrr] &&&&
\\
15 & \alpha_8 \ar '[]+R [ruuuuuu]+L &  & \\
   & \alpha_{8/2} \ar '[]+R [ruuuuuu]+L & & \\
   & \alpha_{8/3} \ar '[]+R [ruuuuuu]+L &&  \\
   & \alpha_{8/4} \ar '[]+R [ruuuuuu]+L &&  \\
   & \alpha_{8/5} \ar '[]+R [ruuuuuu]+L &&  \\
   & \kappa\eta  & \nu[13] &  \\
\ar@{.}[rrrr] &&&&
\\
16 & \eta_4 \ar '[]+R [ruu]+L & (\sigma\eta^2+\epsilon\eta)[8] & \nu[11,4] \\
   & \eta\alpha_{8/5} & \nu^2[11] & \\
\ar@{.}[rrrr] &&&&
\\
17 & \eta\eta_4 \ar '[]+R [ruu]+L & \theta_3[4]  & \nu^2[9,4]  \\
   &   & \kappa[4] & \eta[14,4] \\
&&& 1[15,4] \\
   &  \nu\kappa & \sigma\eta^2[9] & \nu[11,5] \\
   &\alpha_9 & \sigma\eta[10] &  \eta^3[12,4] \\
   & \eta^2\alpha_{8/5} & \sigma[11] &  \eta^2[13,4] \\
\ar@{.}[rrrr] &&&&
\\
18 & 4\nu^* \ar '[]+R [ruuuu]+L & \alpha_{8/5}[4] \ar '[]+R [ruuu]+L & \sigma[9,4] \\
   & 2\nu^* \ar '[]+R [ruuuu]+L & \alpha_5[10] \ar '[]+R [ruuu]+L|{(*)} & \nu^2[10,4] \\
   & \nu^* \ar '[]+R [ruuuu]+L & \alpha_{8/4}(8)[4] & \\
&& \kappa\eta[4] & \\
&& \kappa[5] & \\
&& \alpha_6[8] & \\
&& \eta\alpha_5[9] & \\
   && \alpha_8[4] & \nu[13,4] \\
   & \eta\alpha_9 & \theta_3[5] & 1[15,5] \\
\ar@{.}[rrrr] &&&&
\\
19 & \alpha_{10/3} \ar '[]+R [ruuu]+L & \eta_4[4] \ar '[]+R [ruuu]+L & \sigma\eta[9,4] \\
   & \br{\sigma} \ar '[]+R [ruuu]+L & \theta_3[6] \ar '[]+R [ruuu]+L|{(*)} & \epsilon[9,4] \\
   &&  \eta\alpha_{8/5}[4] &  8\sigma[10,4] \\
&& \alpha_{8/5}[5] \\
   & \alpha_{10} & \kappa\eta[5] & \nu^2[11,4] \\
   & \alpha_{10/2} & \kappa[6] & \nu[13,5] \\
\ar@{.}[rrrr] &&&&
\\
   & 2\br{\kappa} \ar '[]+R [ruuu]+L & \eta\eta_4[4] \ar '[]+R  [ruuu]+L &  \\
   & \br{\kappa} \ar '[]+R [ruuu]+L &  \eta_4[5] \ar '[]+R [ruuu]+L &
}$$

\subsection{The GSS for $\pi_{n+5}(S^5)$}\label{tab:S5GSS} $\quad$

$$
\xymatrix@R-2.7em{
\underline{n} & 
\underline{\pi_n(L(0)_5)} &
\underline{\pi_{n-1}(L(1)_5)} &
\underline{\pi_{n-2}(L(2)_5)} 
\\
0 & 1(\infty) && \\
\ar@{.}[rrrr] &&&&
\\
1 & \eta & & \\
\ar@{.}[rrrr] &&&&
\\
2 & \eta^2 & && \\
\ar@{.}[rrrr] &&&& 
\\
3 & 4\nu  && \\
  & 2\nu  &  & \\
  & \nu && \\
\ar@{.}[rrrr] &&&&
\\
4 && 1[5] & \\
\ar@{.}[rrrr] &&&&
\\
5 && \eta[5] &  \\
\ar@{.}[rrrr] &&&&
\\
6  & \nu^2 & \eta^2[5] & \\
  && \eta[6] & \\
  && 1[7] & \\
\ar@{.}[rrrr] &&&&
\\
7 & 8\sigma  && \\
  & 4\sigma \ar '[]+R [ruuuuu]+L && \\
  & 2\sigma \ar '[]+R [ruuuuu]+L && \\
  & \sigma \ar '[]+R [ruuuuu]+L & \nu[5] & \\
\ar@{.}[rrrr] &&&&
\\
8 & \sigma\eta \ar '[]+R [ruu]+L && \\
  & \epsilon  & &   \\
\ar@{.}[rrrr] &&&&
\\
9 & \sigma\eta^2  &  & \\
  &  \alpha_5 && \\
  & \epsilon\eta && \\
\ar@{.}[rrrr] &&&&
\\
10 && 4\nu[8] & \\
   && \eta^2[9] & \\
   & \eta\alpha_5 & \eta[10] & \\
\ar@{.}[rrrr] &&&&
\\
11   & \alpha_{6} & \sigma[5] &  \\
   & \alpha_{6/2} & \nu^2[6] &  \\
&\alpha_{6/3} && \\
\ar@{.}[rrrr] &&&&
\\
12  && \sigma\eta[5] & \\
&& \epsilon[5] & \\
&& 8\sigma[6] & \\
} $$ 
\vspace{2in}

$$
\xymatrix@R-2.7em{
\ar@{.}[rrrr] &&&&
\\
13 
   && \alpha_5[5] &  \\
  && \nu^2[8] &  \\
   && \epsilon[6] &   \\
   && \sigma[7] &  \\
\ar@{.}[rrrr] &&&&
\\
14 & \kappa \ar '[]+R [ruuu]+L & \eta^2[13] &  \\
   & \theta_3 \ar '[]+R [ruuu]+L &  \eta\alpha_5[5] & 1[11,5] \\
   && \alpha_5[6] &  \\
   && 8\sigma[8] &  \\
   && 4\nu[12] &  \\
\ar@{.}[rrrr] &&&&
\\
15 & \alpha_{8/2} \ar '[]+R [ruuuuu]+L & & \\
   & \alpha_{8/3} \ar '[]+R [ruuuuu]+L &&  \\
   & \alpha_{8/4} \ar '[]+R [ruuuuu]+L &&  \\
   & \alpha_{8/5} \ar '[]+R [ruuuuu]+L &&  \\
&\alpha_8 && \\
   & \kappa\eta  & \nu[13] &  \\
\ar@{.}[rrrr] &&&&
\\
16 & \eta_4 \ar '[]+R [ruu]+L & (\sigma\eta^2+\epsilon\eta)[8] &  \\
   & \eta\alpha_{8/5} & \nu^2[11] & \\
\ar@{.}[rrrr] &&&&
\\
17 & \eta\eta_4 \ar '[]+R [ruu]+L &   &   \\
   &  \nu\kappa & \sigma\eta^2[9] & \nu[11,5] \\
   &\alpha_9 & \sigma\eta[10] &   \\
   & \eta^2\alpha_{8/5} & \sigma[11] &   \\
\ar@{.}[rrrr] &&&&
\\
18 & 4\nu^* \ar '[]+R [ruuuu]+L & \kappa[5] &  \\
   & 2\nu^* \ar '[]+R [ruuuu]+L & \alpha_6[8]  &  \\
   & \nu^* \ar '[]+R [ruuuu]+L & \eta\alpha_5[9] & \\
&& \alpha_5[10] & \\
   & \eta\alpha_9 & \theta_3[5] & 1[15,5] \\
\ar@{.}[rrrr] &&&&
\\
19 & \br{\sigma} \ar '[]+R [ruu]+L & \theta_3[6] \ar '[]+R [ruu]+L|{(*)}  \\
   & \alpha_{10} & \alpha_{8/5}[5] &  \\
   & \alpha_{10/2} & \kappa\eta[5] &  \\
 & \alpha_{10/3} & \kappa[6] & \nu[13,5] \\
\ar@{.}[rrrr] &&&&
\\
   & 2\br{\kappa} \ar '[]+R [ruuu]+L & \eta_4[5] \ar '[]+R  [ruu]+L &  \\
   & \br{\kappa} \ar '[]+R [ruuu]+L &    &
}$$
\vfill

\subsection{The GSS for $\pi_{n+6}(S^6)$}\label{tab:S6GSS} $\quad$

$$
\xymatrix@R-2.7em{
\underline{n} & 
\underline{\pi_n(L(0)_6)} &
\underline{\pi_{n-1}(L(1)_6)} &
\underline{\pi_{n-2}(L(2)_6)} 
\\
0 & 1(\infty) && \\
\ar@{.}[rrrr] &&&&
\\
1 & \eta & & \\
\ar@{.}[rrrr] &&&&
\\
2 & \eta^2 & && \\
\ar@{.}[rrrr] &&&& 
\\
3 & 4\nu  && \\
  & 2\nu  &  & \\
  & \nu && \\
\ar@{.}[rrrr] &&&&
\\
4 &&& \\
\ar@{.}[rrrr] &&&&
\\
5 && 1(\infty)[6] &  \\
\ar@{.}[rrrr] &&&&
\\
6  & \nu^2 & \eta[6] & \\
  && 1[7] & \\
\ar@{.}[rrrr] &&&&
\\
7 & 8\sigma  && \\
  & 4\sigma  && \\
  & 2\sigma \ar '[]+R [ruuuuu]+L && \\
  & \sigma \ar '[]+R [ruuuuu]+L &  & \\
\ar@{.}[rrrr] &&&&
\\
8 & \sigma\eta  & 2\nu[6] & \\
  & \epsilon  & \nu[6] &   \\
\ar@{.}[rrrr] &&&&
\\
9 & \sigma\eta^2  &  & \\
  &  \alpha_5 && \\
  & \epsilon\eta && \\
\ar@{.}[rrrr] &&&&
\\
10 && 4\nu[8] & \\
   && \eta^2[9] & \\
   & \eta\alpha_5 & \eta[10] & \\
\ar@{.}[rrrr] &&&&
\\
11   & \alpha_{6} & \nu^2[6] &  \\
   & \alpha_{6/2} & \nu[9] &  \\
&\alpha_{6/3} && \\
\ar@{.}[rrrr] &&&&
\\
12  && 8\sigma[6] & \\
&& 4\sigma[6] & \\
&& 2\sigma[6] & \\
&& \sigma[6] & \\
} $$ 
\vspace{3in}

$$
\xymatrix@R-2.7em{
\ar@{.}[rrrr] &&&&
\\
13 
  && \nu^2[8] &  \\
   && \epsilon[6] &   \\
   && \sigma[7] &  \\
\ar@{.}[rrrr] &&&&
\\
14 & \kappa \ar '[]+R [ruuu]+L & \nu^2[9] & \\
   & \theta_3 \ar '[]+R [ruuu]+L &  \eta^2[13] & \\
  && \eta[14] & \\
   && \alpha_5[6] &  \\
   && 8\sigma[8] &  \\
   && 4\nu[12] &  \\
\ar@{.}[rrrr] &&&&
\\
15 
&\alpha_8 && \\
& \alpha_{8/2} & & \\
   & \alpha_{8/3} \ar '[]+R [ruuuuuu]+L &&  \\
   & \alpha_{8/4} \ar '[]+R [ruuuuuu]+L &&  \\
   & \alpha_{8/5} \ar '[]+R [ruuuuuu]+L &&  \\
   & \kappa\eta  & \nu[13] &  \\
\ar@{.}[rrrr] &&&&
\\
16 & \eta_4 \ar '[]+R [ruu]+L & \alpha_{6/2}[6] & \\
&& \alpha_{6/3}[6] & \\
&& \sigma\eta^2[8] &  \\
&& \epsilon\eta[8] & \\
   & \eta\alpha_{8/5} & \nu^2[11] & \\
\ar@{.}[rrrr] &&&&
\\
17 & \eta\eta_4 \ar '[]+R [ruu]+L &   &   \\
   &  \nu\kappa & \sigma\eta^2[9] & 1[13,6] \\
   &\alpha_9 & \sigma\eta[10] &   \\
   & \eta^2\alpha_{8/5} & \sigma[11] &   \\
\ar@{.}[rrrr] &&&&
\\
18 & 4\nu^* \ar '[]+R [ruuuu]+L & &  \\
   & 2\nu^* \ar '[]+R [ruuuu]+L & \alpha_6[8]  &  \\
   & \nu^* \ar '[]+R [ruuuu]+L & \eta\alpha_5[9] & \\
   & \eta\alpha_9 & \alpha_5[10] & \eta[13,6] \\
\ar@{.}[rrrr] &&&&
\\
19 & \br{\sigma}  & \theta_3[6] & \eta^2[13,6] \\
   & \alpha_{10} &  & \eta[14,6] \\
   & \alpha_{10/2} &  & 1[15,6] \\
 & \alpha_{10/3} & \kappa[6] &  \\
\ar@{.}[rrrr] &&&&
\\
   & \br{\kappa} \ar '[]+R [ruu]+L &    &
}$$

\appendix
\chapter{Transfinite spectral sequences associated to towers}\label{sec:Hu}

In this appendix we review Hu's notion of a transfinite spectral sequence \cite{Hu}, specialized to the situation of a space or spectrum which is the inverse limit of a transfinite tower.  We restrict our attention to \emph{degreewise finite} towers, as these give rise to spectral sequences with particularly good convergence properties.  
In Section~\ref{sec:Groth} we review the Grothendieck group of ordinals, which serves as the indexing set for transfinite spectral sequences.  In Section~\ref{sec:TSS} we associate a transfinite spectral sequence to a transfinite tower of spaces or spectra, following Bousfield and Kan.    In Section~\ref{sec:GBT}
 prove some general ``geometric boundary'' type theorems for such transfinite spectral sequences.

Throughout this appendix we work in either the category of pointed spaces or spectra.

\section{The Grothendieck group of ordinals}\label{sec:Groth}

Fix a countable ordinal $\nu = \omega^\beta$ (where $\beta < \nu$ is an ordinal) closed under addition.  We consider transfinite spectral sequences indexed on the Grothendieck group $\mc{\nu}$ of ordinals less than $\nu$ with respect to maximal addition (see \cite[Sec.~1]{Hu}).  A general ordinal less than $\nu$ takes the form
$$ j_1 \omega^\alpha_1+\cdots + j_k\omega^{\alpha_k}, \quad j_s \in \mb{N}, \: \alpha_1 < \cdots < \alpha_k < \beta. $$
A general element of the Grothendieck group $\mc{G}(\nu)$ takes the form 
$$ j_1 \omega^\alpha_1+\cdots + j_k\omega^{\alpha_k}, \quad j_s \in \ZZ, \: \alpha_1 < \cdots < \alpha_k < \beta. $$
Addition in $\mc{G}(\nu)$ is defined by
$$ \left( j_1 \omega^\alpha_1+\cdots + j_k\omega^{\alpha_k}\right)+
\left( j'_1 \omega^\alpha_1+\cdots + j'_k\omega^{\alpha_k}\right)
=  (j_1+j'_1) \omega^\alpha_1+\cdots +( j_k+j'_k)\omega^{\alpha_k}. $$
The elements of $\mc{G}$ are ordered via right lexicographical ordering.

\section{Towers}\label{sec:towers}

A \emph{$\mc{G}(\nu)$-indexed tower under $X$} consists of a set  
$$ \{ X \rightarrow X_\mu \}_{\mu \in \mc{G}} $$  
of objects under $X$, and compatible maps 
$$
\xymatrix{
& X \ar[dl] \ar[dr] \\
X_{\mu} \ar[rr] && X_{\mu'} 
}
$$
for $\mu < \mu'$.  
We shall say that the tower $\{ X_{\mu} \}$ is \emph{degreewise finite} if the following conditions hold.
\begin{enumerate}
\item For each $t$, and each $(\alpha,\beta)$-gap in $\mc{G}(\nu)$, with left subset $\mc{L}$ and right subset $\mc{R}$ (see \cite[Def.~1]{Hu}), there exist $\gamma \in \mc{L}$ and $\gamma' \in \mc{R}$ such that for all $\gamma \le \delta \le \gamma'$ the maps
$$ \pi_t(X_{\gamma}) \rightarrow \pi_t(X_\delta) \rightarrow \pi_t(X_{\gamma'}) $$
are isomorphisms.
\item For a fixed $t$, the maps
$$ \pi_t(X_{\mu}) \rightarrow \pi_t(X_{\mu+1}) $$
are isomorphisms for all but finitely many $\mu \in \mc{G}(\nu)$.
\item For a fixed $t$, there exist $\beta, \beta' \in \mc{G}(\nu)$ such that the map
$$ \pi_t(X) \rightarrow \pi_t(X_\mu) $$ is an isomorphism for $\mu \le \beta$, and $$\pi_t(X_\mu) = 0 $$ for $\mu \ge \beta'$.
\end{enumerate}

\section{The transfinite homotopy spectral sequence of a tower}\label{sec:TSS}

Given a degreewise finite $\mc{G}(\nu)$-indexed tower under $X$, 
define fibers
$$ A_\mu \rightarrow X_\mu \rightarrow X_{\mu+1}. $$
Applying $\pi_*$ gives rise to a $\mc{G}(\nu)$-indexed exact couple, and we get a $\mc{G}(\nu)$-indexed spectral sequence
\begin{equation}\label{eq:towerSS}
 E^1_{t, \mu} = \pi_t(A_\mu) \Rightarrow \pi_t(X).
\end{equation}

\begin{rmk}
In the case where we are working in the category $\Top_*$, some care must be taken in the handling of $\pi_1$ and $\pi_0$ (analogous to the treatment of \cite[IX.4]{BousfieldKan}).  In the cases we are considering in this book, each of the spaces $X_\mu$ and $A_\mu$ are actually connected with abelian fundamental group.
\end{rmk}

\begin{rmk}
Our degreewise finite assumption on the tower $\{ X_\mu\}$ under $X$ implies that (in the language of \cite{Hu}) the transfinite spectral sequence (\ref{eq:towerSS}) is locally of finite type, eventually constant, and converges strongly.  
\end{rmk}
 
We have chosen a different indexing scheme from that used in \cite{Hu}.  In our indexing scheme, $d_\alpha$ differentials take the form
$$ d_\alpha: E^\alpha_{t, \mu} \rightarrow E^\alpha_{t-1, \mu-\alpha}. $$
One way of thinking of these differentials is to consider the associated ``increasing filtration'' $\{ X^\mu \}$ on $X$ defined by the fiber sequences
$$ X^\mu \rightarrow X \rightarrow X_{\mu+1}. $$
Using Verdier's axiom (see \cite[Prop.~6.3.6]{Hovey} in the $\Top_*$ case) we have fiber sequences
$$ X^{\mu-1} \rightarrow X^{\mu} \rightarrow A_\mu. $$ 
The differentials $d_\alpha$ are then given by the zigzag
$$ \pi_tA_\mu \xrightarrow{\partial} \pi_{t-1} X_{\mu-1} \leftarrow \pi_{t-1} X_{\mu-\alpha} \rightarrow \pi_{t-1} A_{\mu-\alpha}. $$
Convergence means that with respect to the $\mc{G}(\nu)$-indexed filtration
$$ F_\mu \pi_t X := \ker \left( \pi_t X \rightarrow \pi_t X_{\mu+1} \right) $$
we have
$$ E^\nu_{t,\mu} \cong \frac{F_\mu \pi_t X}{F_{\mu-1} \pi_t X}. $$

In Section~\ref{sec:GBT}, we shall make use of the following notation.  For $\mu \le \mu'$, define $X^{\mu'}_{\mu}$ to be the fiber
$$ X^{\mu'}_{\mu} \rightarrow X_{\mu} \rightarrow X_{\mu'+1}. $$
Then we have
$$ A_\mu = X^{\mu}_\mu. $$
Our degreewise finite assumption implies that we have
\begin{align*}
X_\mu^\nu & := \varinjlim_{\mu' \ge \mu} X_{\mu}^{\mu'} \simeq X_\mu, \\
X_{-\nu}^{\mu} & := \varprojlim_{\mu' \ge \mu} X_{\mu}^{\mu'} \simeq X^{\mu'}, \\
X^\nu_{-\nu} & := \varprojlim_{\mu \to -\nu} \varinjlim_{\mu' \ge \mu} X_{\mu}^{\mu'} \simeq \varinjlim_{\mu' \to \nu} \varprojlim_{\mu' \ge \mu} X_{\mu}^{\mu'} \simeq X.
\end{align*}
With this notation, there are fiber sequences
\begin{gather*}
X_{\mu}^{\mu'} \rightarrow X_{\mu}^{\mu''} \rightarrow X^{\mu''}_{\mu'+1}, \qquad -\nu \le \mu \le \mu' < \mu'' \le \nu. \\
\end{gather*}
An alternative way to view a $d_\alpha$ differential is to lift $x \in \pi_t X^\mu_{\mu} = E^1_{t,\mu}$ to $\td{x} \in \pi_t X^\mu_{\mu-\alpha+1}$.  Then $d_\alpha(x)$ is detected by the image of $\td{x}$ under the composite
$$ \pi_t X^\mu_{\mu-\alpha+1} \xrightarrow{\partial} \pi_{t-1} X^{\mu-\alpha} \rightarrow \pi_{t-1} X^{\mu-\alpha}_{\mu-\alpha} = E^1_{t-1, \mu-\alpha}. $$

\section{Geometric boundary theorem}\label{sec:GBT}

Suppose that 
$$ \Omega Z \xrightarrow{j} X \xrightarrow{i} Y \xrightarrow{p} Z $$
is a fiber sequence, and that $\{X_\mu\}$, $\{Y_\mu\}$ and $\{ Z_\mu \}$ are compatible locally finite $\mc{G}(\nu)$-indexed towers under $X$, $Y$, and $Z$, such that each of the sequences
$$ X_\mu \rightarrow Y_\mu \rightarrow Z_\mu $$
is also a fiber sequence.  Let 
$$ E^{\alpha}_{t+1,\mu}(Z) \xrightarrow{j_*} E^\alpha_{t,\mu}(X) \xrightarrow{i_*} E^\alpha_{t,\mu}(Y) \xrightarrow{p_*} E^\alpha_{t,\mu}(Z) $$ denote the corresponding maps of spectral sequences.  Let $d^X_\alpha$, $d^Y_\alpha$, and $d^Z_\alpha$ denote the differentials in each of these spectral sequences.

The following lemma is a kind of technical generalization of the ``geometric boundary theorem.''  The classical geometric boundary theorem \cite[Thm.~2.3.4]{Ravenel} is closely related to Case (5) when $\alpha = 1$ in this lemma.

\begin{lem}\label{lem:GBT}
Suppose that $y \in E^\alpha_{t,\mu}(Y)$ and that there is a non-trivial differential
$$ d_\alpha (y) = y' \ne 0 \in E^\alpha_{t-1, \mu-\alpha}(Y). $$
Then one of the following five possibilities occurs.

\begin{enumerate}
\item  There is a non-trivial differential
$$ d^Z_\alpha(p_* y) = p_* y' $$

\item There exists $\alpha' \ge 1$ and classes 
\begin{align*}
x & \in E^1_{t-1, \mu-\alpha}(X), \\
0 \ne z & \in E^{1}_{t-1, \mu-\alpha-\alpha'}(Z), \quad \text{(non-zero in $E^{\alpha'+1}$)} \\
0 \ne x' & \in E^1_{t-2, \mu-\alpha-\alpha'}(X),
\end{align*}
so that
\begin{align*}
& i_*(x) = y', \\
& d^X_{\alpha'}(x) = x', \\
& j_* z = x', \\
& d^Z_{\alpha+\alpha'}(p_* y) = z.
\end{align*}

\item There exists $\alpha', \alpha'' \ge 1$ and classes
\begin{align*}
x  & \in E^1_{t-1, \mu-\alpha}(X), \\
0 \ne x'' & \in E^{1}_{t-2,\mu - \alpha-\alpha'-\alpha''}(X), \quad \text{(non-zero in $E^{\alpha''+1}$)} \\
0 \ne z & \in E^{1}_{t-1, \mu-\alpha-\alpha'}(Z), \quad \text{(non-zero in $E^{\alpha'+1}$)} \\
y'' & \in E^1_{t-1, \mu-\alpha-\alpha'}(Y)
\end{align*}
so that
\begin{align*}
& i_*(x) = y', \\
& d^X_{\alpha'+\alpha''}(x) = x'', \\
& d^Y_{\alpha''}(y'') = -i_*(x''), \\
& p_*(y'') = z \\
& d^Z_{\alpha+\alpha'}(p_* y) = z.
\end{align*}

\item There exists $\alpha' \ge 1$ and classes
\begin{align*}
x & \in E^1_{t-1, \mu-\alpha}(X), \\
0 \ne y'' & \in E^{\alpha'+1}_{t-1, \mu-\alpha-\alpha'}(Y), \\
0 \ne z & \in E^{\alpha'+1}_{t-1, \mu-\alpha-\alpha'}(Z), \\
\bar{x} & \in \pi_{t-1}(X), 
\end{align*}
so that
\begin{align*}
& i_*(x) = y', \\
& \bar{x} \: \text{is} \: \begin{cases}
\text{detected by} \: x, & \text{$x$ is not the target of a differential in $E^*_{*,*}(X)$}, \\
\text{in $F_{\mu-\alpha-1}\pi_{t-1}(X)$}, & \text{$x$ is the target of a differential in $E^*_{*,*}(X)$},
\end{cases} \\
& i(\bar{x})  \: \text{is} \: \begin{cases}
\text{detected by} \: -y'', & \text{$y''$ is not the target of a differential in $E^*_{*,*}(Y)$}, \\
\text{in $F_{\mu-\alpha-\alpha'-1}\pi_{t-1}(Y)$}, & \text{$y''$ is the target of a differential in $E^*_{*,*}(Y)$},
\end{cases} \\
& p_*(y'') = z, \\
& d^Z_{\alpha+\alpha'}(p_* y) = z.
\end{align*}

\item There exist classes 
\begin{align*}
\bar{z} & \in \pi_{t}(Z), \\
0 \ne x & \in E^\alpha_{t-1,\mu-\alpha}(X)
\end{align*}
so that
\begin{align*}
& \bar{z} \: \text{is} \: \begin{cases}
\text{detected by} \: p_*(y), & \text{$p_*(y)$ is not the target of a differential in $E^*_{*,*}(Z)$}, \\
\text{in $F_{\mu-1}\pi_{t}(Z)$}, & \text{$p_*(y)$ is the target of a differential in $E^*_{*,*}(Z)$},
\end{cases} \\
& i_*(x) = y', \\
& j(\bar{z}) \: \text{is} \: \begin{cases}
\text{detected by} \: x, & \text{$x$ is not the target of a differential in $E^*_{*,*}(X)$}, \\
\text{in $F_{\mu-\alpha-1}\pi_{t-1}(X)$}, & \text{$x$ is the target of a differential in $E^{\gamma}_{*,*}(X)$, $\gamma > \alpha$.}
\end{cases}
\end{align*}

\end{enumerate}
\end{lem}

\begin{proof}
We will need to introduce some more notation.  For $\alpha \in \mc{G}(\nu)$, define fibers $Y^{\mu, \mu-\alpha}$ by
$$ Y^{\mu, \mu-\alpha} \rightarrow Y^\mu \rightarrow Z^\mu_{\mu-\alpha+1}. $$
Verdier's axiom gives a diagram of fiber sequences
$$
\xymatrix@C-2em@R-1em{
\\
X^\mu \ar@/^1.7pc/[rr] \ar[dr] &&
Y^\mu \ar@/^1.7pc/[rr] \ar[dr] &&
Z^{\mu}_{\mu-\alpha+1}
\\
& Y^{\mu, \mu-\alpha} \ar[ur] \ar[dr] &&
Z^\mu \ar[ur] 
\\
&& Z^{\mu-\alpha} \ar[ur]
}
$$
We then define $\td{Y}^\mu_{\mu-\alpha, \mu-\alpha -\alpha'}$ , for $\alpha, \alpha' \ge 0$, by fibre sequences
$$ \td{Y}^\mu_{\mu-\alpha, \mu-\alpha - \alpha'} \rightarrow Y^{\mu-\alpha-1, \mu-\alpha - \alpha' - 1} \rightarrow Y^\mu. $$
Although, in the $\Top_*$ case, $\td{Y}^\mu_{\mu-\alpha, \mu-\alpha-\alpha'}$ is not necessarily a loop space, we \emph{define}
$$ \pi_t Y^\mu_{\mu-\alpha, \mu-\alpha - \alpha'} := \pi_{t-1}(\td{Y}^\mu_{\mu-\alpha, \mu-\alpha-\alpha'}). $$
so that 
$$ \td{Y}^\mu_{\mu-\alpha, \mu-\alpha-\alpha'} =  ``{\Omega Y^\mu_{\mu-\alpha, \mu-\alpha-\alpha'}}".$$

Having established the notation above, we are now ready to begin proving the lemma.
Since $p$ induces a map of spectral sequences, we always have
$$ d_\alpha(p_* y) = p_*(y'). $$
If $p_*(y') \ne 0$ in the $E^\alpha$-page, then we are in Case (1).

Otherwise $p_*(y') = 0$ in $E^\alpha_{t, \mu-\alpha}(Z)$.  This, together with the fact that $y$ persists to $E^\alpha_{t,\mu}(Y)$, implies that there exists a lift
$$ \td{y} \in \pi_tY^\mu_{\mu-\alpha+1, \mu-\alpha} \rightarrow \pi_t Y^\mu_\mu \ni y. $$
Let $\alpha' \le \nu$ be maximal such that there exists a lift 
$$ \tdd{y} \in \pi_tY^\mu_{\mu-\alpha+1, \mu-\alpha - \alpha'+1} \rightarrow \pi_t Y^\mu_\mu \ni y. $$
By the above, $\alpha' \ge 1$. Suppose that $\alpha' < \nu$.  
Fix a lift $\tdd{y}$ as above.  Consider the following commutative diagram, where $\mu' := \mu-\alpha$ and $\mu'' := \mu-\alpha-\alpha'$.
\begin{equation}\label{eq:diag1}
\xymatrix{
&&&&& \\
&& \pi_t Z^\mu_{\mu''+1} \ar[rrr] \ar[dl]_\partial
&&& \pi_tZ^\mu_\mu 
\ar@{~>} `u[lllll] `[lllllddd]_{d^Z_{\alpha+\alpha'}} `[llllddd]  [llllddd]
\\
&\pi_{t-1} Z^{\mu''} \ar[dd] \ar[ddr]_\partial  & 
\pi_t Y^\mu_{\mu'+1, \mu''+1} \ar[rr] \ar[d]^\partial \ar[u] &&
\pi_tY^\mu_{\mu'+1} \ar[d]_\partial \ar[r] &
\pi_t Y^\mu_\mu \ar[u]_{p_*}  \ar@{~>}[ddl]^{d^Y_\alpha}
\\
&& \pi_{t-1} Y^{\mu', \mu''} \ar[r] \ar[ul] &
\pi_{t-1}X^{\mu'}_{\mu''+1} \ar[d] \ar[dl]^\partial &
\pi_{t-1}Y^{\mu'} \ar[d] 
\\
&\pi_{t-1} Z^{\mu''}_{\mu''} \ar[dr]_{j_*} &
\pi_{t-2}X^{\mu''} \ar[d] &
\pi_{t-1} X^{\mu'}_{\mu'} \ar[r]_{i_*} \ar@{~>}[dl]^{d^X_{\alpha'}} &
\pi_{t-1} Y^{\mu'}_{\mu'}
\\
&& \pi_{t-2} X^{\mu''}_{\mu''}
}
\end{equation}
The lift $\tdd{y}$ gives compatible elements in all entries of the diagram.  Let 
\begin{align*}
x & \in \pi_{t-1}X^{\mu'}_{\mu'} \\
x' & \in \pi_{t-2}X^{\mu''}_{\mu''} \\
z & \in \pi_{t-1}Z^{\mu''}_{\mu''}
\end{align*}
The maximality of $\alpha'$ guarantees that $z$ is non-trivial in the $E^{\alpha'+1}$-page.  Assume that $x'$ is non-trivial on the $E^1$-page.
Then we are in Case (2), and the consequence of this is obtained by chasing the perimeter of the diagram.

Suppose now that we are in the situation in the previous paragraph, except that $x' = 0 \in E^1_{t-1, \mu''}$.  Let $\td{x} \in \pi_{t-1}X^{\mu'}_{\mu''+1}$ be the image of $\tdd{y}$.  Let $\alpha'' \le \nu$ be maximal so that $\td{x}$ lifts to an element
$$ \tdd{x} \in \pi_{t-1}X^{\mu'}_{\mu''-\alpha''+1}. $$
Fix such a lift.
Because of our assumption on the triviality of $x'$, we have $\alpha'' \ge 1$.  Suppose that $\alpha'' < \nu$.  Then we are in Case (3), as we will now demonstrate.  Define $\mu''' := \mu - \alpha-\alpha'-\alpha''$.  Since the relative homotopy class
$$ \partial \tdd{y} - \tdd{x} \in \pi_{t-1}(Y^{\mu', \mu''}, X^{\mu'''}) $$
maps to $0 \in \pi_{t-1}(X^{\mu'}_{\mu''+1})$, it lifts to give a relative homotopy class
$$ \tdd{y''} \in \pi_{t-1}(Y^{\mu''}, X^{\mu'''}) = \pi_{t-1}(Y^{\mu''}_{\mu'''+1, -\nu}). $$
Consider the following commutative diagram.
\begin{equation}\label{eq:diag2}
\xymatrix@C-1.83em{
&&&&& \\
&& \pi_t Z^\mu_{\mu''+1} \ar[rrr] \ar[dl]_\partial
&&& \pi_tZ^\mu_\mu 
\ar@{~>} `u[lllll] `[llllld]_{d^Z_{\alpha+\alpha'}} [llllld]
\\
\pi_{t-1} Z^{\mu''}_{\mu''} & 
\pi_{t-1} Z^{\mu''} \ar[l]  & 
\pi_t Y^\mu_{\mu'+1, \mu''+1} \ar[rr] \ar[d]^\partial \ar[u] &&
\pi_tY^\mu_{\mu'+1} \ar[d]_\partial \ar[r] &
\pi_t Y^\mu_\mu \ar[u]_{p_*}  \ar@{~>}[ddl]^{d^Y_\alpha}
\\
\pi_{t-1}Y^{\mu''}_{\mu''} \ar[u]^{p_*} 
\ar@{~>} `l[ddd] `[ddd]_{d^{Y}_{\alpha''}} [ddd] && 
\pi_{t-1} Y^{\mu', \mu''} \ar[r] \ar[ul] &
\pi_{t-1}X^{\mu'}_{\mu''+1} \ar[d] &
\pi_{t-1}Y^{\mu'} \ar[d] 
\\
\pi_{t-1}Y^{\mu''}_{\mu'''+1} \ar[u] \ar[d]_\partial &
\pi_{t-1}Y^{\mu''}_{\mu'''+1, -\nu} \ar[l] \ar[uu] \ar[dr]_\partial  &
\pi_{t-1}X^{\mu'}_{\mu'''+1} \ar[d]^\partial \ar[ur] &
\pi_{t-1} X^{\mu'}_{\mu'} \ar[r]_{i_*} \ar@{~>}[ddl]^{d^X_{\alpha'+\alpha''}} &
\pi_{t-1} Y^{\mu'}_{\mu'}
\\
\pi_{t-2}Y^{\mu'''} \ar[d] &&
\pi_{t-2}X^{\mu'''} \ar[d] \ar[ll]
\\
\pi_{t-2}Y^{\mu'''}_{\mu'''} &&
\pi_{t-2} X^{\mu'''}_{\mu'''}  \ar[ll]^{i_*}
}
\end{equation}
The elements $\tdd{y}$, $\tdd{x}$, and $\tdd{y''}$ induce a compatible collection of elements in each of the entries of the above diagram, except that the images of $\tdd{x}$ and $\tdd{y''}$ in the entries 
$$ \pi_{t-2}X^{\mu'''}, \quad \pi_{t-2}X^{\mu'''}_{\mu'''}, \quad \pi_{t-2}Y^{\mu'''}, \quad \text{and} \quad 
\pi_{t-2}Y^{\mu'''}_{\mu'''} $$ 
differ by a sign.
We get images
\begin{align*}
x & \in \pi_{t-1}X^{\mu'}_{\mu'}, \\
x'' & \in \pi_{t-2}X^{\mu'''}_{\mu'''}, \quad \text{(Image of $\tdd{x}$)} \\
z & \in \pi_{t-1}Z^{\mu''}_{\mu''}, \\
y'' & \in \pi_{t-1}Y^{\mu''}_{\mu''}, \quad \text{(Image of $\tdd{y''}$)}. 
\end{align*}
Note that the maximality of $\alpha'$ implies that $z$ is non-zero in the $E^{\alpha'+1}$-page, and the maximality of $\alpha''$ implies that $x''$ is non-zero in the $E^{\alpha''+1}$-page.
The conclusion of Case (3) is deduced from chasing around the perimeter of the diagram.  

Assume that we are in the situation of Case (3), except that $\alpha'' = \nu$.  Then we are in Case (4).  Since $\alpha'' = \nu$, the element $\td{x}$ lifts to an element
$$ \tdd{x} \in \pi_{t-1}X^{\mu'}. $$
Fix such a lift.
Since the homotopy class
$$ \partial \tdd{y} - \tdd{x} \in \pi_{t-1}(Y^{\mu', \mu''}) $$
maps to $0 \in \pi_{t-1}(X^{\mu'}_{\mu''+1})$, it lifts to give a homotopy class
$$ \tdd{y''} \in \pi_{t-1}(Y^{\mu''}). $$
Consider the following commutative diagram.
\begin{equation}\label{eq:diag3}
\xymatrix@C-1.83em{
&&&&& \\
&& \pi_t Z^\mu_{\mu''+1} \ar[rrr] \ar[dl]_\partial
&&& \pi_tZ^\mu_\mu 
\ar@{~>} `u[lllll] `[llllld]_{d^Z_{\alpha+\alpha'}} [llllld]
\\
\pi_{t-1} Z^{\mu''}_{\mu''} & 
\pi_{t-1} Z^{\mu''} \ar[l]  & 
\pi_t Y^\mu_{\mu'+1, \mu''+1} \ar[rr] \ar[d]^\partial \ar[u] &&
\pi_tY^\mu_{\mu'+1} \ar[d]_\partial \ar[r] &
\pi_t Y^\mu_\mu \ar[u]_{p_*}  \ar@{~>}[ddl]^{d^Y_\alpha}
\\
\pi_{t-1}Y^{\mu''}_{\mu''} \ar[u]^{p_*} \ar@{=>}[ddr] &
\pi_{t-1}Y^{\mu''} \ar[l] \ar[d] \ar[u] & 
\pi_{t-1} Y^{\mu', \mu''} \ar[r] \ar[ul] &
\pi_{t-1}X^{\mu'}_{\mu''+1} \ar[d] &
\pi_{t-1}Y^{\mu'} \ar[d] 
\\
& \pi_{t-1}Y^{\mu'} \ar[d]  &
\pi_{t-1}X^{\mu'} \ar[d] \ar[ur] \ar[l] &
\pi_{t-1} X^{\mu'}_{\mu'} \ar[r]_{i_*} \ar@{=>}[dl] &
\pi_{t-1} Y^{\mu'}_{\mu'}
\\
& \pi_{t-1} Y &
\pi_{t-1} X \ar[l]^i
}
\end{equation}
The elements $\tdd{y}$, $\tdd{x}$, and $\tdd{y''}$ induce a compatible collection of elements in each of the entries of the above diagram, except that the images of $\tdd{x}$ and $\tdd{y''}$ in the entries 
$$  \pi_{t-1} Y^{\mu'} \quad \text{and} \quad 
\pi_{t-1}Y $$ 
differ by a sign.
We get images
\begin{align*}
x & \in \pi_{t-1}X^{\mu'}_{\mu'}, \\
\bar{x} & \in \pi_{t-1}X, \quad \text{(Image of $\tdd{x}$)} \\
z & \in \pi_{t-1}Z^{\mu''}_{\mu''}, \\
y'' & \in \pi_{t-1}Y^{\mu''}_{\mu''}.  
\end{align*}
Note that the maximality of $\alpha'$ implies that $z$ is non-zero in the $E^{\alpha'+1}$-page.  As $p_*(y'') = z$, the element $y''$ must also be non-zero in the $E^{\alpha'+1}$-page. 
The conclusion of Case (4) is deduced from chasing around the perimeter of the diagram.

Assume that we are in the situation of Case (2), except that $\alpha' = \nu$.  Then we are in Case (5).  Since $\alpha' = \nu$, there exists a lift 
$$ \tdd{y} \in \pi_tY^\mu_{\mu-\alpha, -\nu} \rightarrow \pi_t Y^\mu_\mu \ni y. $$
Consider the following diagram.
$$
\xymatrix@C-1em{
\\
\pi_tZ \ar[d]_j &&
\pi_tZ^\mu \ar[ll] \ar[rr] \ar[dl]_\partial &&
\pi_tZ^\mu_\mu \ar@{=>} @/_2pc/ [llll]
\\
\pi_{t-1}X &
\pi_{t-1}X^\mu \ar[l] &
\pi_{t} Y^\mu_{\mu'+1, -\nu} \ar[u] \ar[r] \ar[d]^\partial &
\pi_t Y^{\mu}_{\mu'+1} \ar[d]_\partial \ar[r] &
\pi_t Y^\mu_\mu \ar@{~>}[ddl]^{d^Y_\alpha} \ar[u]_{p_*}
\\
&& \pi_{t-1} X^{\mu'} \ar[r] \ar[ul] \ar[d] &
\pi_{t-1}Y^{\mu'} \ar[d] 
\\
&& \pi_{t-1} X^{\mu'}_{\mu'} \ar[r]_{i_*} \ar@{=>}[uull] &
\pi_{t-1} Y^{\mu'}_{\mu'} 
}
$$
The images of $\tdd{y}$ give compatible elements of each of the entries in the diagram above.  In particular, we get
\begin{align*}
x & \in \pi_{t-1}X^{\mu'}_{\mu'}, \\
\bar{z} & \in \pi_t(Z).
\end{align*}
Chasing the perimeter gives the desired conclusions.  Note that since $y'$ was assumed to be non-zero in the $E^\alpha$-page, and $i_*(x) = y'$, the element $x$ must also be nonzero in the $E^\alpha$-page. 
\end{proof}

\begin{rmk}
Lemma~\ref{lem:GBT} and its proof, while technical, have a geometric interpretation that may help the reader.  Suppose that $X$ is a subspace of $Y$, that $Y$ is the union of an increasing sequence of subspaces $Y^\mu$, and let
$$ X^\mu := Y^\mu \cap X $$
be the induced filtration on $X$.  Then we have
\begin{align*}
\pi_t(Z) & = \pi_t(Y,X), \\
\pi_t(Z^\mu) & = \pi_t(Y^\mu, X^\mu).
\end{align*} 
Then an element of $\pi_tY^{\mu}_{\mu'+1, \mu''+1}$ can be represented as a map of a disk
$$ D^{t} \rightarrow Y^\mu $$
whose boundary lies in $X^{\mu'} \cup Y^{\mu''}$, and an element of $\pi_t(Y^{\mu})$ can be represented by a map
$$ S^t \rightarrow X^{\mu'} \cup Y^{\mu''} $$
(see Figure~\ref{fig:GBTY}).  Lemma~\ref{lem:GBT} can be regarded as a kind of moving lemma for a boundary.  Specifically, the differential
$$ d^Y_\alpha y = y' $$
corresponds to a map
$$ y: D^t \rightarrow Y^\mu $$
such that $y(\partial D^t) \subseteq Y^{\mu'}$, giving
$$ y' = y\vert_{\partial D^t}: S^{t-1} \rightarrow Y^{\mu'}. $$
Then Figure~\ref{fig:GBT} displays the geometric situation represented by each of the cases of Lemma~\ref{lem:GBT}, as described below.
\begin{description}
\item[Case (1)] The element $p_*(y') \in \pi_t(Y^\mu, X^\mu)$ is non-trivial.
\item[Case (2)] The element $y$ can be represented by a map of a disk whose boundary maps into $Y^{\mu',\mu''}$, with $\mu''$ minimal, and so that the relative class $y' \in \pi_{t-1}(Y^{\mu',\mu''}, X^{\mu'} \cup Y^{\mu''-1})$ is non-zero.
\item[Case (3)] As in Case (2), except that the relative class $y' \in \pi_{t-1}(Y^{\mu',\mu''}, X^{\mu'} \cup Y^{\mu''-1})$ is zero, and there exists a relative homotopy class
$$ x' \in \pi_{t-1}(X^{\mu''},X^{\mu'''})  $$
with essential boundary such that the images of $x'$ and $y'$ in the relative homotopy group
$\pi_{t-1}(X^\mu, X^{\mu''})$ are equal.
\item[Case (4)] As in Case (3), except that the boundary of the relative homotopy class $x'$ is null.
\item[Case (5)] As in Case (2), except that the relative homotopy class $y' \in \pi_{t-1}(Y^{\mu'}, X^{\mu'})$ is zero.
\end{description}
\end{rmk}

\pagebreak
\vfill

\begin{figure}
\begin{center}
\includegraphics[width=0.4\textwidth]{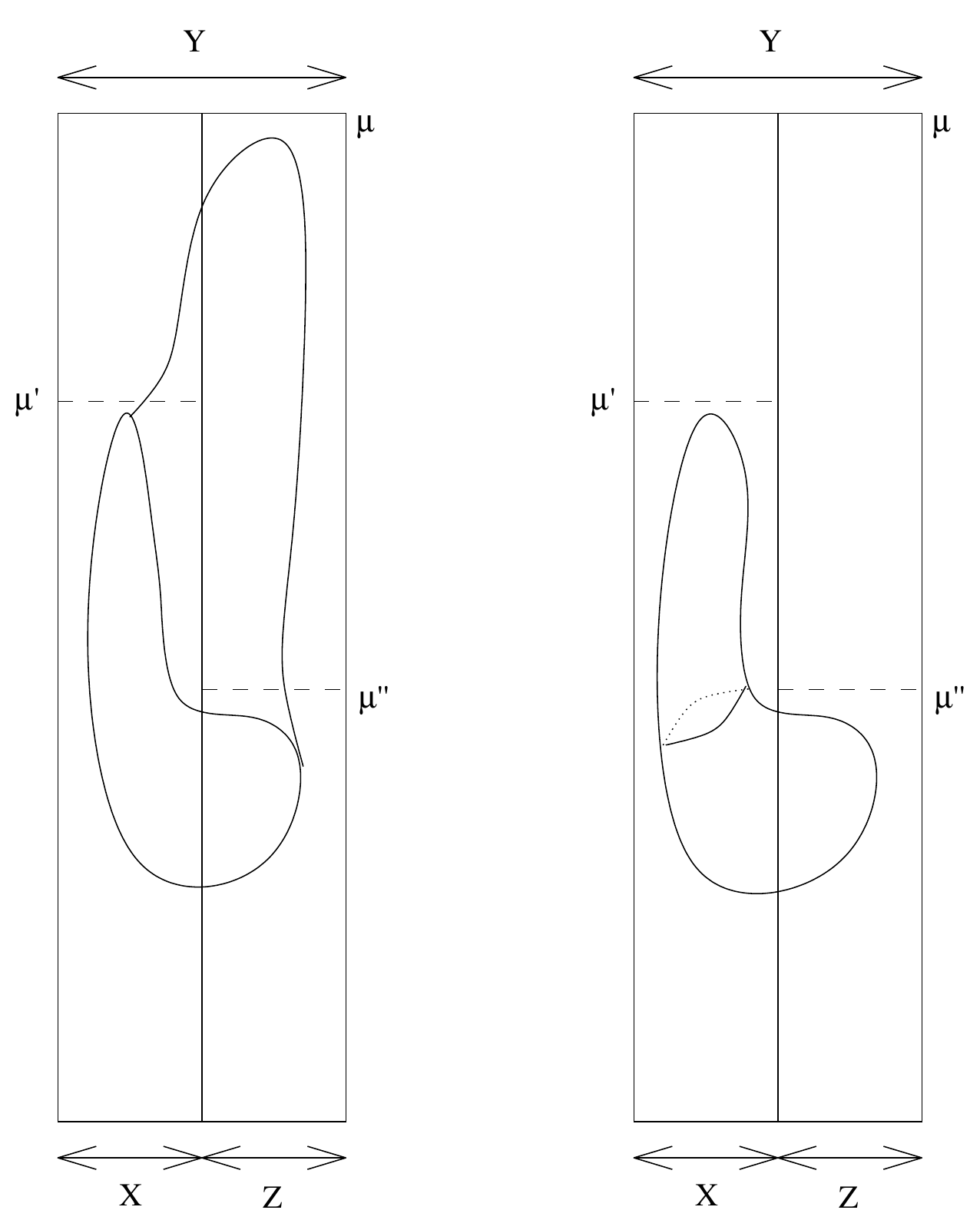}
\end{center}
\caption{Geometric interpretation of an element of $\pi_tY^\mu_{\mu'+1,\mu''+1}$ (left) and an element of $\pi_tY^{\mu',\mu''}$ (right)}
\label{fig:GBTY}
 \end{figure}

\begin{figure}
\begin{center}
\includegraphics[width=\textwidth]{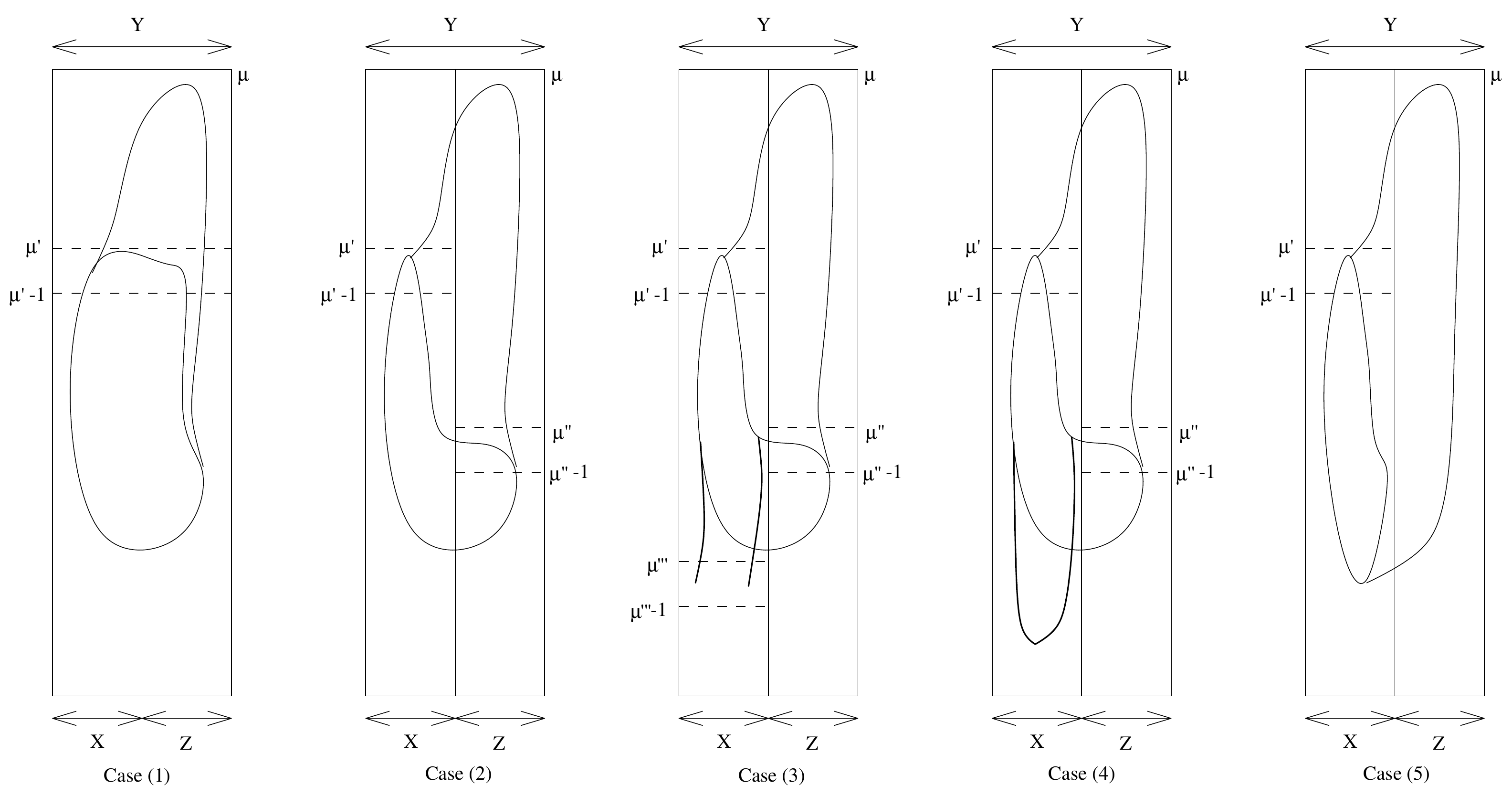}
\end{center}
\caption{Geometric interpretation of the various cases of Lemma~\ref{lem:GBT}}
\label{fig:GBT}
 \end{figure}

$\quad$
\pagebreak
\vfill

Lemma~\ref{lem:GBT}, Case (5), gives a method of computing the map $j$, but it is insufficient if one wants to compute $j(\bar{z})$ for a \emph{particular} choice of $\bar{z}$.  We therefore offer the following variant, whose proof is almost identical to that of Lemma~\ref{lem:GBT}, Case (5).

\begin{lem}\label{lem:GBT2}
Suppose that $\bar{z} \in \pi_t(Z)$ is detected by $z \in E^1_{t,\mu}(Z)$ where $j_*(z) = 0$, and suppose that $x \in E^1_{t-1, \mu-\alpha}$ detects $j(\bar{z})$.  Then one of the following two alternatives holds.
\begin{enumerate}
\item There is a $y \in E^1_{t, \mu}(Y)$ and a non-trivial differential
$$ d^Y_\alpha(y) = i_*(x). $$
 \item Either $i_*(x) = 0$, or there is a non-trivial differential
$$ d^Y_{\alpha'}(y') = i_*(x) $$
for $\alpha' < \alpha$.
\end{enumerate}
\end{lem}

\backmatter

\bibliographystyle{amsalpha}
\nocite{*}
\bibliography{GoodEHPmem}

\end{document}